\documentclass[12pt,twoside]{article}
\usepackage[hmargin=0.8in,vmargin=0.9in]{geometry}
\usepackage{fancyhdr}
\usepackage{graphicx}
\usepackage{amssymb}
\usepackage{amsmath}
\usepackage{amsthm}
\usepackage{amsfonts}
\usepackage{mathrsfs}
\usepackage{mathtools}
\usepackage{bm}
\usepackage{color}
\usepackage{setspace}
\usepackage{exscale}
\usepackage{relsize}
\usepackage{float}
\usepackage{cite}
\usepackage{hyperref}
\usepackage{nicefrac}
\usepackage{placeins}
\usepackage{soul}

\usepackage{booktabs,multirow} 
\usepackage{array} 
\usepackage{paralist} 

\usepackage{sectsty}

\sectionfont{\fontsize{14}{15}\selectfont}

\theoremstyle{plain}                    
\newtheorem{thm}{Theorem}[section]

\newtheorem{rmk}[thm]{Remark}
\newenvironment{acknowledgment}{{\flushleft \bf Acknowledgment:}}{}

\usepackage{tikz}
\usetikzlibrary{positioning}
\usepackage{xcolor}

\allowdisplaybreaks[1]

\numberwithin{equation}{section}
\numberwithin{figure}{section}
\numberwithin{table}{section}

\newcommand\eref[1]{(\ref{#1})}

\newcommand*\xbar[1]{%
  \hbox{%
    \vbox{%
      \hrule height 0.5pt 
      \kern0.4ex
      \hbox{%
        \kern-0.05em
        \ensuremath{#1}%
        \kern-0.00em
      }%
    }%
  }%
}

\usepackage[utf8]{inputenc}
\usepackage[english]{babel}
\usepackage{empheq}
\usepackage{subcaption} 
\usepackage{epsfig}
\usepackage{algorithm}
\usepackage{algpseudocode}
\usepackage{bbm}
\usepackage{sectsty}
\usepackage[normalem]{ulem}
\newcommand{\bb}[1]{{\bf{#1}}} 
\newcommand{\mc}[1]{{\mathcal{#1}}} 

\newcommand{\p}[1]{{\left( #1 \right)}}

\def\d{\partial}
\def\hf {\frac{1}{2}}

\newcommand{\kph}{{k+\frac{1}{2}}}
\newcommand{\kmh}{{k-\frac{1}{2}}}
\newcommand{\jph}{{j+\frac{1}{2}}}
\newcommand{\jmh}{{j-\frac{1}{2}}}
\newcommand{\dx}{\Delta x}
\newcommand{\dy}{\Delta y}

\graphicspath{{Figures/}}

\title{Divergence-Free Flux Globalization Based Well-Balanced Path-Conservative Central-Upwind Schemes for Rotating Shallow Water
Magnetohydrodynamics}
\author{Alina Chertock\thanks{Department of Mathematics and Center for Research in Scientific Computing, North Carolina State University,
Raleigh, NC 27695, USA; {\tt chertock@math.ncsu.edu}}, Alexander Kurganov\thanks{Department of Mathematics, Shenzhen International Center
for Mathematics, and Guangdong Provincial Key Laboratory of Computational Science and Material Design, Southern University of Science and
Technology, Shenzhen, 518055, China; {\tt alexander@sustech.edu.cn}}, Michael Redle\thanks{Applied and Computational Mathematics, RWTH
Aachen University, 52062, Aachen, Germany, and Department of Mathematics, North Carolina State University, Raleigh, NC 27695, USA;
{\tt redle@acom.rwth-aachen.de}}, and Vladimir Zeitlin\thanks{Laboratoire de M\'et\'eorologie Dynamique, Sorbonne Universit\'e (SU), Ecole
Normale Sup\'erieure (ENS), CNRS, Paris, 75231, France, and Shenzhen International Center for Mathematics, Southern University of Science
and Technology, Shenzhen, 518055, China; {\tt zeitlin@lmd.ens.fr}}}

\begin{document}
\date{}
\maketitle

\begin{abstract}
We develop a new second-order flux globalization based path-conservative central-upwind (PCCU) scheme for rotating shallow water
magnetohydrodynamic equations. The new scheme is designed not only to maintain the divergence-free constraint of the magnetic field at the
discrete level but also to satisfy the well-balanced (WB) property by exactly preserving some physically relevant steady states of the
underlying system. The locally divergence-free constraint of the magnetic field is enforced by following the method recently introduced in
[A. Chertock, A. Kurganov, M. Redle, and K. Wu, ArXiv preprint (2022), arXiv:2212.02682]: we consider a Godunov-Powell modified version of
the studied system, introduce additional equations by spatially differentiating the magnetic field equations, and modify the reconstruction
procedures for magnetic field variables. The WB property is ensured by implementing a flux globalization approach within the PCCU scheme, leading to a method capable of preserving both still- and moving-water equilibria exactly. In addition to provably achieving both the WB and
divergence-free properties, the new method is implemented on an unstaggered grid and does not require any (approximate) Riemann problem
solvers. The performance of the proposed method is demonstrated in several numerical experiments that confirm the lack of spurious
oscillations, robustness, and high resolution of the obtained results.
\end{abstract}



\section{Introduction}\label{sec1}
Rotating shallow water magnetohydrodynamics (MHD) equations, also known as magnetic rotating shallow water (MRSW) equations, were introduced
in the pioneering paper \cite{Gilman2000Magnetohydrodynamic} as a model of the solar tachocline. Written in conservative form, the MRSW
equations on the tangent plane to a rotating star/planet read as
\begin{equation}
\begin{aligned}
&h_t+\nabla\!\cdot\!(h\bm u)=0,\\
&(h\bm u)_t+\nabla\!\cdot\!(h\bm u\otimes\bm u-h\bm b\otimes\bm b)+\nabla\left(\frac{g}{2}h^2\right)=-gh\nabla Z-fh\bm u^\perp,\\
&(h\bm b)_t+\nabla\!\cdot\!(h\bm b\otimes\bm u-h\bm u\otimes\bm b)=0,\\
&\nabla\!\cdot\!(h\bm b)=0,
\end{aligned}
\label{1.1}
\end{equation}
where $x$ and $y$ are the spatial variables in the plane, $t$ denotes time, $h$ represents the fluid thickness, $\bm u=(u,v)^\top $ is the
horizontal velocity, $\bm b=(a,b)^\top$ denotes the horizontal magnetic field in units of velocity, $\otimes$ denotes the tensor product,
$g$ is the constant gravitational acceleration, $Z$ is the time-independent bottom topography, $\bm u^\perp=(-v,u)^\top$ denotes the vector
obtained by rotating velocity $\bm u$ by the angle $\pi/2$, and $f=f(y)$ is the Coriolis parameter. In what follows, the $x$- and
$y$-components of vector fields will be often called zonal and meridional, respectively, according to the standard astro- and geophysical
terminology. If the effects of the curvature are neglected, we get the simplest $f$-plane approximation, where $f(y)\equiv f_c$ is constant.
If the curvature is taken into account to the first order, we obtain the beta-plane approximation, with $f(y)=f_c+\beta y$,
$\beta={\rm Const}$. The model \eref{1.1} represents the conservation of mass, momentum, and divergence-free magnetic flux in the presence
of rotation and bottom topography, the latter two acting as specific sources. Topography was absent in the original formulation
\cite{Gilman2000Magnetohydrodynamic} but can be important, for instance, in geophysical applications. Note that in the absence of a
magnetic field, the system becomes the standard rotating shallow water (RSW) model, abundantly studied in the physical and mathematical
literature.

The MRSW model can be systematically derived from the full MHD equations for a magnetic rotating fluid in the Boussinesq and hydrostatic
approximations, the latter being valid for large-scale motions, by vertical averaging \cite{Zeitlin2013Remarks}. A descendant
MQG model for motions close to the magneto-geostrophic equilibrium, that is, an equilibrium between the pressure
(magnetic plus hydrodynamic) and the Coriolis forces, follows from the MRSW equations by filtering the fast waves
\cite{Zeitlin2013Remarks,zeitlin2015geostrophic}. At present, the MRSW and MQG models, as well as their variants, are used both in
astrophysical (see, e.g., \cite{Petrosyan1}) and geophysical (see, e.g., \cite{Raphaldini1}) applications. The MRSW model allows us to
describe the essential dynamical entities of the full MHD, such as (magnetized) vortices, (magneto-)inertia-gravity, Alfv\'en waves, and
their interactions, and also turbulent regimes \cite{Petrosyan1,TobiasDiamondHughes}. It is well known (see, e.g., \cite{LandauLifshitz})
that the MHD equations admit shocks of various geometries and contact discontinuities, and this property is inherited by the MRSW model.
Besides that, the MRSW equations admit Rossby waves, which arise in the configurations close to magneto-geostrophic equilibria in the
presence of differential rotation and are of particular interest for applications; see, e.g., the review papers
\cite{Petrosyan2,Zaqarashvili1}. In addition, exact steady-moving balanced vortex dipole solutions of the MQG equations, with a magnetic
anomaly, either trapped inside or expelled from their cores, were recently found \cite{Lahaye2022Coherent}, and a question of their
counterparts in the full MHD equations arise. 

The present paper aims to develop a numerical scheme for \eref{1.1}, which would provide a reliable tool for investigating various
aforementioned nonlinear dynamical processes and beyond. However, the construction of such a scheme has to rise to two main challenges. The
first challenge is to ensure a well-balanced (WB) property, that is, to develop a scheme capable of exactly preserving (several) physically
relevant steady states that correspond to an exact balance of flux-divergence and source terms in \eref{1.1}. Formulating a method that
ensures the WB property is, however, nontrivial. For example, a straightforward, shock-capturing discretization can often result in spurious
oscillations or spurious numerical waves that could be orders of magnitude larger than the small perturbation (of the steady state) to be
captured. While using a very fine mesh may be able to fix these issues, such an approach would drastically increase computational time and
may thus be impractical. For the standard RSW system, that is, \eref{1.1} without magnetic field, several WB schemes were developed; see,
e.g., \cite{ADDGNP,BouchutZ,Cao2022Flux,CLP,Chertock2018Well,DM22,DongLi,LNK} and references therein.

In this paper, we develop a new WB scheme for the MRSW system \eref{1.1} using a flux globalization approach introduced in
\cite{Chertock2018_2by2,CDH2009,DDMA11,GC2001,MGD11} and recently successfully applied to a variety of systems of balance laws; see, e.g.,
\cite{Cao2022Flux,Cao2023Flux,Cheng2019new,Chertock2018Well,Kurganov2023Well,Kurganov2020Well}. In this approach, both the source and
nonconservative product terms are incorporated into the flux leading to an equivalent quasi-conservative system with a global flux. The
resulting system is then integrated using a Riemann-problem-solver-free central-upwind (CU) scheme. The CU schemes were introduced in
\cite{Kurganov2007Reduction,kurganov2001semidiscrete,Kurganov2017second,Kurganov2000New,Kurganov2002solution} as a robust ``black-box''
solver for general multidimensional systems of conservation laws and then were almost directly applied to the quasi-conservative systems
with global fluxes in \cite{Cao2022Flux,Cao2023Flux,Cheng2019new,Chertock2018Well,Chertock2018_2by2,Kurganov2023Well,Kurganov2020Well}.
The CU schemes were extended to nonconservative hyperbolic systems in \cite{Castro_Diaz2019Path}, where path-conservative CU (PCCU) schemes
were introduced. The path-conservative technique has been recently incorporated into the flux globalization framework, and WB flux
globalization based PCCU schemes have been introduced in \cite{Cao2022Flux,Cao2023Flux,Kurganov2023Well}. These WB schemes are capable of
exactly preserving a wide variety of steady states, including some discontinuous ones.

The second main challenge is the maintenance of the zero-divergence constraint for the magnetic flux, $\nabla\!\cdot\!(h\bm b)=0$, at the
discrete level. It is well-known that the enforcement of this constraint helps to prevent the appearance of nonphysical structures or 
spurious oscillations in the solutions; see, e.g., \cite{Balsara1999Staggered,Brackbill1980effect,Li2005Locally,Toth2000Div}. A large
variety of methods that preserve the divergence-free constraint in the context of the full MHD equations has been proposed; we refer the
reader to, for instance,
\cite{chertock2022new,Dumbser2019divergence,Fu2018Globally,Helzel2013high,Londrillo2004divergence,Mishra2012constraint,Xu2016divergence} and
references therein. For the shallow water MHD, early efforts in discrete divergence preservation were introduced through constrained
transport methods in \cite{Rossmanith2002wave,Sterck2001multi}. The main idea of constrained transport methods was to stagger the magnetic
field in such a way that exactly preserves the divergence-free constraint. The constrained transport methods have since been further
developed to be robust on unstaggered grids (see, e.g., \cite{Touma2010unstaggered}). Still, they are typically based on exact or
approximate Riemann problem solvers. Another commonly used approach, known as the divergence cleaning \cite{Brackbill1980effect}, uses a
Hodge decomposition to project the magnetic field into a divergence-free subspace and then to advect the divergence errors with the flow in
such a way that does not cause accumulation. The divergence cleaning methods formulated for the shallow water MHD include, among others,
Roe-type methods \cite{Kemm2016Roe}, conservation element/solution element (CE/SE) methods \cite{Ahmed2019higher,Qamar2006application}, and
a kinetic flux-vector splitting method \cite{Qamar2010kinetic}. These divergence cleaning methods, however, do not ensure an identically
zero divergence, even though the divergence errors are controlled. A divergence-free finite volume evolution Galerkin method was introduced
in \cite{Kroger2005evolution}, and divergence-free entropy stable methods were proposed in \cite{Duan2021High,Winters2016entropy}. In the
latter works, the divergence constraint is enforced by including and discretizing an additional source term, known as a Janhunen source.

Several numerical methods, which are both WB and divergence-free, are available in the literature. In \cite{Bouchut2017multi}, a WB
divergence-free method for one-dimensional (1-D) shallow-water MHD equations was constructed. In \cite{Zia2014numerical}, a WB method was 
proposed, which employs a divergence cleaning approach to reduce but not completely diminish the discrete divergence. In addition, a
high-order entropy stable finite-difference divergence-free method proposed in \cite{Duan2021High} is WB in the sense that it is capable of
exactly preserving ``lake-at-rest'' steady states in both one and two dimensions.

In this paper, we construct a divergence-free flux globalization based WB PCCU scheme for MRSW system \eref{1.1}. Thanks to a flux
globalization technique from \cite{Kurganov2023Well}, which we modify to treat the MRSW equations, our scheme preserves not only simple
``lake-at-rest'' steady states, but also some of the moving-water equilibria. The divergence-free constraint is enforced using the technique
we have recently introduced in \cite{chertock2022new}. Namely, we use (i) a Godunov-Powell modified version of the MRSW system \eref{1.1};
(ii) additional equations obtained by spatially differentiating the magnetic field equations in \eref{1.1}; and (iii) reconstruction
adjustments for magnetic field variables. The resulting method is successfully tested on a number of both 1-D and two-dimensional (2-D)
numerical examples and produces accurate and non-oscillatory results.

The paper is organized as follows. In \S\ref{sec2}, we present the WB PCCU divergence-free method for the 1-D MRSW system: we first discuss
the MRSW model modifications (\S\ref{sec21}) and then introduce a new 1-D numerical method (\S\ref{sec22}). We organize \S\ref{sec3} the
same as \S\ref{sec2}: we first present the 2-D MRSW model modifications (\S\ref{sec31}) and then introduce a new 2-D numerical method
(\S\ref{sec32}). In \S\ref{sec41} and \S\ref{sec42}, we show the 1-D and 2-D numerical experiments, respectively. We make our concluding
remarks in \S\ref{sec6}.

\section{1-D Well-Balanced PCCU Scheme for MRSW Equations}\label{sec2}
\subsection{Governing Equations}\label{sec21}
In the 1-D case, the MRSW system \eref{1.1} reduces to
\begin{equation}
\begin{aligned}
&h_t+(hv)_y=0,\\
&(hu)_t+(huv-hab)_y=fhv,\\
&(hv)_t+\left(hv^2+\frac{g}{2}h^2-hb^2\right)_y=-fhu-ghZ_y,\\
&(ha)_t+(hav-hbu)_y=0,\\
&(hb)_y=0,
\end{aligned}
\label{2.1}
\end{equation}
Note that this system is sometimes referred to as the 1.5-D MRSW system; see, for example, \cite{zeitlin2015geostrophic}; as translational
symmetry is imposed on the 2-D model in the $x$-direction. In other words, the $x$-dependence is discarded, but we keep the $x$-components
of the vector fields $\bm u$ and $\bm b$ to have a model that includes rotational effects. In this paper, we will refer to the system
\eref{2.1} as the 1-D MRSW model. In addition, notice that \eref{2.1} satisfies the divergence-free constraint $\nabla\!\cdot\!(h\bm b)=0$,
which in the 1-D case simplifies to
\begin{equation}
hb={\rm Const},
\label{2.2}
\end{equation}
as long as this condition is satisfied initially.

Before discretization, we make several adjustments to the system \eref{2.1} that will make it easier to devise a scheme that exactly
preserves both the divergence-free constraint \eref{2.2} and can exactly preserve some of the steady states of \eref{2.1}. 
We start modifications by adjusting the system \eref{2.1} to include the additional Godunov-Powell source terms (see, e.g.,
\cite{de2001hyperbolic,Dellar2002Hamiltonian,Kemm,Kroger2005evolution,Winters2016entropy}):
\begin{equation}
\begin{aligned}
&h_t+(hv)_y=0,\\
&(hu)_t+(huv-hab)_y=fhv-a\left[(hb)_y\right],\\
&(hv)_t+\left(hv^2+\frac{g}{2}h^2-hb^2\right)_y=-fhu-ghZ_y- b\left[(hb)_y\right],\\
&(ha)_t+(hav - hbu)_y=-u\left[(hb)_y\right],\\
&(hb)_t=-v\left[(hb)_y\right],
\end{aligned}
\label{2.3}
\end{equation}
which can be written in the following vector form:
\begin{equation}
\bm U_t+\bm F(\bm U)_y=Q(\bm U)\bm U_y+\bm S(\bm U),
\label{2.5}
\end{equation}
where $\bm U=(h,hu,hv,ha,hb)^\top$ and
\begin{equation}
\bm F(\bm U)=\begin{pmatrix}hv\\huv-hab\\hv^2+\frac{g}{2}h^2-hb^2\\hav-hbu\\0\end{pmatrix},~~
Q(\bm U)=\begin{pmatrix}0&0&0&0&0\\
0&0&0&0&-a\\
0&0&0&0&-b\\
0&0&0&0&-u\\
0&0&0&0&-v
\end{pmatrix},~~
\bm S(\bm U)=\begin{pmatrix}0\\fhv\\-fhu-ghZ_y\\0\\0\end{pmatrix}.
\label{2.6}
\end{equation}
Note that the additional terms $Q(\bm U)\bm U_y$ are theoretically zero due to \eref{2.2}. However, including them helps to enforce the
divergence-free condition \eref{2.2} at the discrete level.

There are several ways to discretize the Godunov-Powell modified MRSW system \eref{2.3} in a divergence-free way (see, e.g.,
\cite{fuchs2011approximate,janhunen2000positive,waagan2011robust} and related works for ideal MHD
\cite{powell1999solution,powell1995upwind,WuShu2018,WuShu2021GQL}). In this paper, we follow \cite{chertock2022new} and achieve this goal
by first introducing the new variable $B:=(hb)_y$ that satisfies the additional evolution equation
\begin{equation}
B_t+(vB)_y=0,
\label{2.4}
\end{equation}
which is obtained by taking the $y$-derivative of the $hb$ induction equation in \eref{2.3}, and then numerically solving the augmented
system \eref{2.5}--\eref{2.4}. While this system is set in a way that makes it relatively easy to design a scheme that preserves the
divergence-free constraint \eref{2.2}, its direct discretization does not necessarily lead to a WB scheme.

One can show that the system \eref{2.5}--\eref{2.4} possesses steady-state solutions that satisfy
\begin{equation}
hv\equiv{\rm Const},~~E:=\frac{v^2}{2}+g(h+Z)-\frac{b^2}{2}+\int\limits_{\widehat y}^yf(\eta)u(\eta,t)\,{\rm d}\eta\equiv{\rm Const},~~
hb\equiv{\rm Const},~~B\equiv0,
\label{2.11}
\end{equation}
with $\widehat y$ being an arbitrary number and
\begin{equation}
\begin{cases}
hvu_y-hba_y-fhv=0,\\
hva_y-hbu_y=0.
\end{cases}
\label{2.12}
\end{equation}
Furthermore, since both $hv$ and $hb$ are constants at equilibrium, \eref{2.12} is, in fact, a linear system for $u_y$ and $a_y$, which can
be easily solved to obtain
\begin{equation}
u_y=\frac{f(hv)^2}{(hv)^2-(hb)^2},\quad a_y=\frac{f(hv)(hb)}{(hv)^2-(hb)^2}.
\label{2.13}
\end{equation} 
Since $f(y)=f_c+\beta y$, one can integrate \eref{2.13} to obtain
\begin{equation}
u(y)=\frac{(hv)^2}{(hv)^2-(hb)^2}\left(f_cy+\frac{\beta}{2}y^2\right)+u_c,\quad
a(y)=\frac{(hv)(hb)}{(hv)^2-(hb)^2}\left(f_cy+\frac{\beta}{2}y^2\right)+a_c,
\label{2.16}
\end{equation}
where $u_c$ and $a_c$ are constants of integration.

In order to derive a WB scheme capable of exactly preserving steady states satisfying \eref{2.11}, \eref{2.13}, we use the flux
globalization approach. To do so, we rewrite the augmented system \eref{2.5}--\eref{2.4} in the following quasi-conservative form:
\begin{equation}
\bm W_t+\bm H(\bm W)_y=\bm0,\quad\bm W:=\begin{pmatrix}\bm U\\B\end{pmatrix},\quad
\bm H(\bm W):=\begin{pmatrix}\bm K(\bm U)\\vB\end{pmatrix},\quad\bm K(\bm U):=\bm F(\bm U)-\bm R(\bm U),
\label{2.7}
\end{equation}
with the global variables $\bm R=(0,R_2,R_3,R_4,R_5)^\top$:
\begin{equation}
\bm R(\bm U)=\int\limits_{\widehat y}^y\big[Q(\bm U(\xi,t))\bm U_\xi(\xi,t)-\bm S(\bm U(\xi,t))\big]{\rm d}\xi,
\label{2.8}
\end{equation}
whose nonzero components are
\begin{equation}
\begin{aligned}
&R_2=-\int\limits_{\widehat y}^y\big[a(\xi,t)(hb)_\xi(\xi,t)\,-f(\xi)(hv)(\xi,t)\big]{\rm d}\xi,\\
&R_3=-\int\limits_{\widehat y}^y\big[b(\xi,t)(hb)_\xi(\xi,t)+f(\xi)(hu)(\xi,t)+gh(\xi,t)Z_\xi(\xi)\big]{\rm d}\xi,\\
&R_4=-\int\limits_{\widehat y}^yu(\xi,t)(hb)_\xi(\xi,t)\,{\rm d}\xi,\quad
R_5=-\int\limits_{\widehat y}^yv(\xi,t)(hb)_\xi(\xi,t)\,{\rm d}\xi.
\end{aligned}  
\label{2.9}
\end{equation}
Note that the derivative of the global flux, $\bm K(\bm U)_y$, can be rewritten in the following matrix-vector form:
\begin{equation*}
\bm K(\bm U)_y=\bm F(\bm U)_y-\bm R(\bm U)_y=M(\bm U)\bm E_y-fhv\,\bb e_2,
\end{equation*}
where $\bb e_2:=(0,1,0,0,0)^\top$, $\bm E:=(hv,u,E,a,hb)^\top$ is the vector of equilibrium variables with $E$ defined in \eref{2.11}, and
\begin{equation}
M(\bm U):= 
\begin{pmatrix}
1&0&0&0&0\\
u&hv&0&-hb&0\\
v&0&h&0&0\\
a&-hb&0&hv&0\\
0&0&0&0&v
\end{pmatrix}.
\label{2.15}
\end{equation}
At the steady states, $\bm K(\bm U)_y=\bm0$ or, equivalently, $M(\bm U)\bm E_y=fhv\,\bb e_2$, where $\bm E$ denotes the equilibrium
variables, which are used in WB finite-volume methods to perform piecewise polynomial interpolations/reconstructions. The obtained
piecewise polynomial approximant will be exact when the discrete data are at a steady state. Typically, equilibrium variables are constant at
steady states, but unlike similar situations in \cite{Cao2022Flux,Cao2023Flux,Kurganov2023Well}, where relations similar to \eref{2.13} were
established, only three of the components of $\bm E$ ($hv$, $E$, and $hb$) are constant at the steady states. At the same time, $u$ and $a$ are either
linear (if the Coriolis parameter $f(y)\equiv f_c$) or quadratic (if $f(y)=f_c+\beta y$) functions of $y$; see \eref{2.16}. Such functions
can, however, be exactly recovered using piecewise polynomial interpolations/reconstructions, and therefore, the WB scheme that will be
developed in \S\ref{sec22} will rely on the reconstruction of $\bm E$ instead of $\bm U$.
\begin{rmk}
Notice that \eref{2.16} is not valid if $hv=hb$. The case $hv=hb\ne0$ does not correspond to any steady state as this would not satisfy the
system \eref{2.12} for nonzero $f(y)$. The case $hv=hb\equiv0$ corresponds to a very simple ``lake-at-rest'' steady state, which will be
automatically preserved by the scheme designed in \S\ref{sec22}.
\end{rmk}

\subsection{Numerical Method}\label{sec22}
This section describes the 1-D semi-discrete divergence-free flux globalization based WB PCCU scheme for \eref{2.7}--\eref{2.9}. 
We start by introducing a 1-D uniform Cartesian grid with finite-volume cells $C_k=[y_\kmh,y_\kph]$ of size $y_\kph-y_\kmh\equiv\dy$,
$k=1,\ldots,N$. Throughout \S\ref{sec22}, we will use the lower boundary of integration $\widehat y=y_\hf$ to evaluate all of the
global quantities.

We assume that a numerical solution realized in terms of its cell averages,
$$
\xbar{\bm W}_k\approx\frac{1}{\dy}\int\limits_{C_k}\bm W(y,t)\,{\rm d}y,
$$
is available at a certain time $t$. Note that, for the sake of brevity, we omit the time dependence of $\xbar{\bm U}_k$ and other indexed
quantities here and throughout the paper. Within a semi-discrete framework, the solution is evolved in time by solving the following system
of ODEs:
\begin{equation*}
\frac{{\rm d}}{{\rm d}t}\,\xbar{\bm W}_k=-\frac{\bm{\mc{H}}_\kph-\bm{\mc{H}}_\kmh}{\dy},
\end{equation*}
where 
\begin{equation}
\bm{{\cal H}}_\kph=\frac{s_\kph^+\bm H\big(\bm W^-_\kph\big)-s_\kph^-\bm H\big(\bm W^+_\kph\big)}{s_\kph^+-s_\kph^-}+
\frac{s_\kph^+s_\kph^-}{s_\kph^+-s_\kph^-}\left(\widehat{\bm W}^+_\kph-\widehat{\bm W}^-_\kph\right)
\label{2.18}
\end{equation}
are the WB PCCU numerical fluxes from \cite{Kurganov2023Well}, $\bm H$ is given in \eref{2.7}--\eref{2.9}, and $\bm W_\kph^\pm$ and
$\widehat{\bm W}^\pm_\kph$ are two slightly different approximations of the one-sided point values of $\bm W$ at $y=y_\kph$; see
\S\ref{sec221} for details. In addition, $s_\kph^\pm$ in \eref{2.18} denote the one-sided local speeds of propagation, which can be
estimated using the largest and smallest eigenvalues of the matrix $\frac{\d\bm F}{\d\bm U}(\bm U)-Q(\bm U)$ as follows:
\begin{equation*}
\begin{aligned}
&s_\kph^+=\max\left\{v_\kph^-+\sqrt{\big(b_\kph^-\big)^2+gh_\kph^-},\,v_\kph^++\sqrt{\big(b_\kph^+\big)^2+gh_\kph^+},\,0\right\},\\
&s_\kph^-=\min\left\{v_\kph^--\sqrt{\big(b_\kph^-\big)^2+gh_\kph^-},\,v_\kph^+-\sqrt{\big(b_\kph^+\big)^2+gh_\kph^+},\,0\right\},
\end{aligned}
\end{equation*}
where the point values $h_\kph^\pm$, $v_\kph^\pm$, and $b_\kph^\pm$ will be specified in the next section.

\subsubsection{Well-Balanced Reconstruction} \label{sec221}
The development of the proposed WB scheme hinges on reconstructing the equilibrium variables $\bm E$ instead of the conservative variables
$\bm U$. We therefore first need to compute the discrete values $\bm E_k:=\big((\xbar{hv})_k,u_k,E_k,a_k,(\xbar{hb})_k\big)^\top$ out of the
available cell averages $\xbar{\bm U}_k$:
\begin{equation}
u_k=\frac{(\xbar{hu})_k}{\xbar h_k},\quad E_k=\frac{\big((\xbar{hv})_k\big)^2}{2(\xbar h_k)^2}+g(\xbar{h}_k+Z_k)-
\frac{\big((\xbar{hb})_k\big)^2}{2(\xbar h_k)^2}+P_k,\quad a_k=\frac{(\xbar{ha})_k}{\xbar h_k},
\label{2.19f}
\end{equation}
where $Z_k:=Z(y_k)$ and the values $P_k\approx\int^{y_k}_{y_\hf}fu\,{\rm d}y$ are computed using the trapezoidal rule within the following
recursive formula:
\begin{equation}
P_1=\frac{\dy}{4}\big(f_\hf u_\hf+f_1u_1\big),\quad P_k=P_{k-1}+\frac{\dy}{2}\big(f_{k-1}u_{k-1}+f_ku_k\big),\quad k=2,\ldots,N,
\label{2.20f}
\end{equation}
where $f_\hf:=f(y_\hf)$, $f_k:=f(y_k)$, and $u_\hf$ is determined by the boundary conditions.

Now equipped with the values ${\bm E}_k$, $\xbar B_k$, and $Z_k$, we compute the point values $\bm E^\pm_\kph$, $B^\pm_\kph$, and
$Z^\pm_\kph$ at the cell interfaces $y=y_\kph$. For the fields $hv$, $E$, $B$, and $Z$ we perform a generalized minmod reconstruction
described in Appendix \ref{appxA}, while for $hb$ we replace the slope in \eref{2.26} with $((hb)_y)_k=\,\xbar B_k$. The latter is motivated
by \cite{chertock2022new} and is needed to ensure that the discrete divergence-free condition \eref{2.2} is locally satisfied as long as
this condition is held at $t=0$.

Recall that while three of the equilibrium variables, $hv$, $E$, and $hb$, are constant at the steady states, the remaining two equilibrium
variables, $u(y)$ and $a(y)$, are either linear (if $f(y)\equiv f_c$) or quadratic (if $f(y)\equiv f_c+\beta y$) functions. In the former
case, we use the piecewise linear generalized minmod reconstruction, while in the latter case, we utilize the fifth-order WENO-Z
interpolation briefly described in Appendix \ref{appB}. In both cases, the implemented reconstructions recover the exact
point values $u^\pm_\kph$ and $a^\pm_\kph$ at the steady states satisfying \eref{2.16}.

Now that the point values of the equilibrium variables $\bm E^\pm_\kph$ are available, we compute the corresponding values $h^\pm_\kph$ as
follows. We first compute the water surface values $w_k=\xbar h_k+Z_k$, perform the piecewise linear reconstruction described in Appendix
\ref{appxA} to obtain the point values $w^\pm_\kph$, which, in turn, are used to set $\breve h^+_\kph:=w^+_\kph-Z^+_\kph$ or
$\breve h^-_\kph:=w^-_\kph-Z^-_\kph$. We then exactly solve (following \cite{Cheng2019new,KKLZ}) the cubic equations
\begin{equation}
\frac{\big((hv)^\pm_\kph\big)^2-\big((hb)^\pm_\kph\big)^2}{2\big(h^\pm_\kph\big)^2}+g\left(h^\pm_\kph+Z^\pm_\kph\right)+P_\kph=E^\pm_\kph,
\label{2.29}
\end{equation}
which arise from the definition of $E$ in \eref{2.11} and the global terms $P_\kph$ are evaluated using the midpoint rule within the
following recursive formula:
\begin{equation}
P_\hf=0;\quad P_\kph=P_\kmh+\dy f_ku_k,\quad k=1,\ldots,N.
\label{2.31}
\end{equation}

If the equation for $h^+_\kph$ in \eref{2.29}--\eref{2.31} has no positive solutions, we set $h^+_\kph=\breve h^+_\kph$, while if it has 
more than one positive root, we single out a root corresponding to the physically relevant solution by selecting the root closest to
$\breve h^+_\kph$. A similar algorithm is implemented to obtain $h^-_\kph$ in \eref{2.29}--\eref{2.31}. Once $h^\pm_\kph$ are obtained, we
compute 
\begin{equation*}
(hu)^\pm_\kph=h^\pm_\kph u^\pm_\kph,\quad\mbox{and}\quad(ha)^\pm_\kph=h^\pm_\kph a^\pm_\kph.
\end{equation*}

Next, we explain how to obtain the point values $\bm{\widehat W}_\kph^\pm$ that appear in the numerical diffusion terms on the right-hand
side (RHS) of \eref{2.18}. These modified point values are needed to ensure that at steady states,
$\bm{\widehat W}_\kph^-=\bm{\widehat W}_\kph^+$ and hence the numerical diffusion terms in \eref{2.18} vanish. This, in turn, guarantees the
WB property of the designed scheme as proven in \S\ref{sec223}.

We follow \cite{Kurganov2023Well} and first set $(\widehat{hv})_\kph^\pm=({hv})_\kph^\pm$, $(\widehat{hb})_\kph^\pm=({hb})_\kph^\pm$, and
$\widehat B_\kph^\pm=B_\kph^\pm$ as $hv$, $hb$, and $B$ are constant at steady states. The values $\widehat h_\kph^\pm$ are obtained by
solving the following modified versions of the nonlinear cubic equations in \eref{2.29}:
\begin{equation}
\frac{\big((hv)^\pm_\kph\big)^2-\big((hb)^\pm_\kph\big)^2}{2\big(\,\widehat h^\pm_\kph\big)^2}+g\left(\widehat h^\pm_\kph+Z_\kph\right)+
P_\kph=E^\pm_\kph,
\label{2.44}
\end{equation}
where there only change made is the replacement of $Z^\pm_\kph$ with $Z_\kph:=\big(Z_\kph^-+Z_\kph^+\big)/2$.

Equations \eref{2.44} are solved exactly the same way equations \eref{2.29} have been solved, and it is easy to see that whenever the data 
is locally at steady state, that is, if $(hv)^-_\kph=(hv)^+_\kph$, $(hb)^-_\kph=(hb)^+_\kph$, and $E^-_\kph=E^+_\kph$, then
$\widehat h^-_\kph=\widehat h^+_\kph$.

Finally, we compute $(\widehat{hu})_\kph^\pm=\widehat h_\kph^\pm\cdot u_\kph^\pm$,
$(\widehat{ha})_\kph^\pm=\widehat h_\kph^\pm\cdot a_\kph^\pm$.

\subsubsection{Well-Balanced Evaluation of the Global Fluxes}\label{sec222}
In order to compute the numerical fluxes \eref{2.18}, we first use the cell interface point values computed in \S\ref{sec221} to evaluate
the discrete values of the global fluxes $\bm K(\bm U)$ appearing in \eref{2.7}--\eref{2.9}:
\begin{equation}
\bm K(\bm U_\kph^\pm)=\bm F(\bm U_\kph^\pm)-\bm R_\kph^\pm,
\label{2.36}
\end{equation}
where $\bm R_\kph^\pm$ denote the numerical approximation of the integrals $\bm R$ appearing in \eref{2.9}. In order to ensure the resulting
method is WB and since the integrals $\bm R$ contain the nonconservative products, we follow \cite{Cao2022Flux,Cao2023Flux,Kurganov2023Well}
and use the path-conservative technique to evaluate $\bm R_\kph^\pm$ using the following recursive formulae:
\begin{equation}
\bm R_\hf^-=\bm0,\quad\bm R_\hf^+=\bm Q_{\bm\Psi,\hf},\quad\bm R_\kph^-=\bm R_\kmh^++\bm Q_k,\quad
\bm R_\kph^+=\bm R_\kph^-+\bm Q_{\bm\Psi,\kph},\quad k=1\ldots,N,
\label{2.37}
\end{equation}
where $\bm Q_k$ and $\bm Q_{\bm\Psi,\kph}$ are found using appropriate quadratures for 
\begin{equation*}
\bm Q_k\approx\int\limits_{C_k}\left[Q(\bm U)\bm U_y+\bm S(\bm U)\right]{\rm d}y\quad\mbox{and}\quad
\bm Q_{\bm\Psi,\kph}\approx\int\limits_0^1Q(\bm\Psi_\kph(s))\bm\Psi'_\kph(s)\,{\rm d}s.
\end{equation*}
Here, $\bm\Psi_\kph(s):=\bm\Psi(s;\bm U_\kph^-,\bm U_\kph^+)$ is a certain path connecting the states $\bm U_\kph^-$ and $\bm U_\kph^+$ at
the cell interface $y=y_\kph$. In order to ensure the WB property, we use a line segment connecting the left and right cell interface values
of the equilibrium variables:
\begin{equation*}
\bm E_\kph(s):=\bm E_\kph^-+s(\bm E_\kph^+-\bm E_\kph^-),\quad s\in[0,1],
\end{equation*}
which corresponds to a particular path $\bm\Psi_\kph(s)$, whose detailed structure is only given implicitly; see \cite{Kurganov2023Well} for
details. We then use the technique introduced in \cite{Kurganov2023Well} to compute
\begin{equation*}
\bm Q_k\approx\bm F(\bm U_\kph^-)-\bm F(\bm U_\kmh^+)-\hf\left[M(\bm U_\kph^-)+M(\bm U_\kmh^+)\right]\left(\bm E_\kph^--\bm E_\kmh^+\right)
+\dy f_k(\xbar{hv})_k\cdot\bb{e}_2,
\end{equation*}
and
\begin{equation}
\bm Q_{\bm\Psi,\kph}\approx\bm F(\bm U_\kph^+)-\bm F(\bm U_\kph^-)-\hf\left[M(\bm U_\kph^+)+M(\bm U_\kph^-)\right]\left(\bm E_\kph^+-
\bm E_\kph^-\right),
\label{2.26f}
\end{equation}
where $\bm F(\bm U)$ and $M(\bm U)$ are defined in \eref{2.6} and \eref{2.15}, respectively.

\subsubsection{Well-Balanced Property}\label{sec223}
We now prove that the proposed 1-D scheme exactly preserves the steady-states \eref{2.11}, \eref{2.16}.
\begin{thm}
The 1-D semi-discrete flux globalization based WB PCCU scheme is WB in the sense that it can exactly preserve the family of
steady states in \eref{2.11}, \eref{2.16}.
\end{thm}
\begin{proof}
Assume that at a certain time level, the computed solution is at the steady state satisfying \eref{2.11} and \eref{2.16} at the discrete
level, namely: 
\begin{equation}
(hv)_\kph^\pm=(\xbar{hv})_k\equiv(hv)_{\rm eq},\quad E_\kph^\pm=\xbar E_k\equiv E_{\rm eq},\quad
(hb)_\kph^\pm=(\xbar{hb})_k\equiv(hb)_{\rm eq},\quad\forall k,
\label{2.50}
\end{equation} 
where $(hv)_{\rm eq}$, $E_{\rm eq}$, and $(hb)_{\rm eq}$ are constants, and
\begin{equation}
u_\kph^-=u_\kph^+=:u_\kph,\quad a_\kph^-=a_\kph^+=:a_\kph,\quad\forall k.
\label{2.51}
\end{equation} 
In order to show that the steady states are preserved exactly, we must show that
\begin{equation}
\bm K(\bm U_\kph^+)=\bm K(\bm U_\kph^-)=\bm K(\bm U_\kmh^+),\quad\forall k.
\label{2.52}
\end{equation}

We begin with the first equality in \eref{2.52}, which is, according to \eref{2.36}, equivalent to
$$
\bm F(\bm U_\kph^+)-\bm R_\kph^+=\bm F(\bm U_\kph^-)-\bm R_\kph^-, 
$$
which then can be rewritten using the definition of $\bm R_\kph^\pm$ in \eref{2.37}, \eref{2.26f} as
$$ 
\bm F(\bm U_\kph^+)-\bm F(\bm U_\kph^-)-\bm Q_{\bm\Psi,\kph}=\hf\left[M(\bm U_\kph^+)+M(\bm U_\kph^-)\right]\left(\bm E_\kph^+-
\bm E_\kph^-\right)=\bm0,
$$
which is true due to the assumptions in \eref{2.50} and \eref{2.51}.

We then proceed with proving that $\bm K(\bm U_\kmh^+)=\bm K(\bm U_\kph^-)$, which, similarly to the first equality in \eref{2.52}, can be
equivalently rewritten as
\begin{equation}
\hf\left[M(\bm U_\kph^-)+M(\bm U_\kmh^+)\right]\left(\bm E_\kph^--\bm E_\kmh^+\right)-\dy f_k(\xbar{hv})_k\cdot\bb{e}_2=\bm0.
\label{2.31f}
\end{equation}
The assumptions in \eref{2.50} immediately imply that the equalities in \eref{2.31f} are valid for the first, third, and fifth components.
However, proving the second and fourth components in \eref{2.31f} is less straightforward since the profiles of $u$ and $a$ are (generally)
nonconstant when at steady states.

The second and fourth components of \eref{2.31f} read as
\begin{equation*}
\begin{aligned}
\big[u_\kph^-+u_\kmh^+\big]\big((hv)_\kph^-&-(hv)_\kmh^+\big)+\big[(hv)_\kph^-+(hv)_\kmh^+\big]\big(u_\kph^--u_\kmh^+\big)\\
&-\big[(hb)_\kph^-+(hb)_\kmh^+\big]\big(a_\kph^--a_\kmh^+\big)-2\dy f_k(\xbar{hv})_k=0,\\
\big[a_\kph^-+a_\kmh^+\big]\big((hv)_\kph^-&-(hv)_\kmh^+\big)+\big[(hv)_\kph^-+(hv)_\kmh^+\big]\big(a_\kph^--a_\kmh^+\big)\\
&-\big[(hb)_\kph^-+(hb)_\kmh^+\big]\big(u_\kph^--u_\kmh^+\big)=0.
\end{aligned}
\end{equation*}
Using \eref{2.50}, the above can be simplified to obtain that \eref{2.31f} is equivalent to
\begin{equation}
\begin{aligned}
&(hv)_{\rm eq}\big(u_\kph^--u_\kmh^+\big)-(hb)_{\rm eq}\big(a_\kph^--a_\kmh^+\big)-\dy f_k(hv)_{\rm eq}=0,\\
&(hv)_{\rm eq}\big(a_\kph^--a_\kmh^+\big)-(hb)_{\rm eq}\big(u_\kph^--u_\kmh^+\big)=0.
\end{aligned}
\label{2.53}
\end{equation}
Finally, we recall that our piecewise linear (in the case when $\beta=0$) or piecewise quadratic (in the case when $\beta\ne0$)
reconstruction of $u$ and $a$ is exact when the computed solution is at the discrete steady state. Therefore, in this case, we use
\eref{2.16} to obtain
\begin{equation*}
\begin{aligned}
&u_{k\pm\hf}^\mp=u(y_{k\pm\hf})=\frac{(hv)_{\rm eq}^2}{(hv)_{\rm eq}^2-(hb)_{\rm eq}^2}\Big(f_cy_{k\pm\hf}+\frac{\beta}{2}y_{k\pm\hf}^2\Big)
+u_c,\\
&a_{k\pm\hf}^\mp=a(y_{k\pm\hf})=\frac{(hv)_{\rm eq}(hb)_{\rm eq}}{(hv)_{\rm eq}^2-(hb)_{\rm eq}^2}\Big(f_cy_{k\pm\hf}+
\frac{\beta}{2}y_{k\pm\hf}^2\Big)+a_c,
\end{aligned}
\end{equation*}
which we substitute into \eref{2.53} to verify that the equalities there are true. This completes the proof of the theorem.
\end{proof}

\section{2-D Well-Balanced PCCU Scheme for MRSW Equations}\label{sec3}
\subsection{Governing Equations}\label{sec31}
In this section, we extend our studies to the 2-D MRSW system. As in the 1-D case, we modify the 2-D MRSW system by including the
Godunov-Powell source terms to help enforce the divergence-free constraint at the discrete level. The adjusted 2-D system then reads as
\begin{equation}
\begin{aligned}
&h_t+(hu)_x+(hv)_y=0,\\
&(hu)_t+\left(hu^2+\frac{g}{2}h^2-ha^2\right)_x+(huv-hab)_y=fhv-ghZ_x-a\left[(ha)_x+(hb)_y\right],\\
&(hv)_t+(huv-hab)_x+\left(hv^2+\frac{g}{2}h^2-hb^2\right)_y=-fhu-ghZ_y-b\left[(ha)_x+(hb)_y\right],\\
&(ha)_t+(hbu-hav)_y=-u\left[(ha)_x+(hb)_y\right],\\
&(hb)_t+(hav-hbu)_x=-v\left[(ha)_x+(hb)_y\right].
\end{aligned}
\label{3.1}
\end{equation}
It is easy to show that a solution of \eref{3.1} satisfies the 2-D divergence-free constraint 
\begin{equation}
(ha)_x+(hb)_y=0,
\label{3.2}
\end{equation}
as long as it is met initially. The system \eref{3.1} can be rewritten in the following vector form:
\begin{equation}
\bm U_t+\bm F^x(\bm U)_x+\bm F^y(\bm U)_y=Q^x(\bm U)\bm U_x+Q^y(\bm U)\bm U_y+\bm S^x(\bm U)+\bm S^y(\bm U), 
\label{3.4}
\end{equation}
where $\bm U=(h,hu,hv,ha,hb)^\top$ and
\begin{equation}
\hspace*{-0.18cm}
\begin{aligned}
&\bm F^x(\bm U)=\begin{pmatrix}hu\\hu^2+\frac{g}{2}h^2-ha^2\\huv-hab\\0\\hbu-hav\end{pmatrix}\!,~
Q^x(\bm U)=\begin{pmatrix}0&0&0&0&0\\0&0&0&-a&0\\0&0&0&-b&0\\0&0&0&-u&0\\0&0&0&-v&0\end{pmatrix}\!,~
\bm S^x(\bm U)=\begin{pmatrix}0\\fhv-ghZ_x\\0\\0\\0\end{pmatrix}\!,\\
&\bm F^y(\bm U)=\begin{pmatrix}hv\\huv-hab\\hv^2+\frac{g}{2}h^2-hb^2\\hav-hbu\\0\end{pmatrix}\!,~
Q^y(\bm U)=\begin{pmatrix}0&0&0&0&0\\0&0&0&0&-a\\0&0&0&0&-b\\0&0&0&0&-u\\0&0&0&0&-v\end{pmatrix}\!,~
\bm S^y(\bm U)=\begin{pmatrix}0\\0\\-fhu-ghZ_y\\0\\0\end{pmatrix}\!.
\end{aligned}
\label{3.5}
\end{equation}
Notice that the additional terms $Q^x(\bm U)\bm U_x+Q^y(\bm U)\bm U_y$ are theoretically zero due to \eref{3.2}, but their inclusion help to
enforce the divergence-free condition in \eref{3.2} at the discrete level. 

As in the 1-D case, we use the idea introduced in \cite{chertock2022new} to locally preserve the discrete divergence-free constraint: We
introduce the new variables $A:=(ha)_x$ and $B:=(hb)_y$ that satisfy the additional evolution equations for the magnetic field derivatives
\begin{equation}
\begin{aligned}
&A_t+\big(uA-hbu_y\big)_x+\big(vA+hav_x\big)_y=0,\\
&B_t+\big(uB+hbu_y\big)_x+\big(vB-hav_x\big)_y=0,   
\end{aligned}
\label{3.3}
\end{equation}
which are obtained by differentiating the $(ha)$- and $(hb)$-equation in \eref{3.1} with respect to $x$ and $y$, respectively. As shown in
\cite{chertock2022new}, it is relatively easy to discretize the augmented system \eref{3.4}--\eref{3.3} and design a numerical scheme that
preserves the divergence-free constraint \eref{3.2}. However, the resulting scheme will not necessarily be WB. 

In order to derive a WB scheme for the augmented system \eref{3.4}--\eref{3.3}, we follow the 2-D flux globalization approach recently
introduced in \cite{CKL23} and rewrite the studied system in the following quasi-conservative form:
\begin{equation}
\begin{aligned}
&\bm W_t+\bm H^x(\bm W)_x+\bm H^y(\bm W)_y=\bm0,\\
\bm W:=\begin{pmatrix}\bm U\\A\\B\end{pmatrix},\quad&\bm H^x(\bm W):=\begin{pmatrix}\bm K^x(\bm U)\\uA-hbu_y\\uB+hbu_y\end{pmatrix},\quad
\bm H^y(\bm W):=\begin{pmatrix}\bm K^y(\bm U)\\vA+hav_x\\vB-hav_x\end{pmatrix},\\[0.4ex]
\bm K^x(\bm U)&:=\bm F^x(\bm U)-\bm R^x(\bm U),\quad\bm K^y(\bm U):=\bm F^y(\bm U)-\bm R^y(\bm U),
\end{aligned}
\label{3.6}    
\end{equation} 
with the global variables $\bm R^x=(0,R^x_2,R^x_3,R^x_4,R^x_5)^\top$ and $\bm R^y=(0,R^y_2,R^y_3,R^y_4,R^y_5)^\top$:
\begin{equation}
\begin{aligned}
&\bm R^x(\bm U)=\int\limits_{\widehat x}^x\left[Q^x(\bm U(\xi,y,t))\bm U_\xi(\xi,y,t)-\bm S^x(\bm U(\xi,y,t))\right]{\rm d}\xi,\\
&\bm R^y(\bm U)=\int\limits_{\widehat y}^y\left[Q^y(\bm U(x,\eta,t))\bm U_\eta(x,\eta,t)-\bm S^y(\bm U(x,\eta,t))\right]{\rm d}\eta,
\end{aligned}
\label{3.7}
\end{equation}
whose nonzero components are
\begin{equation}
\begin{aligned}
&R^x_2=-\int\limits_{\widehat x}^x\left[a(\xi,y,t)(ha)_\xi(\xi,y,t)-f(y)(hv)(\xi,y,t)+gh(\xi,y,t)Z_\xi(\xi,y)\right]{\rm d}\xi,\\
&R^x_3=-\int\limits_{\widehat x}^xb(\xi,y,t)(ha)_\xi(\xi,y,t)\,{\rm d}\xi,\\
&R^x_4=-\int\limits_{\widehat x}^xu(\xi,y,t)(ha)_\xi(\xi,y,t)\,{\rm d}\xi,\quad
R^x_5=-\int\limits_{\widehat x}^xv(\xi,y,t)(ha)_\xi(\xi,y,t)\,{\rm d}\xi,\\
&R^y_2=-\int\limits_{\widehat y}^ya(x,\eta,t)(hb)_\eta(x,\eta,t)\,{\rm d}\eta,\\
&R^y_3=-\int\limits_{\widehat y}^y\left[b(x,\eta,t)(hb)_\eta(x,\eta,t)+f(\eta)(hu)(x,\eta,t)+gh(x,\eta,t)Z_\eta(x,\eta)\right]{\rm d}\eta,\\
&R^y_4=-\int\limits_{\widehat y}^yu(x,\eta,t)(hb)_\eta(x,\eta,t)\,{\rm d}\eta,\quad
R^y_5=-\int\limits_{\widehat y}^yv(x,\eta,t)(hb)_\eta(x,\eta,t)\,{\rm d}\eta,
\end{aligned}  
\label{3.8}
\end{equation}
and $\widehat x$ and $\widehat y$ are arbitrary numbers.

Note that as in the 1-D case, the corresponding derivatives of the global fluxes, $\bm K^x(\bm U)_x$ and $\bm K^y(\bm U)_y$, can be
rewritten in the following matrix-vector forms:
\begin{equation}
\begin{aligned}
&\bm K^x(\bm U)_x=\bm F^x(\bm U)_x-\bm R^x(\bm U)_x=M^x(\bm U)\bm E^x_x,\\
&\bm K^y(\bm U)_y=\bm F^y(\bm U)_y-\bm R^y(\bm U)_y=M^y(\bm U)\bm E^y_y,
\end{aligned}
\label{3.14}
\end{equation}
where
\begin{equation}
M^x(\bm U):=\begin{pmatrix}1&0&0&0&0\\u&h&0&0&0\\v&0&hu&0&-ha\\0&0&0&u&0\\b&0&-ha&0&hu\end{pmatrix},\quad
M^y(\bm U):=\begin{pmatrix}1&0&0&0&0\\u&hv&0&-hb&0\\v&0&h&0&0\\a&-hb&0&hv&0\\0&0&0&0&v\end{pmatrix},
\label{3.15}
\end{equation}
and $\bm E^x:=(hu,E^x,v,ha,b)^\top$ and $\bm E^y:=(hv,u,E^y,a,hb)^\top$ with
\begin{equation}
\begin{aligned}
&E^x:=\frac{u^2}{2}+g(h+Z)-\frac{a^2}{2}+P^x,&&P^x:=-\int\limits_{\widehat x}^xf(y)v(\xi,y,t)\,{\rm d}\xi,\\
&E^y:=\frac{v^2}{2}+g(h+Z)-\frac{b^2}{2}+P^y,&&P^y:=\int\limits_{\widehat y}^yf(\eta)u(x,\eta,t)\,{\rm d}\eta.
\end{aligned}
\label{3.11a}
\end{equation}

Note that general steady states of the 2-D MRSW system satisfy $\bm K^x(\bm U)_x+\bm K^y(\bm U)_y=\bm0$ or, equivalently,
$M^x(\bm U)\bm E^x_x+M^y(\bm U)\bm E^y_y=\bm0$. It is, however, very hard to design a WB numerical method capable of preserving general
(genuinely 2-D) equilibria. The WB numerical method, which we present in the next section, is constructed to preserve the following
families of quasi 1-D steady states exactly:
\begin{equation}
(hu)_x\equiv0,\quad v\equiv0,\quad(ha)_x\equiv0,\quad b\equiv0,\quad E^x_x\equiv0,\quad E^y_y\equiv0,
\label{3.12}
\end{equation}
and 
\begin{equation}
u\equiv0,\quad(hv)_y\equiv0,\quad a\equiv0,\quad(hb)_y\equiv0,\quad E^x_x\equiv0,\quad E^y_y\equiv0,
\label{3.13}
\end{equation}
for each of which both $\bm K^x(\bm U)_x=M^x(\bm U)\bm E^x_x=\bm0$ and $\bm K^y(\bm U)_y=M^y(\bm U)\bm E^y_y=\bm0$.

\subsection{Numerical Method}\label{sec32}
In this section, we extend the proposed 1-D divergence-free flux globalization based WB PCCU scheme to the 2-D MRSW system
\eref{3.6}--\eref{3.8}.

We assume that at a certain time $t$, the numerical solution realized in terms of its cell averages
$$
\xbar{\bm W}_{j,k}\approx\frac{1}{\dx\dy}\iint\limits_{C_{j,k}}\bm U(x,y,t)\,{\rm d}x\,{\rm d}y,
$$
is available. Here, $C_{j,k}=[x_\jmh,x_\jph]\times[y_\kmh,y_\kph]$ are the 2-D finite volume cells assumed to be uniform, that is,
$x_\jph-x_\jmh\equiv\dx$, $j=1,\ldots,N_x$, and $y_\kph-y_\kmh\equiv\dy$, $k=1,\ldots,N_y$. The cell averages are evolved in time using the
following semi-discretization of \eref{3.6}:
\begin{equation*}
\frac{\rm d}{{\rm d}t}\,\xbar{\bm W}_{j,k}=-\frac{\bm{{\cal H}}^x_{\jph,k}-\bm{{\cal H}}^x_{\jmh,k}}{\dx}-
\frac{\bm{{\cal H}}^y_{j,\kph}-\bm{{\cal H}}^y_{j,\kmh}}{\dy},
\end{equation*}
where 
\begin{equation}
\begin{aligned}
\bm{{\cal H}}^x_{\jph,k}&=\frac{s_{\jph,k}^+\bm H^x\big(\bm W^{\rm E}_{j,k}\big)-s_{\jph,k}^-\bm H^x\big(\bm W^{\rm W}_{j+1,k}\big)}
{s_{\jph,k}^+-s_{\jph,k}^-}+\frac{s_{\jph,k}^+s_{\jph,k}^-}{s_{\jph,k}^+-s_{\jph,k}^-}
\left(\widehat{\bm W}^{\,\rm W}_{j+1,k}-\widehat{\bm W}^{\,\rm E}_{j,k}\right),\\
\bm{{\cal H}}^y_{j,\kph}&=\frac{s_{j,\kph}^+\bm H^y\big(\bm W^{\rm N}_{j,k}\big)-s_{j,\kph}^-\bm H^y\big(\bm W^{\rm S}_{j,k+1}\big)}
{s_{j,\kph}^+-s_{j,\kph}^-}+\frac{s_{j,\kph}^+s_{j,\kph}^-}{s_{j,\kph}^+-s_{j,\kph}^-}
\left(\widehat{\bm W}^{\,\rm S}_{j,k+1}-\widehat{\bm W}^{\,\rm N}_{j,k}\right),
\end{aligned}
\label{3.17}
\end{equation}
are the WB PCCU numerical fluxes from \cite{CKL23}, $\bm H^x$ and $\bm H^y$ are defined in \eref{3.6}--\eref{3.8}, and
$\bm W^{{\rm E,W,N,S}}_{j,k}$ and $\widehat{\bm W}_{j,k}^{\,\rm E,W,N,S}$ are two slightly different approximations of the one-sided point
values of $\bm W$ at the cell interfaces of $C_{j,k}$, see \S\ref{sec321} for details. In addition, $s_{\jph,k}^\pm$ and $s_{j,\kph}^\pm$ in\eref{3.17} are the one-sided speeds of propagation in the $x$- and $y$-directions, respectively. We compute them using the largest and
smallest eigenvalues of matrices $\frac{\partial\bm F^x}{\partial\bm U}(\bm U)-Q^x(\bm U)$ and
$\frac{\partial\bm F^y}{\partial\bm U}(\bm U)-Q^y(\bm U)$ and obtain
$$
\begin{aligned}
&s_{\jph,k}^+=\max\left\{u_{j,k}^{\rm E}+\sqrt{\big(a_{j,k}^{\rm E}\big)^2+gh_{j,k}^{\rm E}},\,
u_{j+1,k}^{\rm W}+\sqrt{\big(a_{j+1,k}^{\rm W}\big)^2+gh_{j+1,k}^{\rm W}},\,0\right\},\\
&s_{\jph,k}^-=\min\left\{u_{j,k}^{\rm E}-\sqrt{\big(a_{j,k}^{\rm E}\big)^2+gh_{j,k}^{\rm E}},\,
u_{j+1,k}^{\rm W}-\sqrt{\big(a_{j+1,k}^{\rm W}\big)^2+gh_{j+1,k}^{\rm W}},\,0\right\},\\
&s_{j,\kph}^+=\max\left\{v_{j,k}^{\rm N}+\sqrt{\big(b_{j,k}^{\rm N}\big)^2+gh_{j,k}^{\rm N}},\,
v_{j,k+1}^{\rm S}+\sqrt{\big(b_{j,k+1}^{\rm S}\big)^2+gh_{j,k+1}^{\rm S}},\,0\right\},\\
&s_{j,\kph}^-=\min\left\{v_{j,k}^{\rm N}-\sqrt{\big(b_{j,k}^{\rm N}\big)^2+gh_{j,k}^{\rm N}},\,
v_{j,k+1}^{\rm S}-\sqrt{\big(b_{j,k+1}^{\rm S}\big)^2+gh_{j,k+1}^{\rm S}},\,0\right\},
\end{aligned}
$$
where the point values $h_{j,k}^{\rm E,W,N,S}$, $u_{j,k}^{\rm E,W}$, $v_{j,k}^{\rm N,S}$, $a_{j,k}^{\rm E,W}$, and $b_{j,k}^{\rm N,S}$ are
specified in the next section.

\subsubsection{Well-Balanced Reconstruction}\label{sec321}
In order to ensure the proposed method preserves the discrete quasi 1-D steady-states \eref{3.12} and \eref{3.13}, it is crucial to
reconstruct the equilibrium variables $\bm E^x:=(hu,E^x,v,ha,b)^\top$ and $\bm E^y:=(hv,u,E^y,a,hb)^\top$ in the $x$- and $y$-directions,
respectively, rather than the conservative variables $\bm U$. To this end, we first need to compute the discrete values
$\bm E^x_{j,k}:=((\xbar{hu})_{j,k},E^x_{j,k},v_{j,k},(\xbar{ha})_{j,k},b_{j,k})^\top$ and
$\bm E^y_{j,k}:=((\xbar{hv})_{j,k},u_{j,k},E^y_{j,k},a_{j,k},(\xbar{hb})_{j,k})^\top$ from the available cell averages $\xbar{\bm U}_{j,k}$:
\begin{equation*}
\begin{aligned}
&v_{j,k}=\frac{(\xbar{hv})_{j,k}}{\xbar h_{j,k}},\quad E^x_{j,k}=\frac{\big((\xbar{hu})_{j,k}\big)^2}{2(\xbar h_{j,k})^2}+
g(\xbar h_{j,k}+Z_{j,k})-\frac{\big((\xbar{ha})_{j,k}\big)^2}{2(\xbar h_{j,k})^2}+P^x_{j,k},\quad
b_{j,k}=\frac{(\xbar{hb})_{j,k}}{\xbar h_{j,k}},\\
&u_{j,k}=\frac{(\xbar{hu})_{j,k}}{\xbar h_{j,k}},\quad E^y_{j,k}=\frac{\big((\xbar{hv})_{j,k}\big)^2}{2(\xbar h_{j,k})^2}+
g(\xbar h_{j,k}+Z_{j,k})-\frac{\big((\xbar{hb})_{j,k}\big)^2}{2(\xbar h_{j,k})^2}+P^y_{j,k},\quad
a_{j,k}=\frac{(\xbar{ha})_{j,k}}{\xbar h_{j,k}},
\end{aligned}
\end{equation*}
where $Z_{j,k}:=Z(x_j,y_k)$ and the discrete values $P^x_{j,k}$ and $P^y_{j,k}$ of the global integral terms in \eref{3.11a} are computed by
setting $\widehat x=x_\hf$ and $\widehat y=y_\hf$, and using the trapezoidal rule within following recursive formulae:
\begin{equation*}
\begin{aligned}
&P^x_{1,k}=-\frac{f_k\dx}{4}\big(v_{\hf,k}+v_{1,k}\big),&&P^x_{j,k}=P^x_{j-1,k}-\frac{f_k\dx}{2}\big(v_{j-1,k}+v_{j,k}\big),&&
j=2,\ldots,N_x,\\
&P^y_{1,k}=\frac{\dy}{4}\big(f_\hf u_{j,\hf}+f_1u_{j,1}\big),&&P^y_{j,k}=P^y_{j,k-1}-\frac{\dy}{2}\big(f_{k-1}u_{j,k-1}+f_ku_{j,k}\big),&&
k=2,\ldots,N_y.
\end{aligned}
\end{equation*}
Here, the values $u_{j,\hf}$ and $v_{\hf,k}$ are determined by the boundary conditions.

Equipped with the values $\bm E^x_{j,k}$, $\bm E^y_{j,k}$, $\xbar A_{j,k}$, $\xbar B_{j,k}$, and $Z_{j,k}$, we compute the corresponding
cell interface point values $(\bm E^x)_{j,k}^{\rm E,W}$, $(\bm E^y)_{j,k}^{\rm N,S}$, $A_{j,k}^{\rm E,W,N,S}$, $B_{j,k}^{\rm E,W,N,S}$, and
$Z_{j,k}^{\rm E,W,N,S}$ in cell $C_{j,k}$. All of these point values except for $(ha)^{\rm E,W}$ and $(hb)^{\rm N,S}$ are reconstructed
using the generalized minmod limiter described in Appendix \ref{appxA}. Computing $(ha)^{\rm E,W}$ and $(hb)^{\rm N,S}$ with the help of
the generalized minmod or any other conventional nonlinear limiter, however, would not guarantee the exact preservation of a discrete
version of the divergence-free constraint \eref{3.2}. We therefore follow \cite{chertock2022new} and enforce the discrete divergence-free
condition $((ha)_x)_{j,k}+((hb)_y)_{j,k}=0$ for all $j,k$ by computing
\begin{equation}
\begin{aligned}
&(ha)^{\rm E}_{j,k}=(\xbar{ha})_{j,k}+\sigma_{j,k}\xbar A_{j,k}\frac{\dx}{2},&&
(ha)^{\rm W}_{j,k}=(\xbar{ha})_{j,k}-\sigma_{j,k}\xbar A_{j,k}\frac{\dx}{2},\\
&(hb)^{\rm N}_{j,k}=(\xbar{hb})_{j,k}+\sigma_{j,k}\xbar B_{j,k}\frac{\dy}{2},&&
(hb)^{\rm S}_{j,k}=(\xbar{hb})_{j,k}-\sigma_{j,k}\xbar B_{j,k}\frac{\dy}{2},
\end{aligned}
\label{3.29}
\end{equation}
where $\sigma_{j,k}=\min\big\{1,\sigma^x_{j,k},\sigma^y_{j,k}\big\}$ and
\begin{equation*}
\begin{aligned}
\sigma^x_{j,k}&=\left\{
\begin{aligned}
&\min\bigg\{1,\frac{((ha)_x)_{j,k}}{\xbar A_{j,k}}\bigg\}&&\mbox{if}~((ha)_x)_{j,k}\,\xbar A_{j,k}>0,\\
&0&&\mbox{otherwise},
\end{aligned}
\right.\\
\sigma^y_{j,k}&=\left\{
\begin{aligned}
&\min\bigg\{1,\frac{((hb)_y)_{j,k}}{\xbar B_{j,k}}\bigg\}&&\mbox{if}~((hb)_y)_{j,k}\,\xbar B_{j,k}>0,\\
&0&&\mbox{otherwise}.
\end{aligned}
\right.
\end{aligned}
\end{equation*}
Here, $((ha)_x)_{j,k}$ and $((hb)_y)_{j,k}$ are the slopes computed using the generalized minmod limiter described in Appendix \ref{appxA}.

As it was shown in \cite{chertock2022new}, the slopes in \eref{3.29} satisfy the local discrete divergence-free property
$\sigma_{j,k}(\xbar A_{j,k}+\xbar B_{j,k})=0$ as long as $\xbar A_{j,k}+\xbar B_{j,k}$ is identically zero initially for all $j,k$.
\begin{rmk}
A scaling similar to the $\sigma_{j,k}$-scaling in \eref{3.29} is redundant in the 1-D case as the 1-D divergence-free condition simply
implies $hb\equiv{\rm Const}$. 
\end{rmk}
\begin{rmk}
Relying on the reconstructions of the magnetic field in \eref{3.29}, one can follow the proof of \cite[Theorem 2.2]{chertock2022new} to
establish the following local divergence-free property of the proposed scheme as stated in the following theorem whose proof we omit for 
the sake of brevity.
\begin{thm}\label{thm3.5}
For the proposed 2-D semi-discrete flux globalization based WB PCCU scheme, the local divergence-free condition
\begin{equation*}
\big((ha)_x\big)_{j,k}+\big((hb)_y\big)_{j,k}=0,
\end{equation*}
holds for all $j,k$ and at all times, provided it is satisfied initially.
\end{thm}
\end{rmk}

In addition to the reconstruction of the aforementioned point values of $\bm E^x$, $\bm E^y$, $A$, $B$, and $Z$, one also needs to evaluate
the point values $(u_y)_{j,k}^{\rm E,W}$ and $(v_x)_{j,k}^{\rm N,S}$ at the corresponding cell interfaces, as they appear in the $A$- and
$B$-equation fluxes in \eref{3.6}. We compute these velocity derivatives using first-order approximations:
\begin{equation}
(u_y)^{\rm E}_{j,k}=(u_y)_{j,k},\quad(u_y)^{\rm W}_{j,k}=(u_y)_{j,k},\quad(v_x)^{\rm N}_{j,k}=(v_x)_{j,k},\quad
(v_x)^{\rm S}_{j,k}=(v_x)_{j,k}.
\label{3.31}
\end{equation}
Recall that the slopes $(u_y)_{j,k}$ and $(v_x)_{j,k}$ have been already computed.
\begin{rmk}
Note that while the reconstructions in \eref{3.31} may lead to a first-order approximation of the magnetic field derivatives $A$ and $B$,
they do not affect the second order of accuracy in the approximation of $\bm U$.
\end{rmk}

Now that the point values of the equilibrium variables $(\bm E^x)_{j,k}^{\rm E,W}$ and $(\bm E^y)_{j,k}^{\rm N,S}$ are available, we compute
the corresponding values $h^{\rm E,W,N,S}$ as follows. As in the 1-D case, we first compute the water surface values
$w_{j,k}=\,\xbar h_{j,k}+Z_{j,k}$ and perform the piecewise linear reconstruction described in Appendix \ref{appxA} to obtain the point
values $w_{j,k}^i$, which, in turn, are used to set $\breve h_{j,k}^i:=w_{j,k}^i-Z_{j,k}$ for $i\in\{\rm E,W,N,S\}$. We then exactly solve 
the following cubic equations:
\allowdisplaybreaks
\begin{align}
&\frac{\big((hu)^{\rm E}_{j,k}\big)^2-\big((ha)^{\rm E}_{j,k}\big)^2}{2\big(h^{\rm E}_{j,k}\big)^2}+
g\left(h^{\rm E}_{j,k}+Z^{\rm E}_{j,k}\right)+P^x_{\jph,k}=(E^x)^{\rm E}_{j,k},\label{3.33}\\
&\frac{\big((hu)^{\rm W}_{j,k}\big)^2-\big((ha)^{\rm W}_{j,k}\big)^2}{2\big(h^{\rm W}_{j,k}\big)^2}+
g\left(h^{\rm W}_{j,k}+Z^{\rm W}_{j,k}\right)+P^x_{\jmh,k}=(E^x)^{\rm W}_{j,k},\label{3.34}\\
&\frac{\big((hv)^{\rm N}_{j,k}\big)^2-\big((hb)^{\rm N}_{j,k}\big)^2}{2\big(h^{\rm N}_{j,k}\big)^2}+
g\left(h^{\rm N}_{j,k}+Z^{\rm N}_{j,k}\right)+P^y_{j,\kph}=(E^y)^{\rm N}_{j,k},\label{3.35}\\
&\frac{\big((hv)^{\rm S}_{j,k}\big)^2-\big((hb)^{\rm S}_{j,k}\big)^2}{2\big(h^{\rm S}_{j,k}\big)^2}+
g\left(h^{\rm S}_{j,k}+Z^{\rm S}_{j,k}\right)+P^y_{j,\kmh}=(E^y)^{\rm S}_{j,k},\label{3.36},
\end{align}
which arise from the definitions of $E^x$ and $E^y$ in \eref{3.11a}. In \eref{3.33}--\eref{3.36}, the global terms $P^x_{\jph,k}$ and
$P^y_{j,\kph}$ are evaluated using the midpoint rule within the following recursive formulae:
\begin{equation}
\begin{aligned}
&P^x_{\hf,k}=0;&&P^x_{\jph,k}=P^x_{\jmh,k}-\dx f_kv_{j,k},\\
&P^y_{j,\hf}=0;&&P^y_{j,\kph}=P^y_{j,\kmh}+\dy f_ku_{j,k},
\end{aligned}\qquad j=1,\ldots,N_x,~k=1,\ldots,N_y.
\label{3.37}
\end{equation}

If \eref{3.33} has no positive solution, we set $h^{\rm E}_{j,k}=\breve h^{\rm E}_{j,k}$. In contrast, if it has more than one positive 
root, we single out a root corresponding to the physically relevant solution by selecting the root closest to $\breve h^{\rm E}_{j,k}$. A 
similar algorithm is implemented to obtain $h^{\rm W}_{j,k}$, $h^{\rm N}_{j,k}$, and $h^{\rm S}_{j,k}$. Once $h_{j,k}^{\rm E,W,N,S}$ are 
obtained, we compute
\begin{equation*}
\begin{aligned}
(hv)_{j,k}^{\rm E,W}=h_{j,k}^{\rm E,W}v_{j,k}^{\rm E,W},\quad(hb)_{j,k}^{\rm E,W}=h_{j,k}^{\rm E,W}b_{j,k}^{\rm E,W},\quad
(hu)_{j,k}^{\rm N,S}=h_{j,k}^{\rm N,S}u_{j,k}^{\rm N,S},\quad(ha)_{j,k}^{\rm N,S}=h_{j,k}^{\rm N,S}a_{j,k}^{\rm N,S}. 
\end{aligned}
\end{equation*}

Next, we explain how to obtain the point values $\widehat{\bm W}_{j,k}^{\,\rm E,W,N,S}$ that appear in the numerical diffusion terms on the
RHS of \eref{3.17}. As in the 1-D case, these modified values are needed to ensure that at steady states,
$\widehat{\bm W}_{j,k}^{\,\rm E}=\widehat{\bm W}_{j+1,k}^{\,\rm W}$ and $\widehat{\bm W}_{j,k}^{\,\rm N}=\widehat{\bm W}_{j,k+1}^{\,\rm S}$,
and hence the numerical diffusion terms in \eref{3.17} vanish. This, in turn, guarantees the WB property of the resulting scheme as proven
\S\ref{sec323}.

We follow \cite{CKL23} and first set
\begin{equation}
\begin{aligned}
&(\widehat{hu})_{j,k}^{\rm E}=({hu})_{j,k}^{\rm E},&&(\widehat{ha})_{j,k}^{\rm E}=({ha})_{j,k}^{\rm E},&&
\widehat A_{j,k}^{\,\rm E}=A_{j,k}^{\rm E},&&\widehat B_{j,k}^{\,\rm E}=B_{j,k}^{\rm E},\\
&(\widehat{hu})_{j+1,k}^{\rm W}=({hu})_{j+1,k}^{\rm W},&&(\widehat{ha})_{j+1,k}^{\rm W}=({ha})_{j+1,k}^{\rm W},&&
\widehat A_{j+1,k}^{\,\rm W}=A_{j+1,k}^{\rm W},&&\widehat B_{j+1,k}^{\,\rm W}=B_{j+1,k}^{\rm W},
\end{aligned}
\label{3.48}
\end{equation}
as the variables $hu$, $ha$, $A$, and $B$ do not vary along the $x$-direction when the computed solution is at either the \eref{3.12} or
\eref{3.13} steady state. We then compute the values $\widehat h_{j,k}^{\,\rm E}$ and $\widehat h_{j+1,k}^{\,\rm W}$ by solving the
following modified versions of the nonlinear equations \eref{3.33} and \eref{3.34}, respectively:
\begin{equation}
\begin{aligned}
&\frac{\big((hu)^{\rm E}_{j,k}\big)^2-\big((ha)^{\rm E}_{j,k}\big)^2}{2\big(\,\widehat h^{\,\rm E}_{j,k}\big)^2}+
g\left(\widehat h^{\,\rm E}_{j,k}+Z_{j,k}\right)+P^x_{\jph,k}=(E^x)^{\rm E}_{j,k},\\
&\frac{\big((hu)^{\rm W}_{j,k}\big)^2-\big((ha)^{\rm W}_{j,k}\big)^2}{2\big(\,\widehat h^{\,\rm W}_{j,k}\big)^2}+
g\left(\widehat h^{\,\rm W}_{j,k}+Z_{j,k}\right)+P^x_{\jmh,k}=(E^x)^{\rm W}_{j,k},
\end{aligned}
\label{3.49}
\end{equation}
where $P^x_{\jph,k}$ is defined in \eref{3.37}, and the only adjustment made to the nonlinear equations \eref{3.33} and \eref{3.34} is the
replacement of $Z_{j,k}^{\rm E}$ and $Z_{j+1,k}^{\rm W}$ with
$$
Z_{\jph,k}=\hf\left(Z_{j,k}^{\rm E}+Z_{j+1,k}^{\rm W}\right).
$$
The equations in \eref{3.49} are solved using the same cubic equation solver as in \S\ref{sec221}.

Notice that when the computed solution is at steady-state, the equations in \eref{3.49} are identical since, in this case,
$(E^x)_{j,k}^{\rm E}=(E^x)_{j+1,k}^{\rm W}$ and the equalities in \eref{3.48} hold. In addition, if $Z_{j,k}^{\rm E}=Z_{j+1,k}^{\rm W}$, the
equations in \eref{3.49} coincide with \eref{3.33} and \eref{3.34} and therefore, at these cell interfaces, we simply set
\begin{equation*}
\widehat h_{j,k}^{\,\rm E}=h_{j,k}^{\rm E}\quad\mbox{and}\quad\widehat h_{j+1,k}^{\,\rm W}=h_{j+1,k}^{\rm W}.
\end{equation*}

Finally, we obtain the discrete modifications of $hv$ and $hb$ as follows: 
\begin{equation*}
(\widehat{hv})_{j,k}^{\rm E}=\widehat h_{j,k}^{\,\rm E}\cdot v_{j,k}^{\rm E},\quad
(\widehat{hv})_{j+1,k}^{\rm W}=\widehat h_{j+1,k}^{\,\rm W}\cdot v_{j+1,k}^{\rm W},\quad
(\widehat{hb})_{j,k}^{\rm E}=\widehat h_{j,k}^{\,\rm E}\cdot b_{j,k}^{\rm E},\quad
(\widehat{hb})_{j+1,k}^{\rm W}=\widehat h_{j+1,k}^{\,\rm W}\cdot b_{j+1,k}^{\rm W},
\end{equation*}
thus completing the computation of $\bm{\widehat W}_{j,k}^{\,\rm E}$ and $\bm{\widehat W}_{j+1,k}^{\,\rm W}$.

The modified values $\bm{\widehat W}_{j,k}^{\,\rm N}$ and $\bm{\widehat W}_{j,k+1}^{\,\rm S}$ in the $y$-direction are computed in a similar
manner; we omit the details for the sake of brevity.

\subsubsection{Well-Balanced Evaluation of the Global Fluxes}\label{sec322}
In order to compute the numerical fluxes \eref{3.17}, we first use the cell interface point values computed in \S\ref{sec321} to evaluate
the discrete values of the global fluxes $\bm K^x(\bm U)$ and $\bm K^y(\bm U)$ appearing in \eref{3.14}:
\begin{equation}
\begin{aligned}
&(\bm K^x)_{j,k}^{\rm E,W}:=\bm K^x\big(\bm U_{j,k}^{\rm E,W}\big)=\bm F^x\big(\bm U_{j,k}^{\rm E,W}\big)-(\bm R^x)_{j,k}^{\rm E,W},\\
&(\bm K^y)_{j,k}^{\rm N,S}:=\bm K^y\big(\bm U_{j,k}^{\rm N,S}\big)=\bm F^y\big(\bm U_{j,k}^{\rm N,S}\big)-(\bm R^y)_{j,k}^{\rm N,S},
\end{aligned}
\label{3.39}
\end{equation}
where $(\bm R^x)_{j,k}^{\rm E,W}$ and $(\bm R^y)_{j,k}^{\rm N,S}$ denote the numerical approximations of the integral terms $\bm R^x$ and
$\bm R^y$ appearing in \eref{3.8}. In order to ensure the resulting method is WB and since the integrals in \eref{3.8} contain the
nonconservative products, we use the path-conservative technique to evaluate $(\bm R^x)_{j,k}^{\rm E,W}$ and $(\bm R^y)_{j,k}^{\rm N,S}$
using the following recursive formulae for all $1\le j\le N_x$ and $1\le k\le N_y$:
\begin{equation}
\hspace*{-0.1cm}\begin{aligned}
&(\bm R^x)_{0,k}^{\rm E}=\bm0,&&\hspace*{-0.15cm}(\bm R^x)_{1,k}^{\rm W}=\bm Q^x_{\bm\Psi,\hf,k},&&\hspace*{-0.15cm}(\bm R^x)_{j,k}^{\rm E}
=(\bm R^x)_{j,k}^{\rm W}+\bm Q^x_{j,k},&&\hspace*{-0.15cm}(\bm R^x)_{j+1,k}^{\rm W}=(\bm R^x)_{j,k}^{\rm E}+\bm Q^x_{\bm\Psi,\jph,k},\\
&(\bm R^y)_{j,0}^{\rm N}=\bm0,&&\hspace*{-0.15cm}(\bm R^y)_{j,1}^{\rm S}=\bm Q^y_{\bm\Psi,j,\hf},&&\hspace*{-0.15cm}(\bm R^y)_{j,k}^{\rm N}
=(\bm R^y)_{j,k}^{\rm S}+\bm Q^y_{j,k},&&\hspace*{-0.15cm}(\bm R^y)_{j,k+1}^{\rm S}=(\bm R^y)_{j,k}^{\rm N}+\bm Q^y_{\bm\Psi,j,\kph},
\end{aligned}
\label{3.40}
\end{equation}
where $\bm Q^x_{j,k}$, $\bm Q^x_{\bm\Psi,\jph,k}$, $\bm Q^y_{j,k}$, and $\bm Q^y_{\bm\Psi,j,\kph}$ are found using appropriate quadratures
for
\begin{equation*}
\begin{aligned}
&\bm Q^x_{j,k}\approx\int\limits_{x_\jmh}^{x_\jph}\left[Q^x(\bm U)\bm U_x+\bm S^x(\bm U)\right]{\rm d}x,&&
\bm Q^x_{\bm\Psi,\jph,k}\approx\int\limits_0^1Q^x(\bm\Psi_{\jph,k}(s))\bm\Psi'_{\jph,k}(s)\,{\rm d}s,\\
&\bm Q^y_{j,k}\approx\int\limits_{y_\kmh}^{y_\kph}\left[Q^y(\bm U)\bm U_y+\bm S^y(\bm U)\right]{\rm d}y,&&
\bm Q^y_{\bm\Psi,j,\kph}\approx\int\limits_0^1Q^y(\bm\Psi_{j,\kph}(s))\bm\Psi'_{j,\kph}(s)\,{\rm d}s.
\end{aligned}
\end{equation*}
Here, $\bm\Psi_{\jph,k}(s):=\bm\Psi(s;\bm U_{j,k}^{\rm E},\bm U_{j+1,k}^{\rm W})$ is a certain path connecting the interface states
$\bm U_{j,k}^{\rm E}$ and $\bm U_{j+1,k}^{\rm W}$. In order to ensure the WB property, we use a line segment connecting the equilibrium
variables,
\begin{equation*}
\bm E_{\jph,k}(s):=(\bm E^x)_{j,k}^{\rm E}+s\left[(\bm E^x)_{j+1,k}^{\rm W}-(\bm E^x)_{j,k}^{\rm E}\right],\quad s\in[0,1],
\end{equation*}
which  corresponds to a particular path $\bm\Psi_{\jph,k}(s)$, whose detailed structure is only given implicitly. Similarly, the path
$\bm\Psi_{j,\kph}(s)$ is implicitly defined using the line segment 
\begin{equation*}
\bm E_{j,\kph}(s):=(\bm E^y)_{j,k}^{\rm N}+s\left[(\bm E^y)_{j,k+1}^{\rm S}-(\bm E^y)_{j,k}^{\rm N}\right],\quad s\in[0,1].
\end{equation*}
Following the technique introduced in \cite{Kurganov2023Well}, we compute
\begin{equation}
\begin{aligned}
&\bm Q^x_{j,k}\approx\bm F^x\big(\bm U_{j,k}^{\rm E}\big)-\bm F^x\big(\bm U_{j,k}^{\rm W}\big)-\hf\left[M^x\big(\bm U_{j,k}^{\rm E}\big)+
M^x\big(\bm U_{j,k}^{\rm W}\big)\right]\left((\bm E^x)_{j,k}^{\rm E}-(\bm E^x)_{j,k}^{\rm W}\right),\\
&\bm Q^y_{j,k}\approx\bm F^y\big(\bm U_{j,k}^{\rm N}\big)-\bm F^y\big(\bm U_{j,k}^{\rm S}\big)-\hf\left[M^y\big(\bm U_{j,k}^{\rm N}\big)+
M^y\big(\bm U_{j,k}^{\rm S}\big)\right]\left((\bm E^y)_{j,k}^{\rm N}-(\bm E^y)_{j,k}^{\rm S}\right),\\
&\bm Q^x_{\Psi,\jph,k}\approx\bm F^x\big(\bm U_{j+1,k}^{\rm W}\big)-\bm F^x\big(\bm U_{j,k}^{\rm E}\big)-
\hf\left[M^x\big(\bm U_{j+1,k}^{\rm W}\big)+M^x\big(\bm U_{j,k}^{\rm E}\big)\right]
\left((\bm E^x)_{j+1,k}^{\rm W}-(\bm E^x)_{j,k}^{\rm E}\right),\\
&\bm Q^y_{\Psi,j,\kph}\approx\bm F^y\big(\bm U_{j,k+1}^{\rm S}\big)-\bm F^y\big(\bm U_{j,k}^{\rm N}\big)-
\hf\left[M^y\big(\bm U_{j,k+1}^{\rm S}\big)+M^y\big(\bm U_{j,k}^{\rm N}\big)\right]
\left((\bm E^y)_{j,k+1}^{\rm S}-(\bm E^y)_{j,k}^{\rm N}\right),
\end{aligned}
\label{3.46}
\end{equation}
where $\bm F^x(\bm U)$ and $\bm F^y(\bm U)$ are defined in \eref{3.5}, and $M^x(\bm U)$ and $M^y(\bm U)$ are defined in \eref{3.15}.

\subsubsection{Well-Balanced Property}\label{sec323}
We now prove that the proposed 2-D scheme exactly preserves the families of quasi 1-D steady-states defined in \eref{3.12} and \eref{3.13}.
\begin{thm}\label{thm3.4}
The 2-D semi-discrete flux globalization based WB PCCU scheme is WB in the sense that it is capable of exactly preserving the families of
quasi 1-D steady states in \eref{3.12} and \eref{3.13}.
\end{thm}
\begin{proof}
Since the proofs of preserving the quasi 1-D steady states in \eref{3.12} and \eref{3.13} are very similar, we only show that the proposed
scheme exactly preserves \eref{3.12}.

Assume that at a certain time level, we have a discrete steady-state solution that satisfies \eref{3.12}, namely:
\begin{equation}
\begin{aligned}
&(hu)_{j,k}^{\rm E,W}=(\xbar{hu})_{j,k}\equiv\big((hu)_{\rm eq}\big)_k,\quad
(ha)_{j,k}^{\rm E,W}=(\xbar{ha})_{j,k}\equiv\big((ha)_{\rm eq}\big)_k,\quad v_{j,k}^{\rm E,W,N,S}=v_{j,k}\equiv0,\\
&b_{j,k}^{\rm E,W,N,S}=b_{j,k}\equiv0,\quad(E^x)_{j,k}^{\rm E,W}=E^x_{j,k}\equiv(E^x_{\rm eq})_k,\quad
(E^y)_{j,k}^{\rm N,S}=E^y_{j,k}\equiv(E^y_{\rm eq})_j,
\end{aligned}
\label{3.29a}
\end{equation}
where the quantities $\big((hu)_{\rm eq}\big)_k$, $\big((ha)_{\rm eq}\big)_k$, and $(E^x_{\rm eq})_k$ are all constants along the
$x$-direction, and $(E^y_{\rm eq})_j$ is constant along the $y$-direction. In order to show that the discrete steady states \eref{3.29a} are
preserved exactly, we need to show that
\begin{equation*}
(\bm K^x)_{j+1,k}^{\rm W}=(\bm K^x)_{j,k}^{\rm E}=(\bm K^x)_{j,k}^{\rm W}\quad\mbox{and}\quad 
(\bm K^y)_{j,k+1}^{\rm S}=(\bm K^y)_{j,k}^{\rm N}=(\bm K^y)_{j,k}^{\rm S}
\end{equation*}
for all $j,k$. To this end, we first use \eref{3.39} to obtain that the equalities $(\bm K^x)_{j+1,k}^{\rm W}=(\bm K^x)_{j,k}^{\rm E}$ and
$(\bm K^y)_{j,k+1}^{\rm S}=(\bm K^y)_{j,k}^{\rm N}$ are equivalent to
\begin{equation*}
\begin{aligned}
\bm F^x\big(\bm U_{j+1,k}^{\rm W}\big)-(\bm R^x)_{j+1,k}^{\rm W}=\bm F^x\big(\bm U_{j,k}^{\rm E}\big)-(\bm R^x)_{j,k}^{\rm E}~\,\mbox{and}~
\bm F^y\big(\bm U_{j,k+1}^{\rm S}\big)-(\bm R^y)_{j,k+1}^{\rm S}=\bm F^y\big(\bm U_{j,k}^{\rm N}\big)-(\bm R^y)_{j,k}^{\rm N},
\end{aligned}
\end{equation*}
respectively. The last two equalities can be rewritten using the definition of $(\bm R^x)_{j,k}^{\rm E,W}$ and $(\bm R^y)_{j,k}^{\rm N,S}$
in \eref{3.40}, \eref{3.46} as
\begin{align*}
&\bm F^x\big(\bm U_{j+1,k}^{\rm W}\big)-\bm F^x\big(\bm U_{j,k}^{\rm E}\big)-\bm Q^x_{\bm\Psi,\jph,k}=
\hf\left[M^x(\bm U_{j+1,k}^{\rm W})+M^x(\bm U_{j,k}^{\rm E})\right]\left((\bm E^x)_{j+1,k}^{\rm W}-(\bm E^{x})_{j,k}^{\rm E}\right)=\bm0,\\
&\bm F^y\big(\bm U_{j,k+1}^{\rm S}\big)-\bm F^y\big(\bm U_{j,k}^{\rm N}\big)-\bm Q^y_{\bm\Psi,j,\kph}=
\hf\left[M^y(\bm U_{j,k+1}^{\rm S})+M^y(\bm U_{j,k}^{\rm N})\right]\left((\bm E^y)_{j,k+1}^{\rm S}-(\bm E^{y})_{j,k}^{\rm N}\right)=\bm0,
\end{align*}
which hold due to the assumptions in \eref{3.29a}. Then, one can similarly proceed to obtain
\begin{align*}
&\bm F^x\big(\bm U_{j,k}^{\rm E}\big)-\bm F^x\big(\bm U_{j,k}^{\rm W}\big)-\bm Q^x_{j,k}=
\hf\left[M^x(\bm U_{j,k}^{\rm E})+M^x(\bm U_{j,k}^{\rm W})\right]\left((\bm E^x)_{j,k}^{\rm E}-(\bm E^x)_{j,k}^{\rm W}\right)=\bm0,\\
&\bm F^y\big(\bm U_{j,k}^{\rm N}\big)-\bm F^y\big(\bm U_{j,k}^{\rm S}\big)-\bm Q^y_{j,k}=
\hf\left[M^y(\bm U_{j,k}^{\rm N})+M^y(\bm U_{j,k}^{\rm S})\right]\left((\bm E^y)_{j,k}^{\rm N}-(\bm E^y)_{j,k}^{\rm S}\right)=\bm0,
\end{align*}
which are also true due to the assumptions in \eref{3.29a}. This concludes the proof of the theorem.
\end{proof}

\section{Numerical Examples}\label{sec4}
In this section, we demonstrate the performance of the proposed flux globalization based WB PCCU schemes in a number of numerical
experiments. In all of the examples, we set the minmod parameter $\Theta=1.3$, take $g=1$, and evolve the solution in time using the
explicit three-stage third-order SSP RK method (see, e.g., \cite{Gottlieb2001Strong, Gottlieb2011Strong}) with the variable time step 
selected using the CFL number 0.25.

In several examples, we compare the performance of the proposed WB schemes with ``non-well-balanced'' (NWB) ones, which are obtained by
replacing the WB generalized minmod reconstruction of the equilibrium variables with the same generalized minmod reconstruction but applied
to the conservative variables. Below, we refer to the studied schemes as the WB and NWB schemes.

\subsection{1-D Numerical Examples}\label{sec41}
\subsubsection{General Facts About the 1-D MRSW Model}\label{sec411}
We first recall some important facts on the properties of the system \eref{2.1}, which are necessary for understanding and interpreting
the obtained numerical results. The first point is that, as already mentioned, resolving the zero-divergence constraint for the magnetic
field is straightforward in the 1-D case, see \eref{2.2} where the constant has a meaning of the mean meridional magnetic field
$\mathfrak{B}$ multiplied by the mean thickness $\mathfrak{H}$. This allows to eliminate the dependent variable $b$ in favor of the variable
$h$ and thus reduces the model to a system of four equations for the variables $u$, $v$, $h$, and $a$. This system is equivalent to a
quasi-linear hyperbolic system for these variables. It can be shown to possess four characteristics: two of them corresponding to
magneto-inertia-gravity waves propagating in the positive and negative directions along the $y$-axis, and another two corresponding to
rotation-modified Alfv\'en waves propagating in the positive and negative directions as well. One can straightforwardly linearize the
resulting 4 by 4 system to see that harmonic waves arise as solutions: high-frequency (fast) magneto-inertia-gravity waves and
low-frequency (slow) rotation-modified Alfv\'en waves; see, e.g., \cite{zeitlin2015geostrophic}. Therefore, any localized perturbation of
the steady state is a source of these waves, propagating out of it towards the domain boundaries. Note that in a particular case
$\mathfrak{B}=0$, the system becomes a 1-D RSW equation with a passively advected field $a$, and the Alfv\'en waves, which can propagate
only on the background of a magnetic field, disappear.

The second point is that the conservation laws of the system \eref{2.1} can be combined to give another one, the energy conservation
(which is obviously valid for smooth solutions of \eref{2.1} only):
\begin{equation}
\hspace*{-0.35cm}\left[h\frac{u^2+v^2+a^2+b^2}{2}+gh\Big(\frac{h}{2}+Z\Big)\right]_t+\left[hv\frac{u^2+v^2+a^2+b^2}{2}+g(h+Z)-hb(au+bv)\right]_y=0
\label{energy1}
\end{equation} 
In fact, the system \eref{1.1} is Hamiltonian (see, e.g., \cite{Dellar2002Hamiltonian}), with the Hamiltonian density given by the
expression in the brackets in the first term on the left-hand side (LHS) of \eref{energy1}. As is well-known, stationary solutions of the
Hamiltonian systems, which are given by \eref{2.11}, \eref{2.16} in the present case, are local minima of the Hamiltonian (energy). So, if
the system is close to one of them, it engages in a relaxation (adjustment) process consisting of shedding an excess of energy and arriving
at the minimum of energy. In non-dissipative systems, wave emission is the only way to evacuate energy (although numerical dissipation
adds up in simulations).

The third point is that among the aforementioned steady states, there are those of particular importance in geophysical and
astrophysical applications---the so-called geostrophic equilibria, that is, the equilibria between the Coriolis and the pressure forces in
the absence of a magnetic field:
\begin{equation}
fu=-gh_y,~~v=0,~~a=0,~~b=0.
\label{geostrophy}
\end{equation}
One can show (see, e.g., \cite{zeitlin2018geophysical}) that for the RSW system, any nontrivial initial condition will evolve towards the
state of the corresponding geostrophic equilibrium by emission of inertia-gravity waves: this is a geostrophic adjustment process. In the
presence of a meridional magnetic field $b\ne0$, the first equation in \eref{geostrophy} is modified to
\begin{equation}
fu=-gh_y+bb_y,
\label{4.5}
\end{equation}
which corresponds to magneto-geostrophic equilibrium between the Coriolis, pressure, and magnetic pressure forces. The magneto-geostrophic
the adjustment process is, however, more complicated than the geostrophic one as the zonal component of the magnetic field $a$ is not constant
at the steady state; see \eref{2.16}.

We would like to recall that the dynamical regimes close to (magneto-)geostrophic equilibrium are characterized by small Rossby and magnetic
Rossby numbers, $Ro$ and $Ro_m$ defined as
$$
Ro=\frac{\mathfrak U}{f\mathfrak L},\quad Ro_m=\frac{\mathfrak B}{f\mathfrak L},
$$
where $\mathfrak U$ and $\mathfrak B$ are typical values of velocity and magnetic field, respectively, and $\mathfrak L$ is a typical scale
of the motions under study. As was shown in \cite{zeitlin2015geostrophic}, at small $Ro\sim Ro_m$, the magneto-geostrophic adjustment of a
localized initial perturbation consists in a rapid evacuation of fast magneto-inertia-gravity waves out of the perturbation location and
slow emission of rotation-modified Alfv\'en waves. Due to the fact that the group velocities of both types of waves tend to zero with
increasing wavelength, a part of initial perturbation persists for a long time, subject to inertial oscillations, like in the case of the
standard RSW equations \cite{zeitlin2003rsw}.

\subsubsection{Details of Numerical Implementation}
In Examples 1 and 2, we use outflow boundary conditions. This is achieved by using an extrapolation of the equilibrium variables. For those
of them that are constant at steady states ($hv$, $E$, and $hb$), we use the zero-order extrapolation, that is, we set
\begin{equation*}
(\xbar{hv})_0:=(\xbar{hv})_1,~E_0:=E_1,~(\xbar{hb})_0:=(\xbar{hb})_1,~(\xbar{hv})_{N+1}:=(\xbar{hv})_N,~E_{N+1}:=E_N,~
(\xbar{hb})_{N+1}:=(\xbar{hb})_N.
\end{equation*}
For $u$ and $a$, whose profiles at steady states are either linear or quadratic, we use \eref{2.16} to obtain
\begin{equation*}
\begin{aligned}
&u_0:=u(y_0)=\frac{(\xbar{hv})_0^2}{(\xbar{hv})_0^2-(\xbar{hb})_0^2}\left(f_cy_0+\frac{\beta}{2}y_0^2\right)+u_c,\\
&a_0:=a(y_0)=\frac{(\xbar{hv})_0(\xbar{hb})_0}{\xbar{(hv})_0^2-(\xbar{hb})_0^2}\left(f_cy_0+\frac{\beta}{2}y_0^2\right)+a_c,\\
&u_{N+1}:=u(y_{N+1})=\frac{(\xbar{hv})_{N+1}^2}{(\xbar{hv})_{N+1}^2-(\xbar{hb})_{N+1}^2}\left(f_cy_{N+1}+\frac{\beta}{2}y_{N+1}^2\right)+
u_c,\\
&a_{N+1}:=a(y_{N+1})=\frac{(\xbar{hv})_{N+1}(\xbar{hb})_{N+1}}{(\xbar{hv})_{N+1}^2-(\xbar{hb})_{N+1}^2}\left(f_cy_{N+1}+
\frac{\beta}{2}y_{N+1}^2\right)+a_c.
\end{aligned}
\end{equation*}
We then use \eref{2.19f} and obtain the corresponding values of $h$ by solving the following cubic equations:
$$
\frac{(\xbar{hv})_0^2-(\xbar{hb})_0^2}{2\,\xbar h_0^{\,2}}+g\left(\xbar h_0+Z_0\right)+P_0=E_0,~~
\frac{(\xbar{hv})_{N+1}^2-(\xbar{hb})_{N+1}^2}{2\,\xbar h_{N+1}^{\,2}}+g\left(\xbar h_{N+1}+Z_{N+1}\right)+P_{N+1}=E_{N+1}
$$
for $h_0$ and $h_{N+1}$, respectively. Here, $P_0$ and $P_{N+1}$ are obtained using a straightforward boundary extension of \eref{2.20f},
namely:
$$
P_0=P_1-\frac{\dy}{2}\big(f_0u_0+f_1u_1\big),\quad P_{N+1}=P_N+\frac{\dy}{2}\big(f_Nu_N+f_{N+1}u_{N+1}\big).
$$
Lastly, the boundary conditions for $hu$ and $ha$ are
$$
(\xbar{hu})_0=\,\xbar h_0u_0,\quad(\xbar{ha})_0=\,\xbar h_0a_0,\quad(\xbar{hu})_{N+1}=\,\xbar h_{N+1}u_{N+1},\quad
(\xbar{ha})_{N+1}=\,\xbar h_{N+1}a_{N+1}.
$$

\subsubsection*{Example 1---Steady-State with Constant Coriolis Parameter ($f(y)\equiv1$)}
In this example, we first demonstrate that the proposed 1-D method exactly preserves moving-water equilibria. To this end, we use the
following initial conditions that satisfy \eref{2.11}, \eref{2.16}:
\begin{equation*}
\begin{aligned}
&(hv)(y,0)=(hv)_{\rm eq}(y)\equiv0.5,\quad E(y,0)=E_{\rm eq}(y)\equiv1,\quad(hb)(y,0)=(hb)_{\rm eq}(y)\equiv3,\\
&u(y,0)=u_{\rm eq}(y)=-\frac{1}{35}y+0.3,\quad a(y,0)=a_{\rm eq}(y)=-\frac{6}{35}y+2,
\end{aligned}
\end{equation*}
and the bottom topography $Z(y)=\hf e^{-y^2}$.

However, the $h_{\rm eq}(y)$ profile can only be computed on the discrete level. Therefore, we take the computational domain $[-10,10]$,
set the outflow boundary conditions, and compute the discrete cell averages $(\xbar{h_{\rm eq}})_k$ by solving the following cubic equations
(see \eref{2.19f}):
\begin{equation}
-\frac{35}{8\big((\xbar{h_{\rm eq}}\big)_k)^2}+g\left((\xbar{h_{\rm eq}})_k+Z(y_k)\right)+P_k=1,
\label{4.3f}
\end{equation}
where $P_k$ are computed using \eref{2.20f} with $u_k=u_{\rm eq}(y_k)$. The obtained discrete profile of $h_{\rm eq}(y)$ is plotted in
Figure \ref{fig41} (left).
\begin{figure}[ht!]
\centerline{\includegraphics[trim=0.6cm 0.6cm 1.0cm 0.9cm, clip, width=5.0cm]{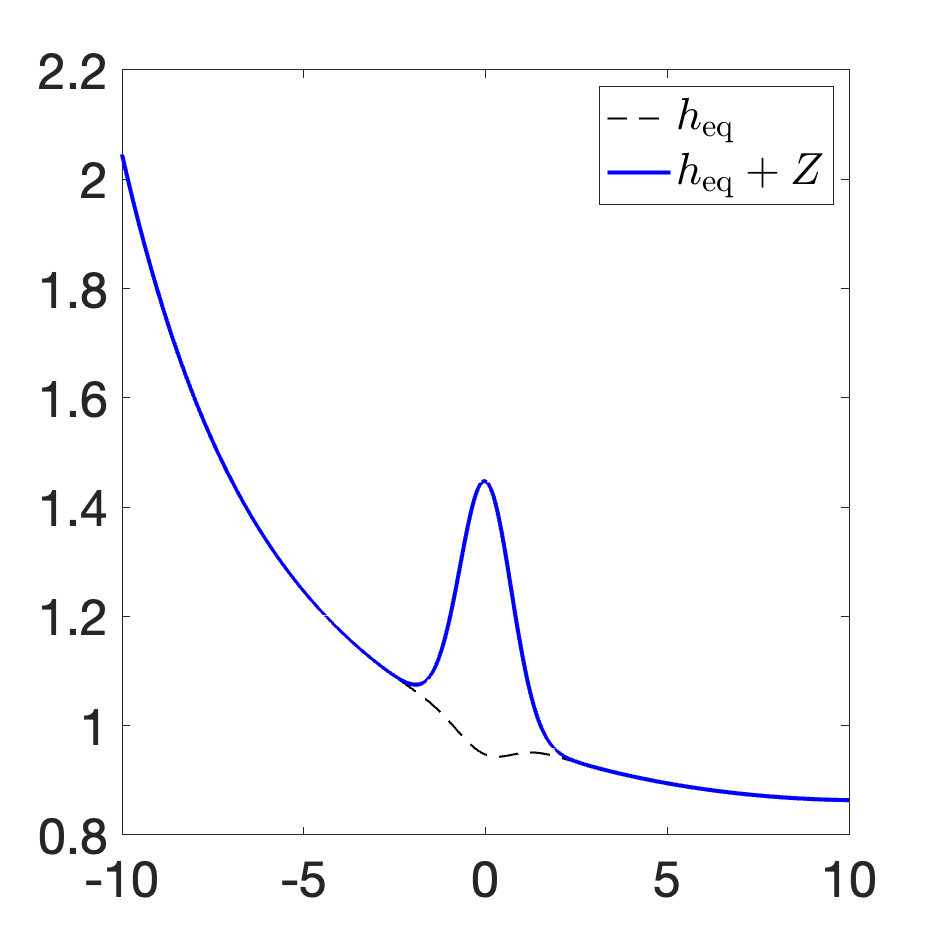}\hspace*{1cm}
            \includegraphics[trim=0.6cm 0.6cm 1.0cm 0.9cm, clip, width=5.0cm]{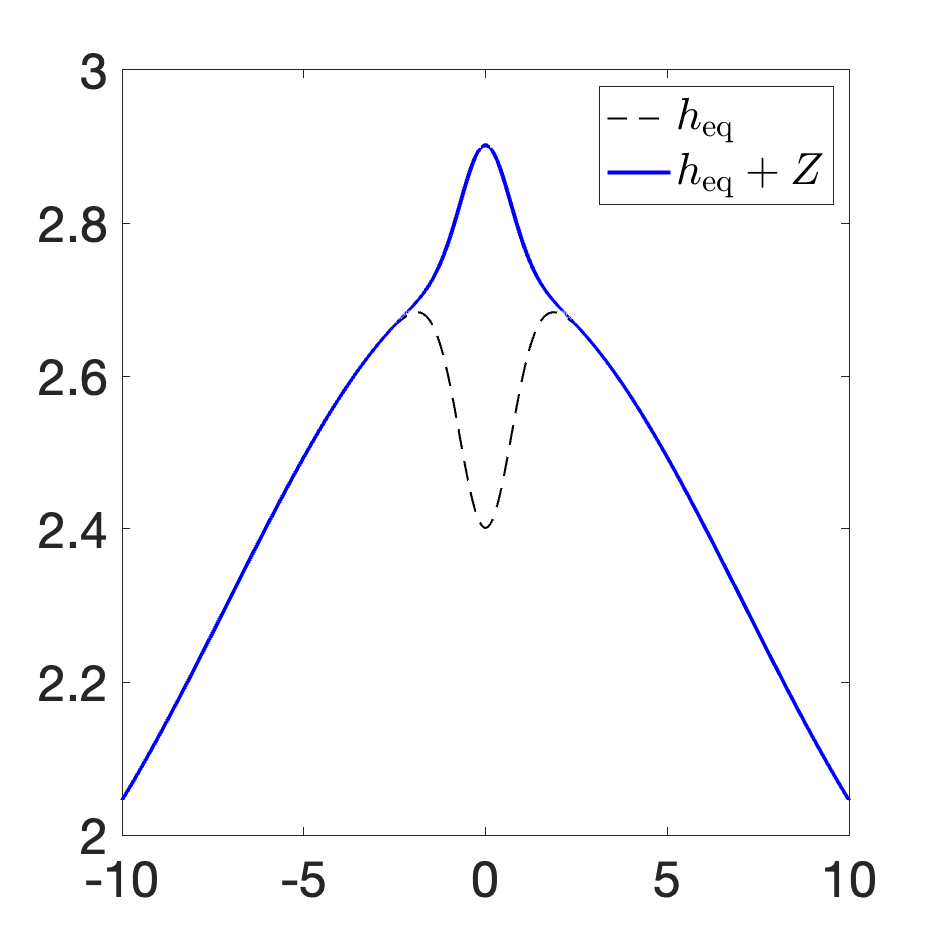}}
\caption{\sf The steady-state fluid depths $h_{\rm eq}(y)$ and fluid levels $h_{\rm eq}(y)+Z(y)$ in Examples 1 (left) and 2 (right).
\label{fig41}}
\end{figure}

We compute the numerical solutions on a uniform mesh with $N=100$ by the WB and NWB schemes until the final time $t=5$. The results
reported in Table \ref{tab41}, show that the WB scheme, as expected, preserves the steady state within the machine accuracy, while the NWB
scheme fails to do so.
\begin{table}[ht!]
\begin{center}
\begin{tabular}{|c|c|c|c|c|}
\hline
Scheme&$\|h(\cdot,5)-h_{\rm eq}\|_\infty$&$\|u(\cdot,5)-u_{\rm eq}\|_\infty$&$\|v(\cdot,5)-v_{\rm eq}\|_\infty$&
$\|a(\cdot,5)-a_{\rm eq}\|_\infty$\\ \hline
WB &1.33e-15&3.22e-15&7.55e-15&4.44e-15\\
NWB&2.18e-03&1.86e-03&1.40e-03&3.97e-03\\
\hline
\end{tabular}
\end{center}
\caption{\sf Example 1 (capturing the steady state): Errors for the WB and NWB schemes.\label{tab41}}
\end{table}

Next, we examine the ability of the proposed WB scheme to correctly capture the evolution of a small perturbation of the studied steady
state. This is done by perturbing the discrete equilibrium fluid depth. Namely, we take the following initial data for $h$:
\begin{equation*}
\xbar h_k=(\xbar{h_{\rm eq}})_k+\begin{cases}10^{-3}&\mbox{if }|y_k+2|<\frac{1}{4},\\0&\textrm{otherwise}.\end{cases}
\end{equation*}

We compute the numerical solutions by both the WB and NWB schemes until the final time $t=1$ on a sequence of meshes with $N=100$, 1000, and
10000 cells. The obtained differences $h(y,1)-h_{\rm eq}(y)$ are plotted in Figure \ref{fig42}, where one can see that the NWB scheme fails
to capture the correct solution on a coarse mesh with $N=100$ and even when the mesh is refined the NWB solution contains visible
oscillations. At the same time, the WB solution is oscillation-free, even on a coarse mesh.
\begin{figure}[ht!]
\centerline{\includegraphics[trim=0.6cm 0.8cm 1.1cm 0.2cm, clip, width=5.7cm]{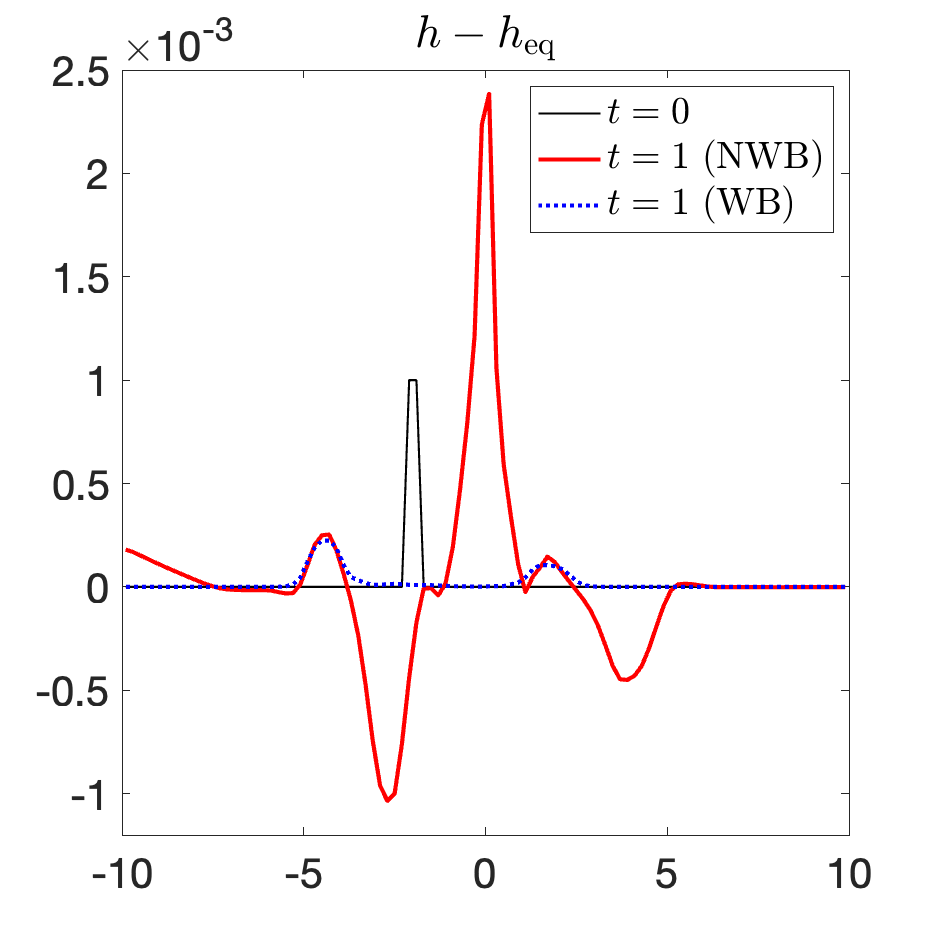}
            \includegraphics[trim=0.6cm 0.8cm 1.1cm 0.2cm, clip, width=5.7cm]{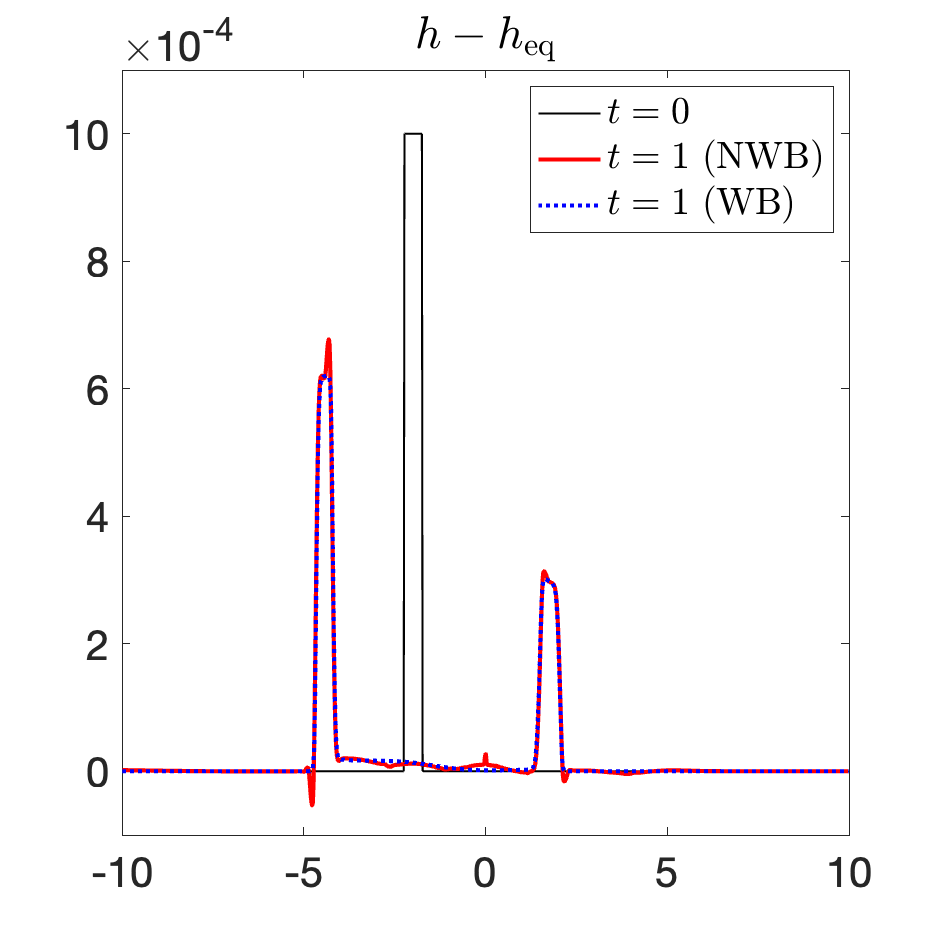}
            \includegraphics[trim=0.6cm 0.8cm 1.1cm 0.2cm, clip, width=5.7cm]{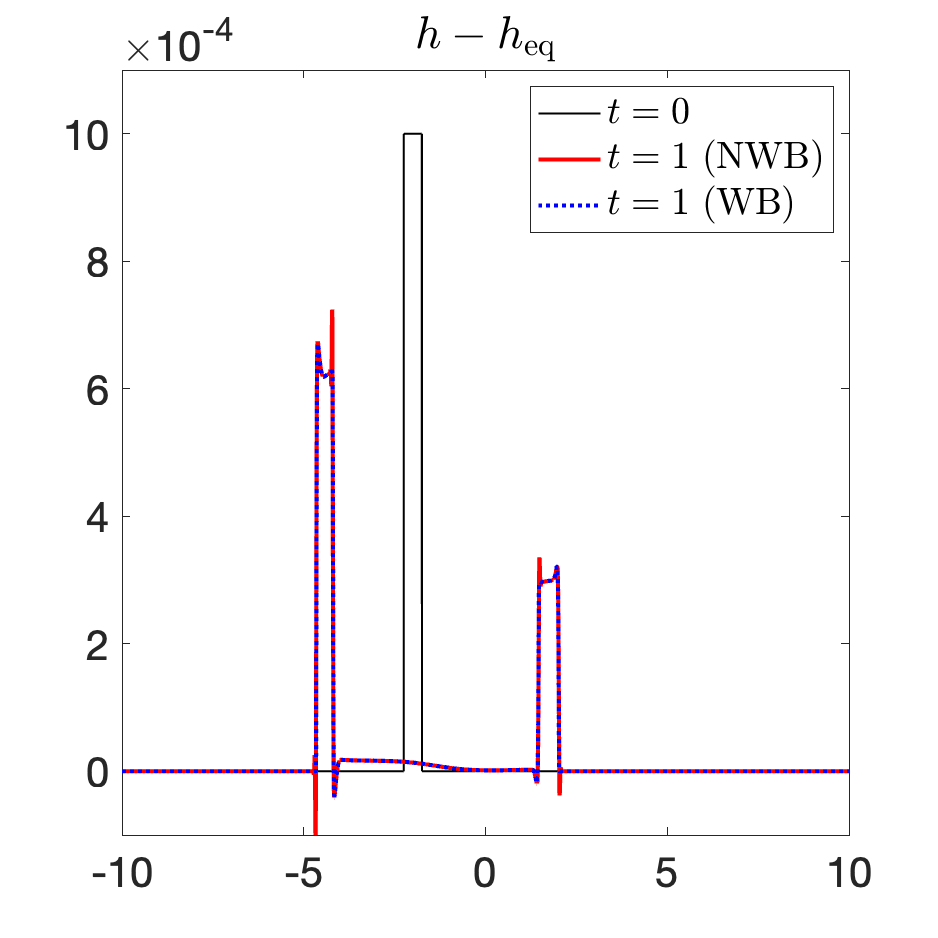}}
\caption{\sf Example 1 (small perturbation of the steady state): $h(y,1)-h_{\rm eq}(y)$ computed by both the WB and NWB schemes with $N=100$
(left), 1000 (middle) and 10000 (right) uniform cells. Notice the difference in vertical scales between the first and other panels in this
and the following figures.\label{fig42}}
\end{figure}

It is also instructive to see the computed differences $u(y,1)-u_{\rm eq}(y)$ and $a(y,1)-a_{\rm eq}(y)$; see Figure \ref{fig43}. They
clearly show that the initial perturbation results in two wave packets propagating to the right and the left of its location. The phase
relations between $u(y,1)-u_{\rm eq}(y)$ and $a(y,1)-a_{\rm eq}(y)$ in these waves (same phase for the left-moving, and opposite phases for
the right-moving waves) match the corresponding phase relations of linear waves on the background of the magnetic field, which can be
straightforwardly deduced from the linearized equations (see \cite{zeitlin2015geostrophic}) with the linearization being justified by the
smallness of the perturbation. We also see that, unlike the NWB scheme, the WB one captures the waves properly, even at the lowest resolution.
\begin{figure}[ht!]
\centerline{\includegraphics[trim=0.6cm 0.8cm 1.1cm 0.2cm, clip, width=5.7cm]{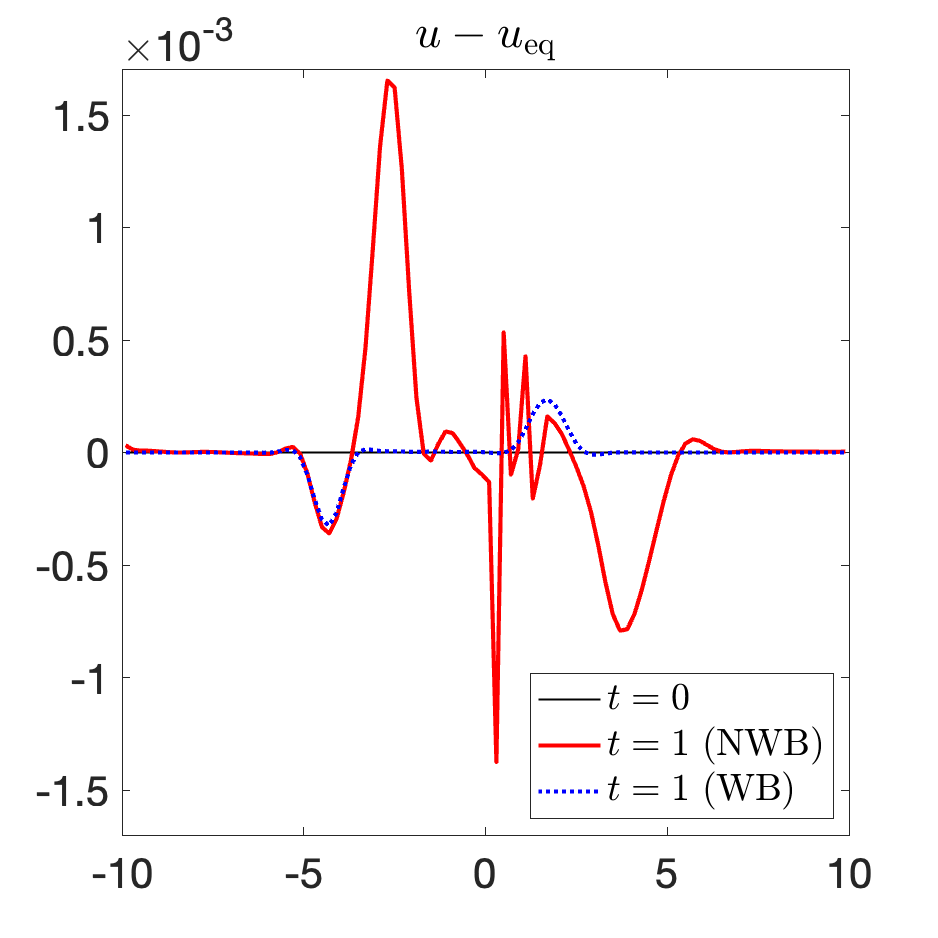}
            \includegraphics[trim=0.6cm 0.8cm 1.1cm 0.2cm, clip, width=5.7cm]{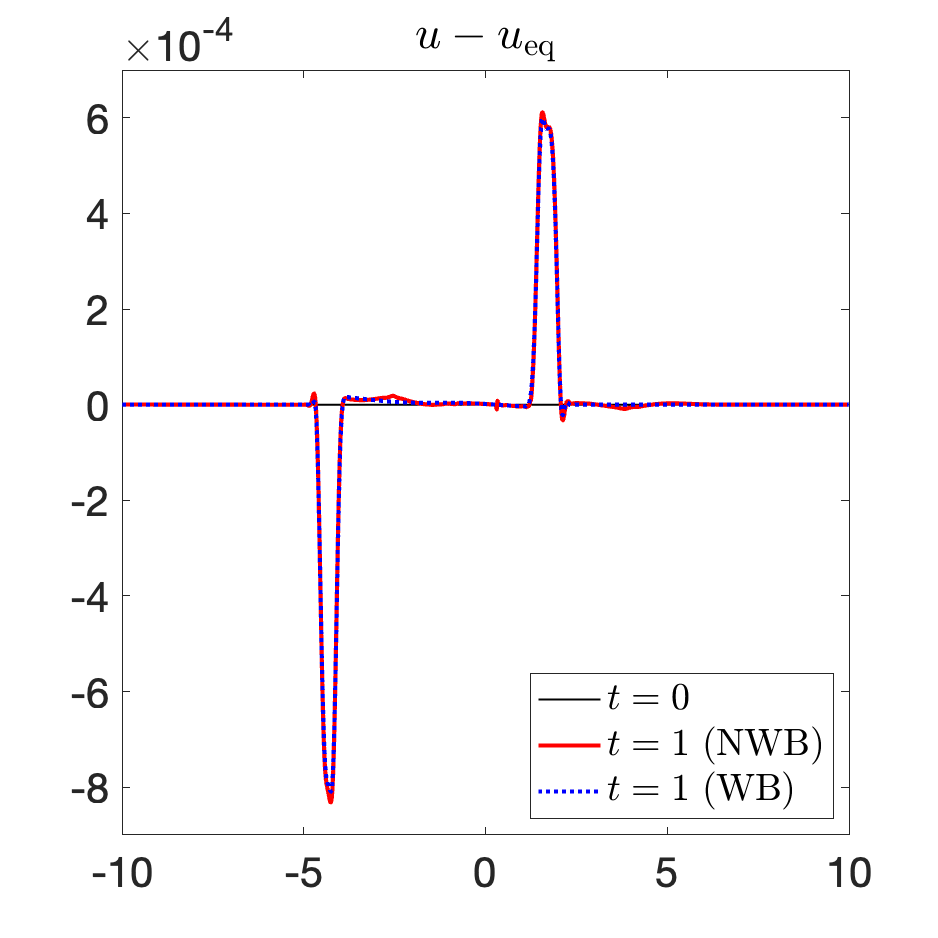}
            \includegraphics[trim=0.6cm 0.8cm 1.1cm 0.2cm, clip, width=5.7cm]{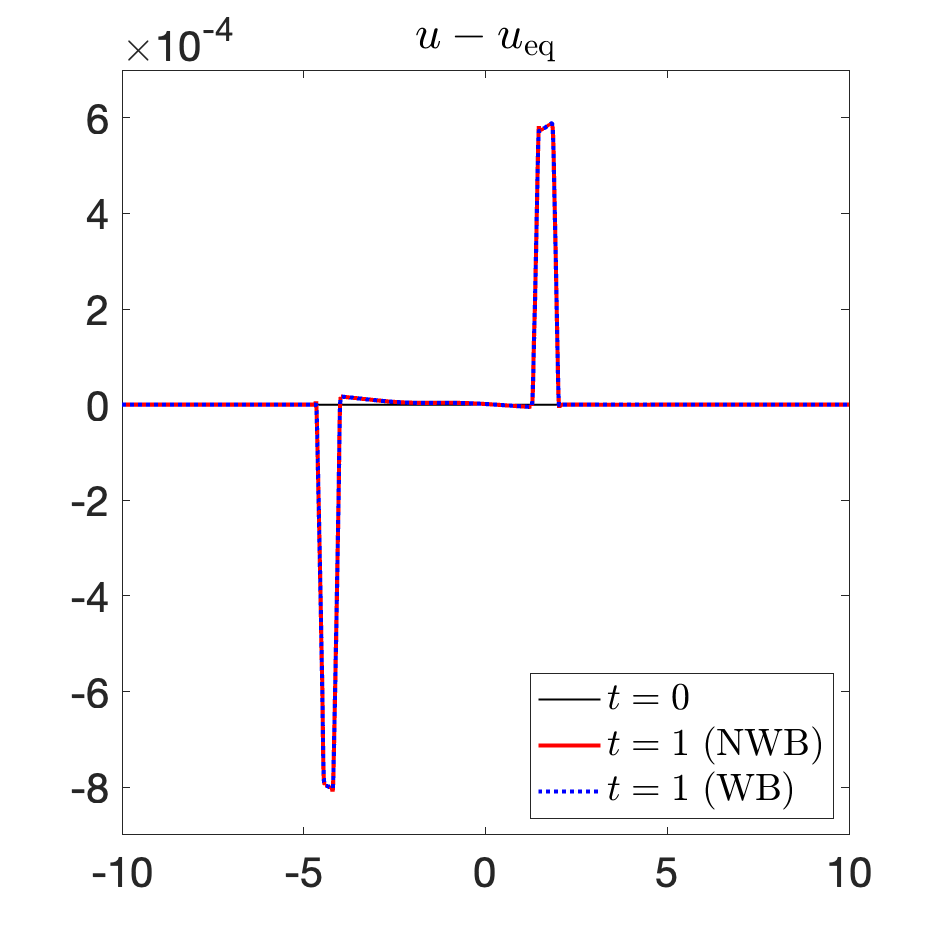}}
\vskip7pt
\centerline{\includegraphics[trim=0.6cm 0.8cm 1.1cm 0.2cm, clip, width=5.7cm]{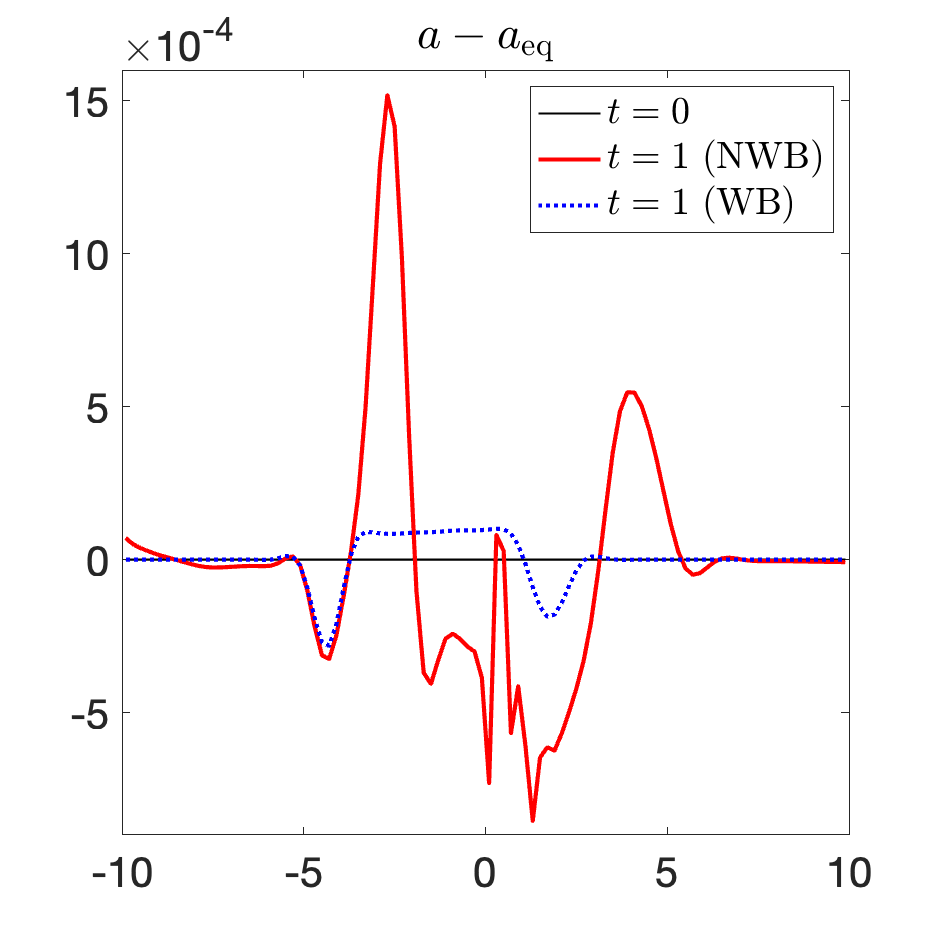}
            \includegraphics[trim=0.6cm 0.8cm 1.1cm 0.2cm, clip, width=5.7cm]{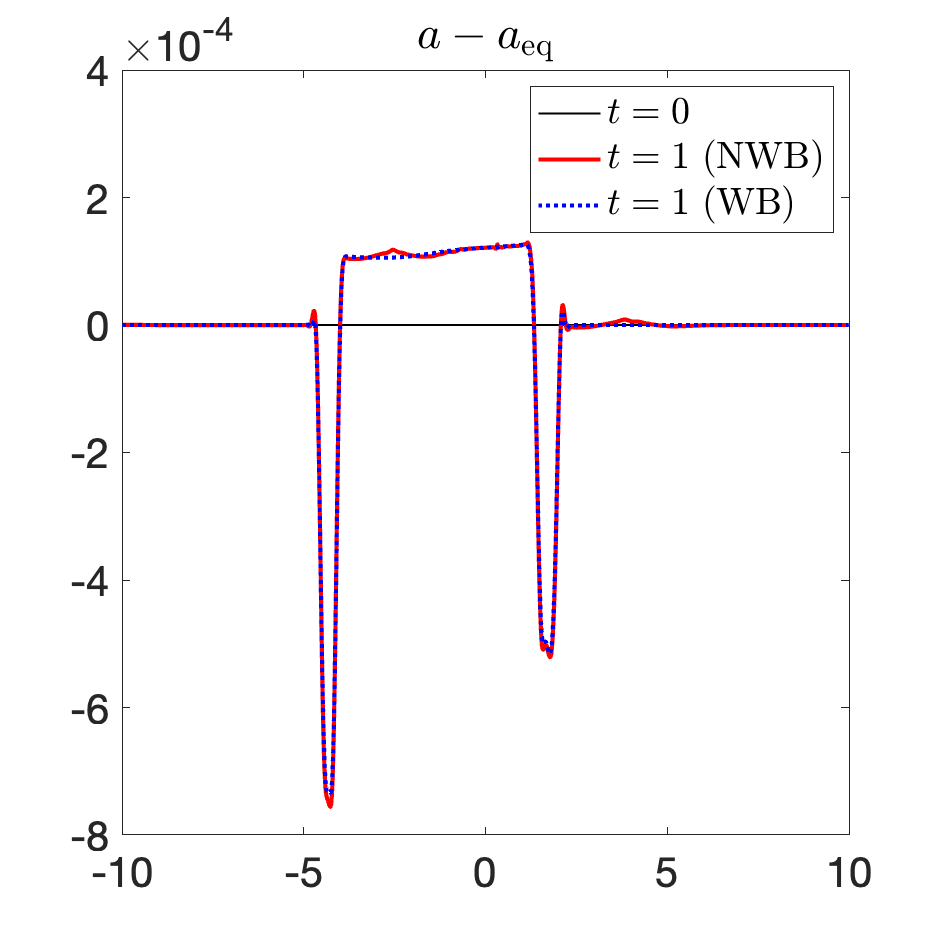}
            \includegraphics[trim=0.6cm 0.8cm 1.1cm 0.2cm, clip, width=5.7cm]{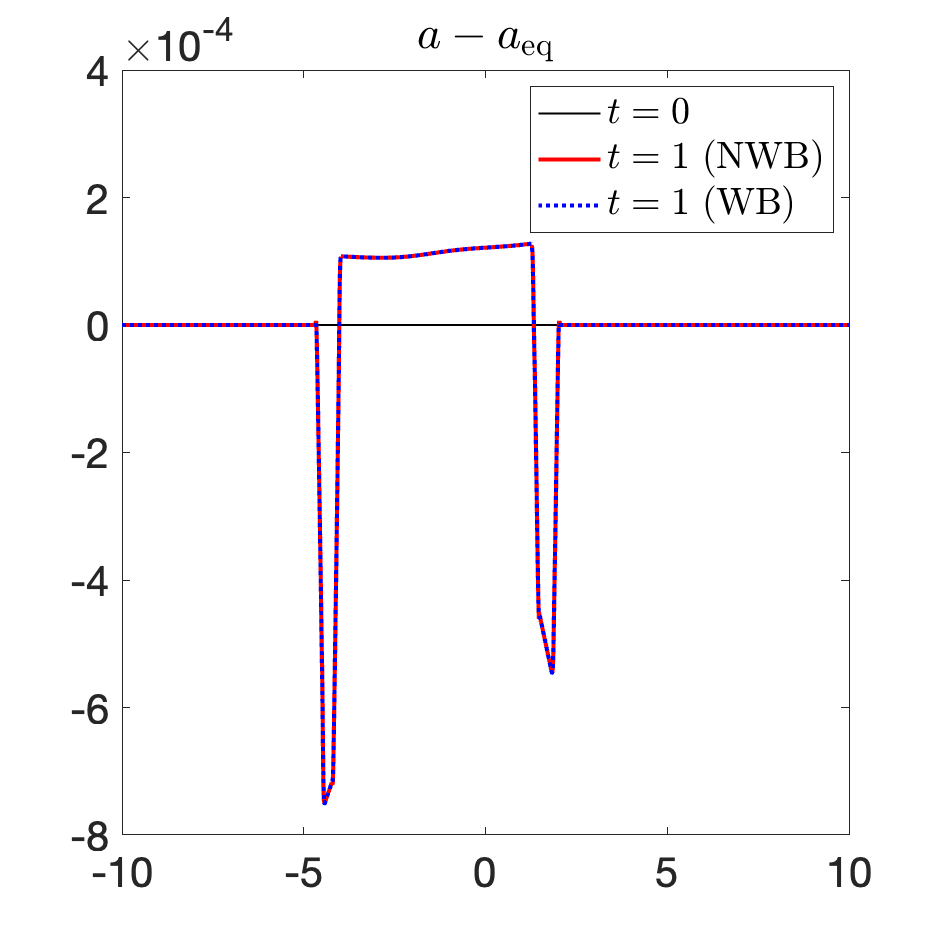}}
\caption{\sf Example 1 (small perturbation of the steady state): $u(y,1)-u_{\rm eq}(y)$ (top row) and $a(y,1)-a_{\rm eq}(y)$ (bottom row)
computed by both the WB and NWB schemes with $N=100$ (left), 1000 (middle), and 10000 (right) uniform cells.\label{fig43}}
\end{figure}

\subsubsection*{Example 2---Steady-State with Linear Coriolis Parameter $(f(y)=0.1y)$}
In the second example, we demonstrate that the proposed 1-D method exactly preserves moving-water equilibria in the so-called equatorial
beta-plane approximation of the Coriolis parameter $f$. (The axis of rotation is parallel to the tangent plane at the equator; thus, the
constant part of $f$ is identically zero.) To this end, we use the following initial conditions that satisfy \eref{2.11}, \eref{2.16}:
\begin{equation*}
\begin{aligned}
&(hv)(y,0)=(hv)_{\rm eq}(y)\equiv0.5,\quad E(y,0)=E_{\rm eq}(y)\equiv1,\quad(hb)(y,0)=(hb)_{\rm eq}(y)\equiv3,\\
&u(y,0)=u_{\rm eq}(y)=-\frac{1}{700}y^2+0.3,\quad a(y,0)=a_{\rm eq}(y)=-\frac{3}{350}y^2+2,
\end{aligned}
\end{equation*}
and the bottom topography $Z(y)=\hf e^{-y^2}$. The computational domain is $[-10,10]$ and the discrete profile of $h_{\rm eq}(y)$, which is
plotted in Figure \ref{fig41} (right), is obtained precisely as in Example 1 by solving the cubic equation \eref{4.3f}.

We compute the numerical solutions on a uniform mesh with $N=100$ by the WB and NWB schemes until the final time $t=5$. The results
reported in Table \ref{tab42}, show that the WB scheme, as expected, preserves the steady state within the machine accuracy, while the NWB
scheme fails to do so.
\begin{table}[ht!]
\begin{center}
\begin{tabular}{|c|c|c|c|c|}
\hline
Scheme&$\|h(\cdot,5)-h_{\rm eq}\|_\infty$&$\|u(\cdot,5)-u_{\rm eq}\|_\infty$&$\|v(\cdot,5)-v_{\rm eq}\|_\infty$&
$\|a(\cdot,5)-a_{\rm eq}\|_\infty$\\ \hline
WB &1.33e-15&2.11e-15&1.08e-15&1.55e-15\\
NWB&1.71e-03&1.79e-03&4.81e-03&1.36e-03\\
\hline
\end{tabular}
\end{center}
\caption{\sf Example 2 (capturing the steady state): Errors for the WB and NWB schemes.\label{tab42}}
\end{table}

Next, we examine the ability of the proposed WB scheme to capture a small perturbation of the studied steady state accurately. We use
precisely the same perturbed initial $h$ as in Example 1 and compute the numerical solutions by both the WB and NWB schemes until the final
time $t=1$ on a sequence of meshes with $N=100$, 1000, and 10000 cells. The obtained differences $h(y,1)-h_{\rm eq}(y)$,
$u(y,1)-u_{\rm eq}(y)$, and $a(y,1)-a_{\rm eq}(y)$ are plotted in Figure \ref{fig45}, where one can see that the NWB scheme fails to capture
the correct solution on a coarse mesh with $N=100$ and even when $N=1000$, the NWB solution is still very oscillatory. At the same time, the
WB solutions are oscillation-free. Notice that finding linear wave solutions with the meridional (poloidal) magnetic field
and with $f(y)\sim y$ is a nontrivial task, so we do not have here readily available theoretical predictions of the properties of such
waves.
\begin{figure}[ht!]
\centerline{\includegraphics[trim=0.6cm 0.8cm 1.1cm 0.2cm, clip, width=5.7cm]{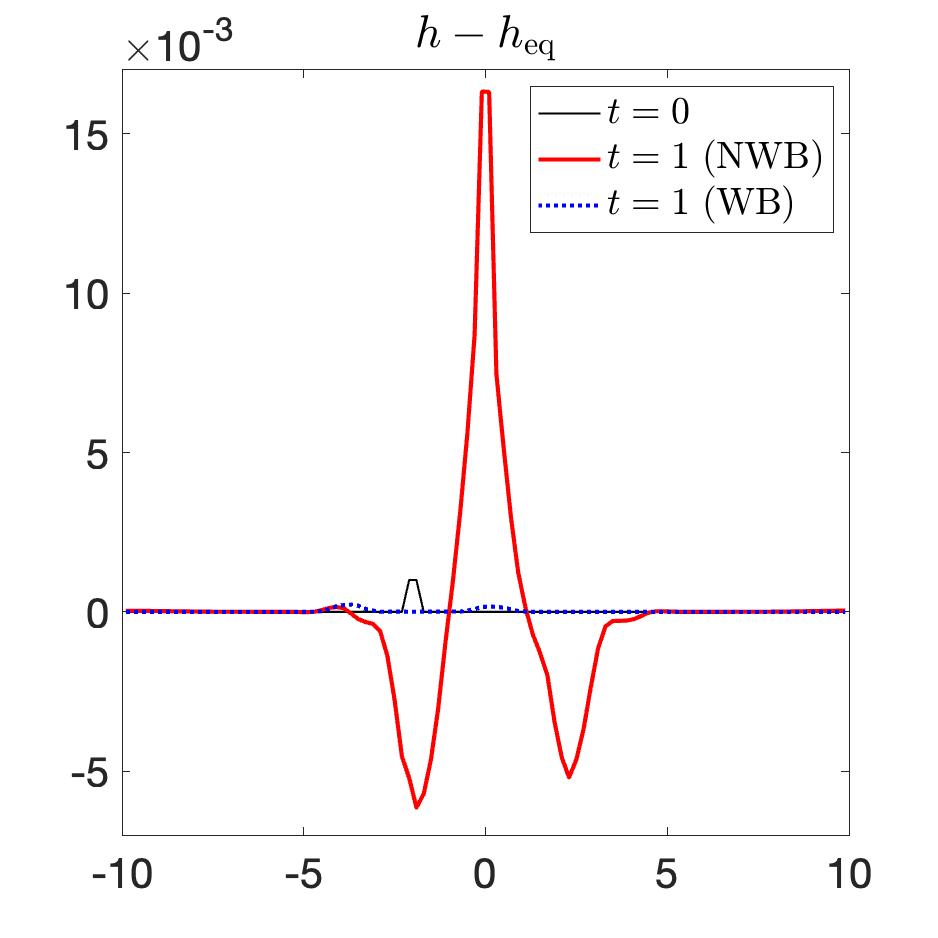}
            \includegraphics[trim=0.6cm 0.8cm 1.1cm 0.2cm, clip, width=5.7cm]{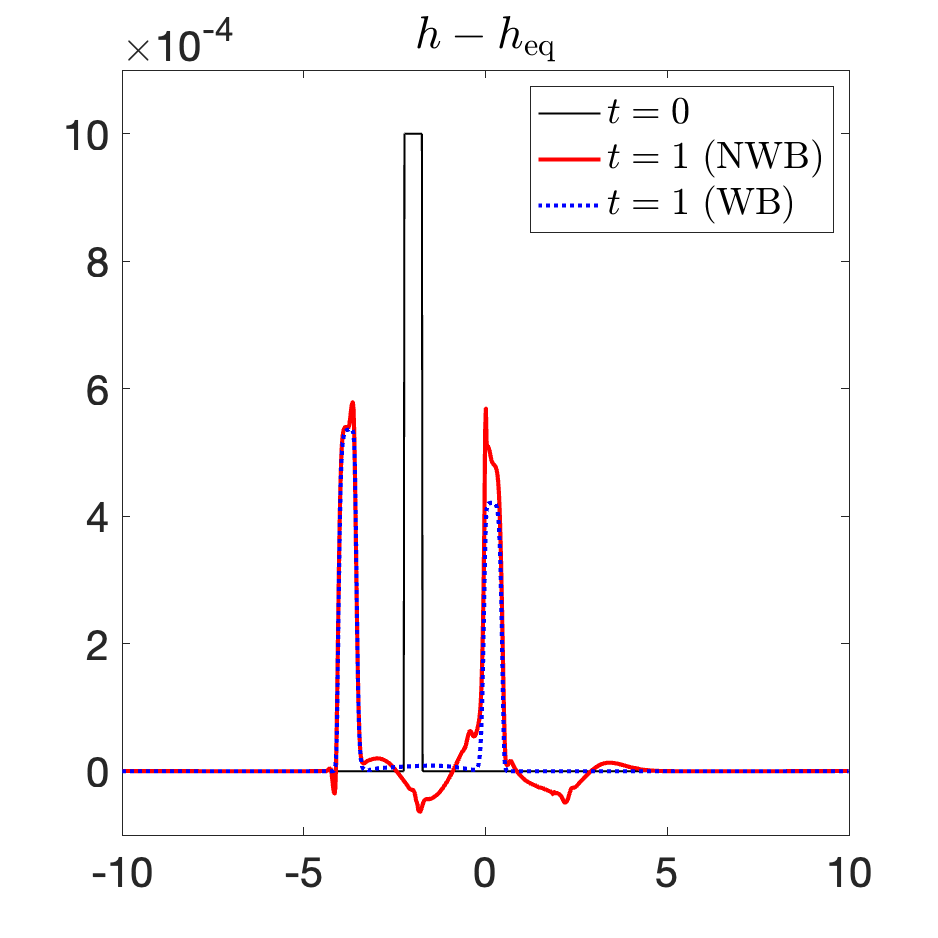}
            \includegraphics[trim=0.6cm 0.8cm 1.1cm 0.2cm, clip, width=5.7cm]{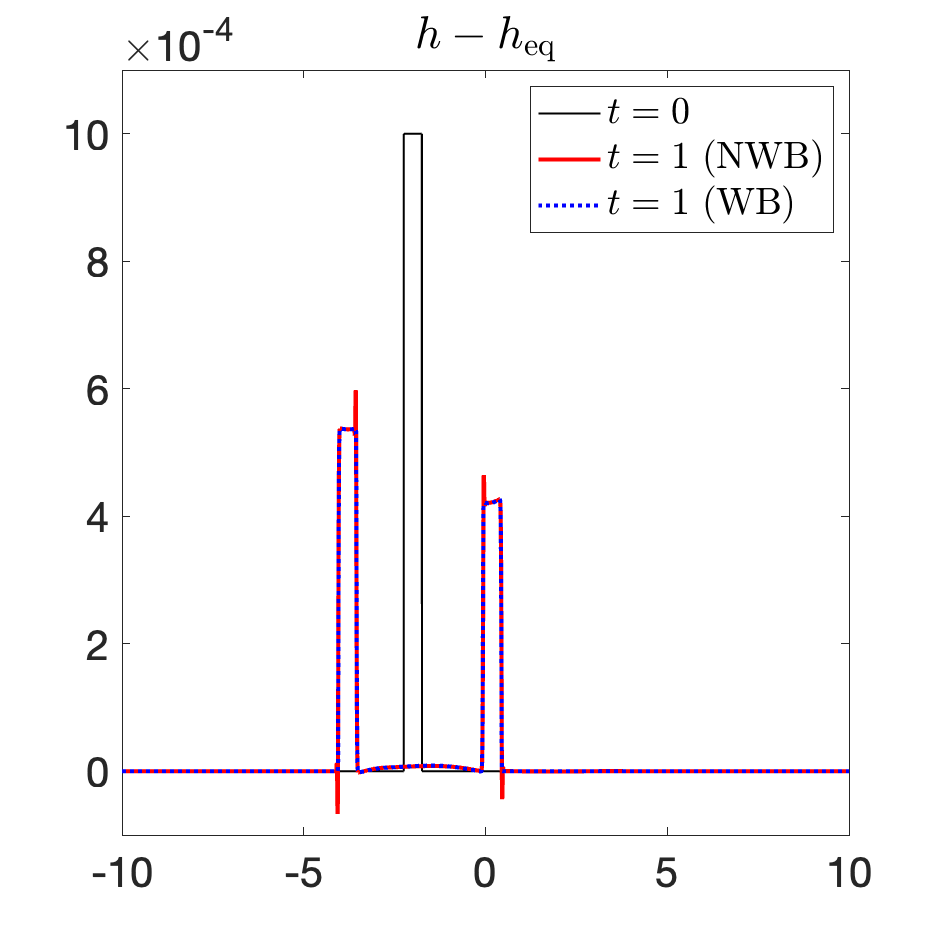}}
\vskip7pt
\centerline{\includegraphics[trim=0.6cm 0.8cm 1.1cm 0.2cm, clip, width=5.7cm]{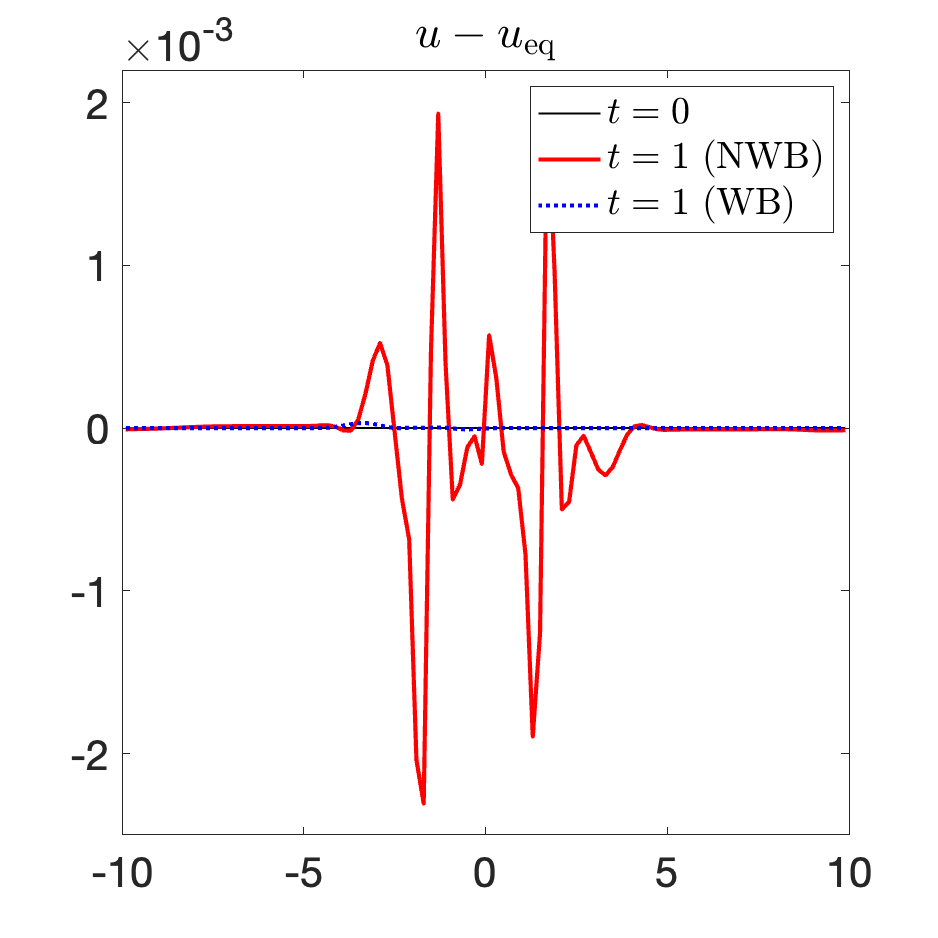}
            \includegraphics[trim=0.6cm 0.8cm 1.1cm 0.2cm, clip, width=5.7cm]{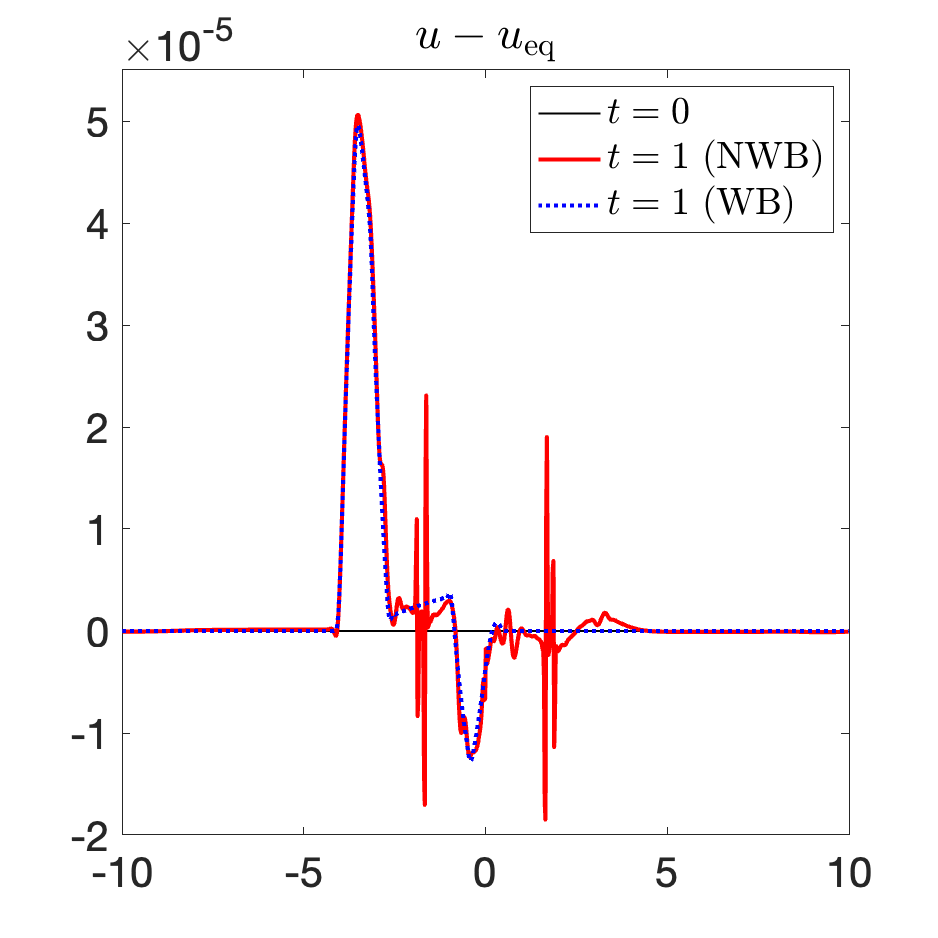}
            \includegraphics[trim=0.6cm 0.8cm 1.1cm 0.2cm, clip, width=5.7cm]{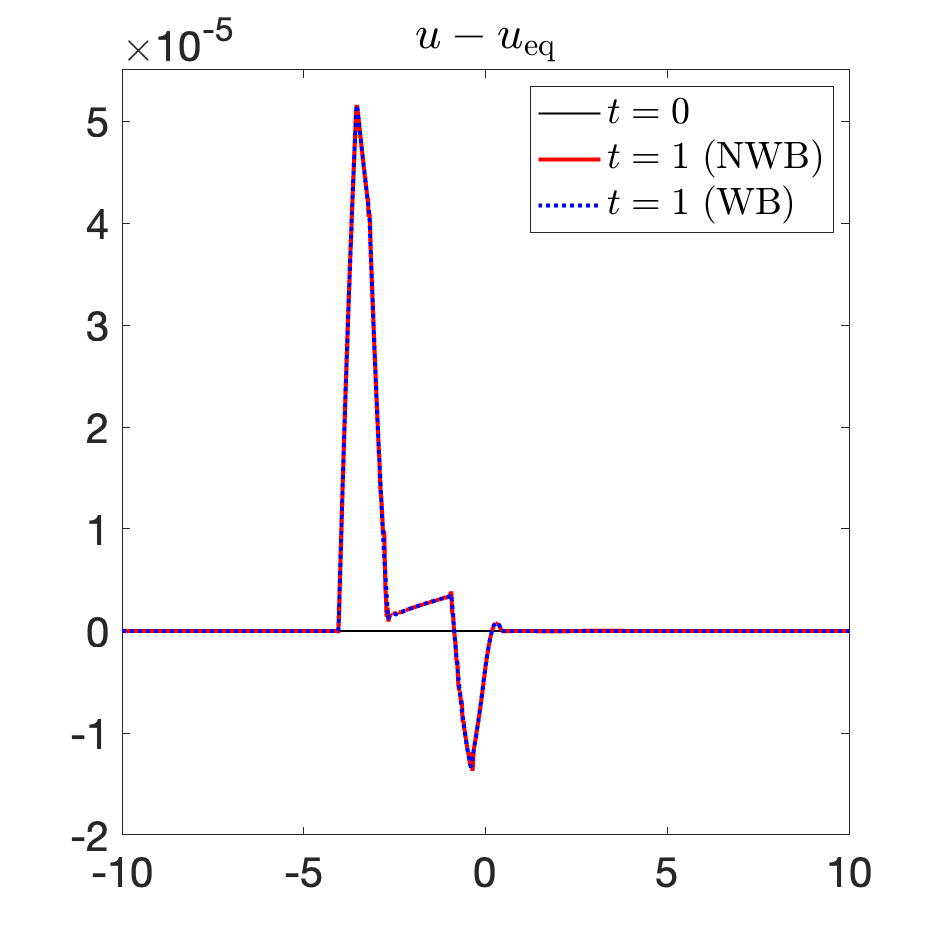}}
\vskip7pt
\centerline{\includegraphics[trim=0.6cm 0.8cm 1.1cm 0.2cm, clip, width=5.7cm]{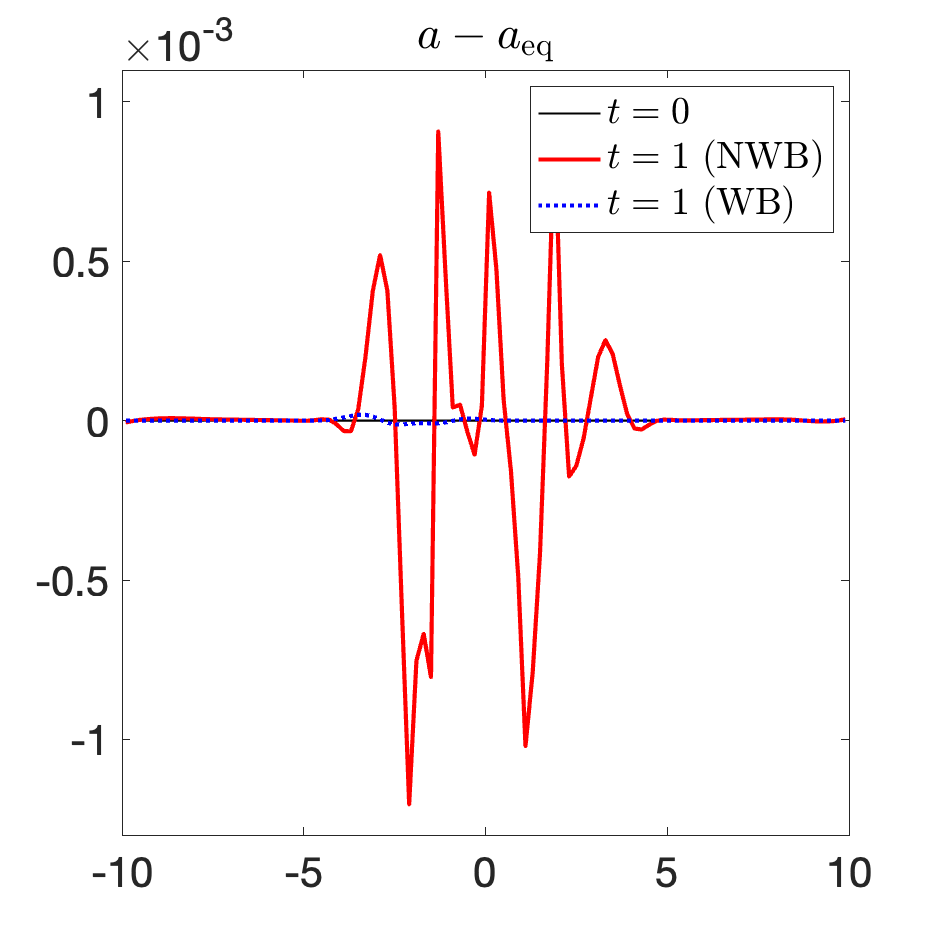}
            \includegraphics[trim=0.6cm 0.8cm 1.1cm 0.2cm, clip, width=5.7cm]{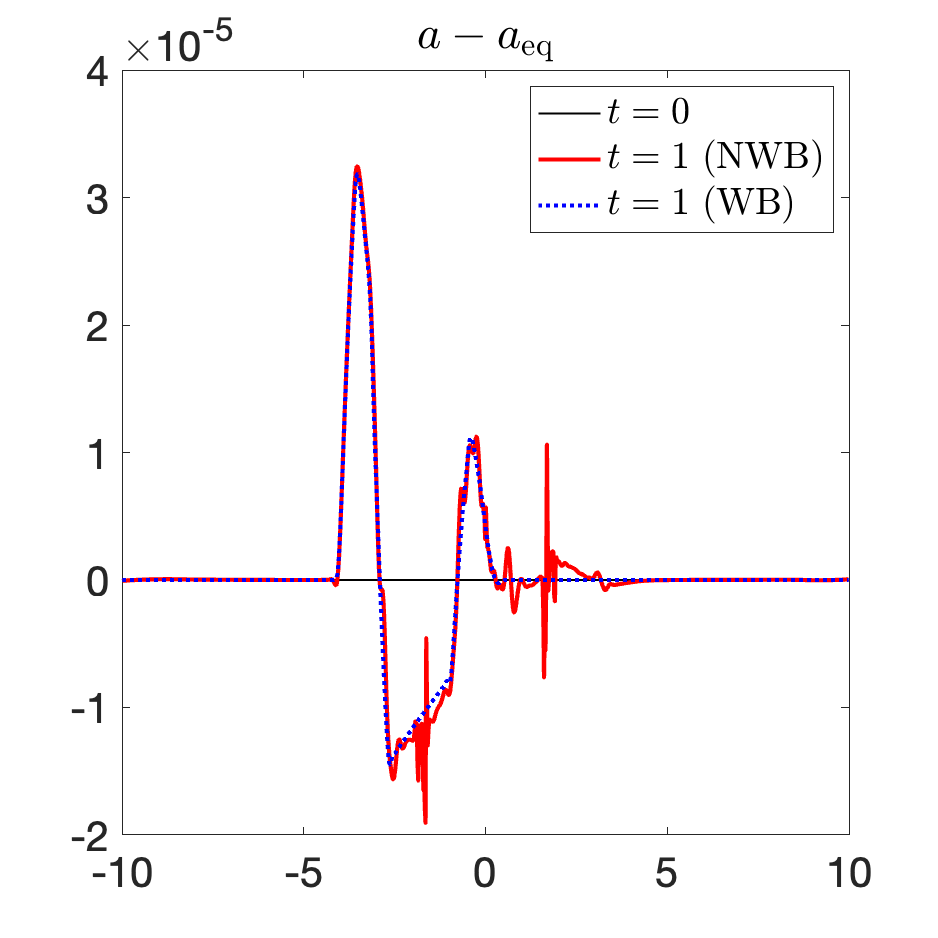}
            \includegraphics[trim=0.6cm 0.8cm 1.1cm 0.2cm, clip, width=5.7cm]{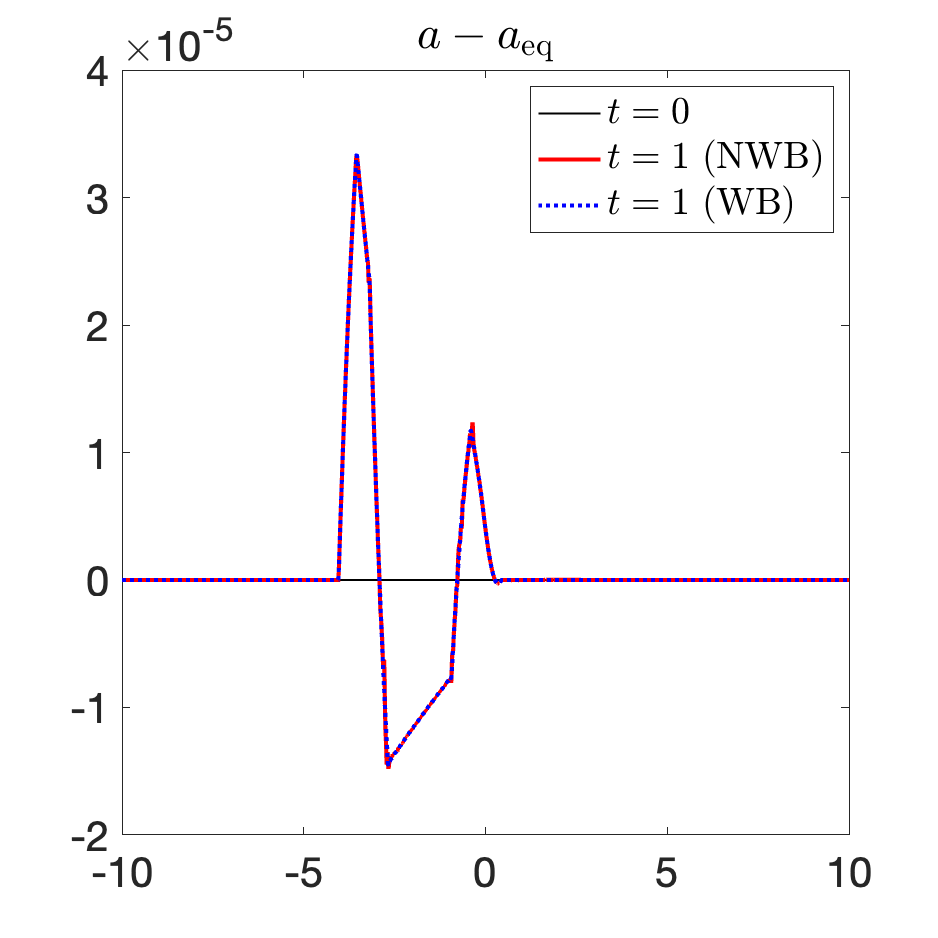}}
\caption{\sf Example 2 (small perturbation of the steady state): $h(y,1)-h_{\rm eq}(y)$ (top row), $u(y,1)-u_{\rm eq}(y)$ (middle row), and
$a(y,1)-a_{\rm eq}(y)$ (bottom row) computed by both the WB and NWB schemes with $N=100$ (left column), 1000 (middle column), and 10000
(right column) uniform cells.\label{fig45}}
\end{figure}

\subsubsection*{Example 3---Magneto-Geostrophic Adjustment at Low Rossby Numbers}
In this example, we consider the magneto-geostrophic adjustment problem with low Rossby numbers, that is, with both $Ro<1$ and $Ro_m<1$. As
a result, smooth outward-moving waves are initially not expected to form shocks as they propagate.

We consider the following initial conditions:
\begin{equation*}
\big(h(y,0),u(y,0),v(y,0),a(y,0),b(y,0)\big)=\big(1,0.1e^{-y^2},0,0,0.1\big),
\end{equation*}
with the constant Coriolis parameter $f(y)\equiv1$ and flat bottom topography $Z(y)\equiv0$ on the computational domain $[-200,200]$ subject
to the outflow boundary conditions.

We first use the above setting to test the experimental rate of convergence achieved by the proposed 1-D flux globalization based WB PCCU
scheme. To this end, we compute the solution until $t=5$ on a uniform mesh with $N=32000$ and plot the obtained $h$, $u$, $v$, and $a$ in
Figure \ref{fig46f}. In order to obtain the experimental $L^1$ rate of convergence, we compute the solution on several different meshes and
then use the Runge formula
\begin{equation*}
{\rm Rate}_N(h)=\log_2\left(\frac{\|h_{N/2}-h_N\|_1}{\|h_N-h_{2N}\|_1}\right),
\end{equation*}
where $h_N$ denotes the water depth $h$ computed on the uniform mesh consisting of $N$ cells (similar formulae can be written for the other
components of the computed solution). The obtained results, reported in Table \ref{tab43}, confirm that the expected second order of
accuracy has been reached.
\begin{figure}[ht!]
\centerline{\includegraphics[trim=0.6cm 0.7cm 1.5cm 0.2cm, clip, width=6.0cm]{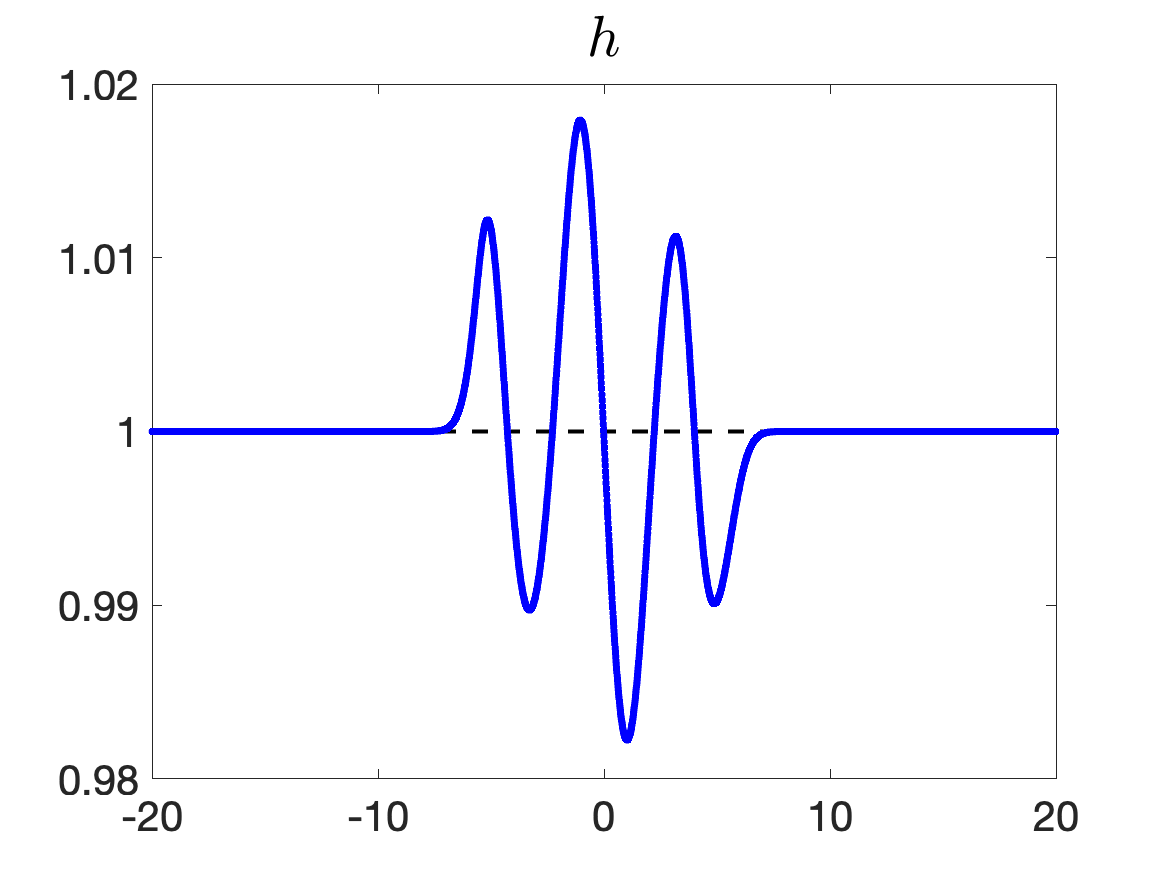}\hspace*{0.2cm}
            \includegraphics[trim=0.6cm 0.7cm 1.5cm 0.2cm, clip, width=6.0cm]{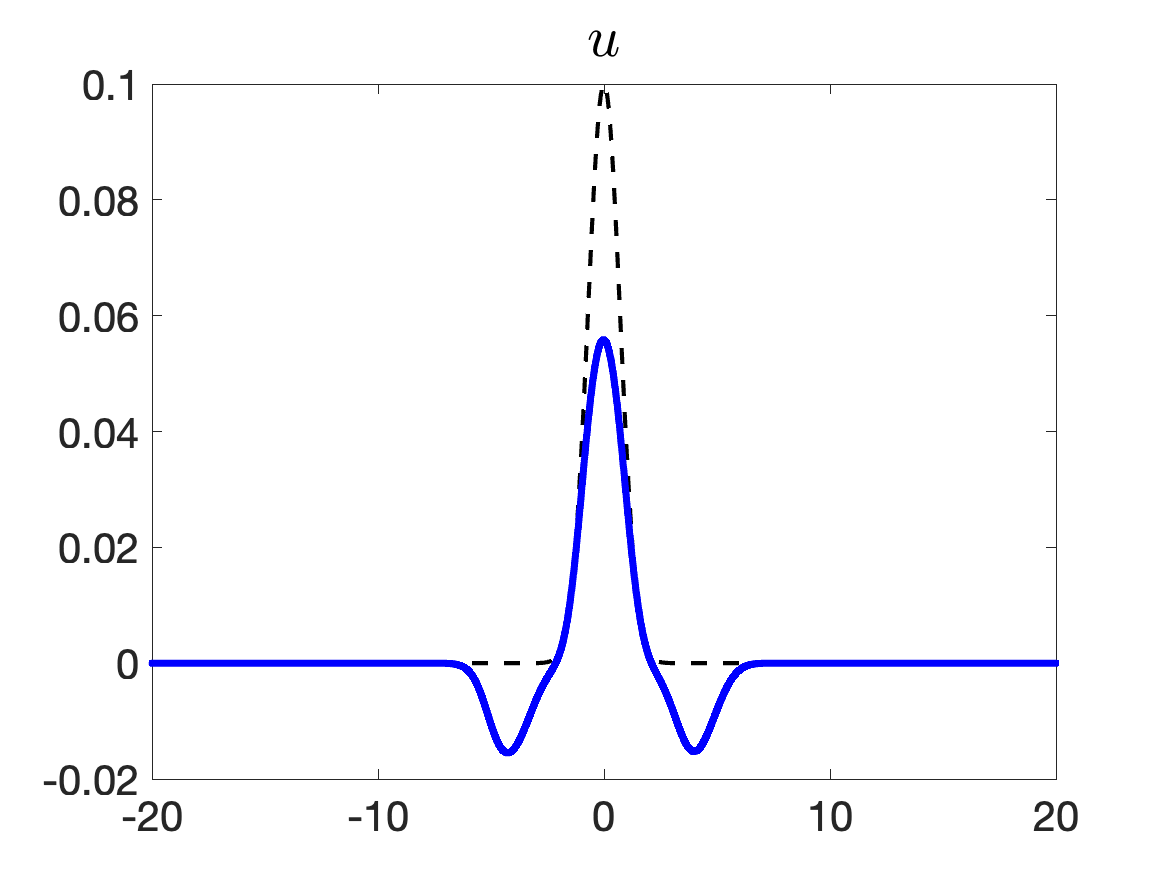}}
\vskip7pt
\centerline{\includegraphics[trim=0.6cm 0.7cm 1.5cm 0.2cm, clip, width=6.0cm]{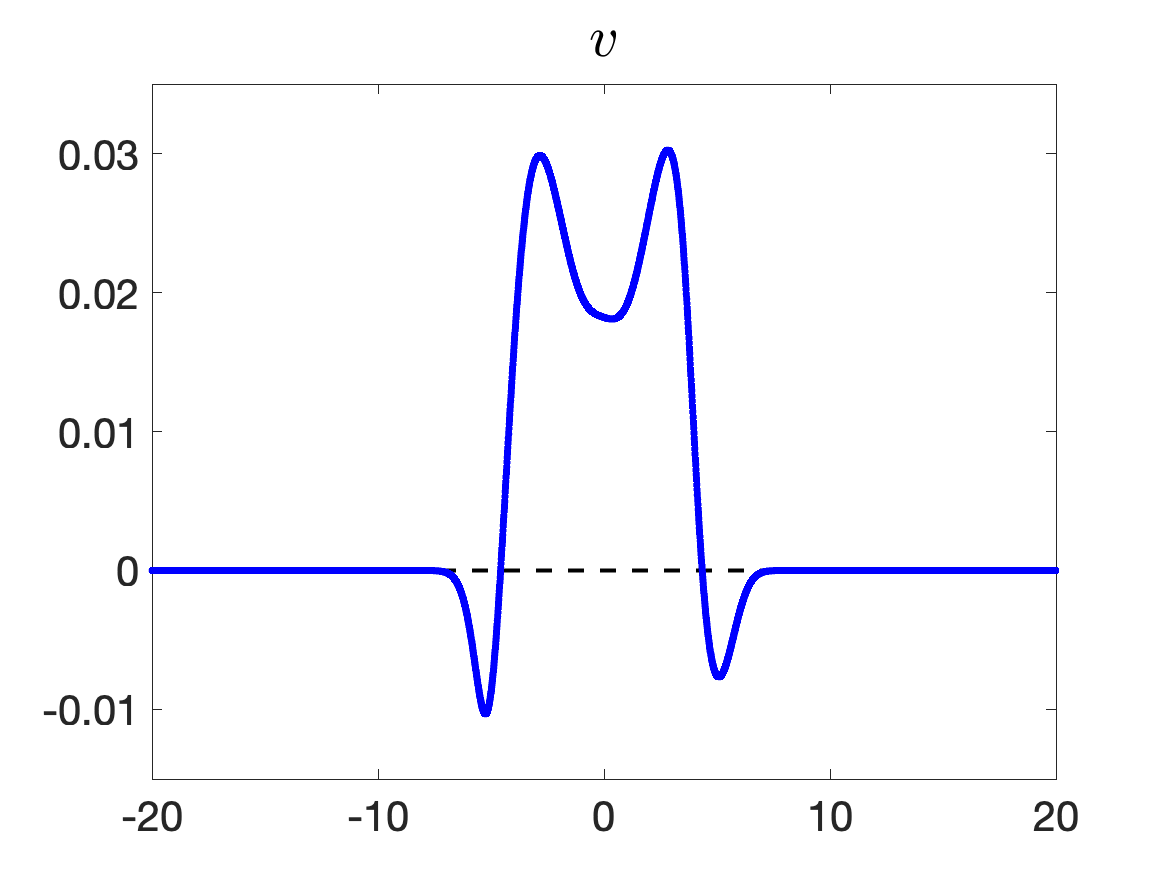}\hspace*{0.2cm}
            \includegraphics[trim=0.6cm 0.7cm 1.5cm 0.2cm, clip, width=6.0cm]{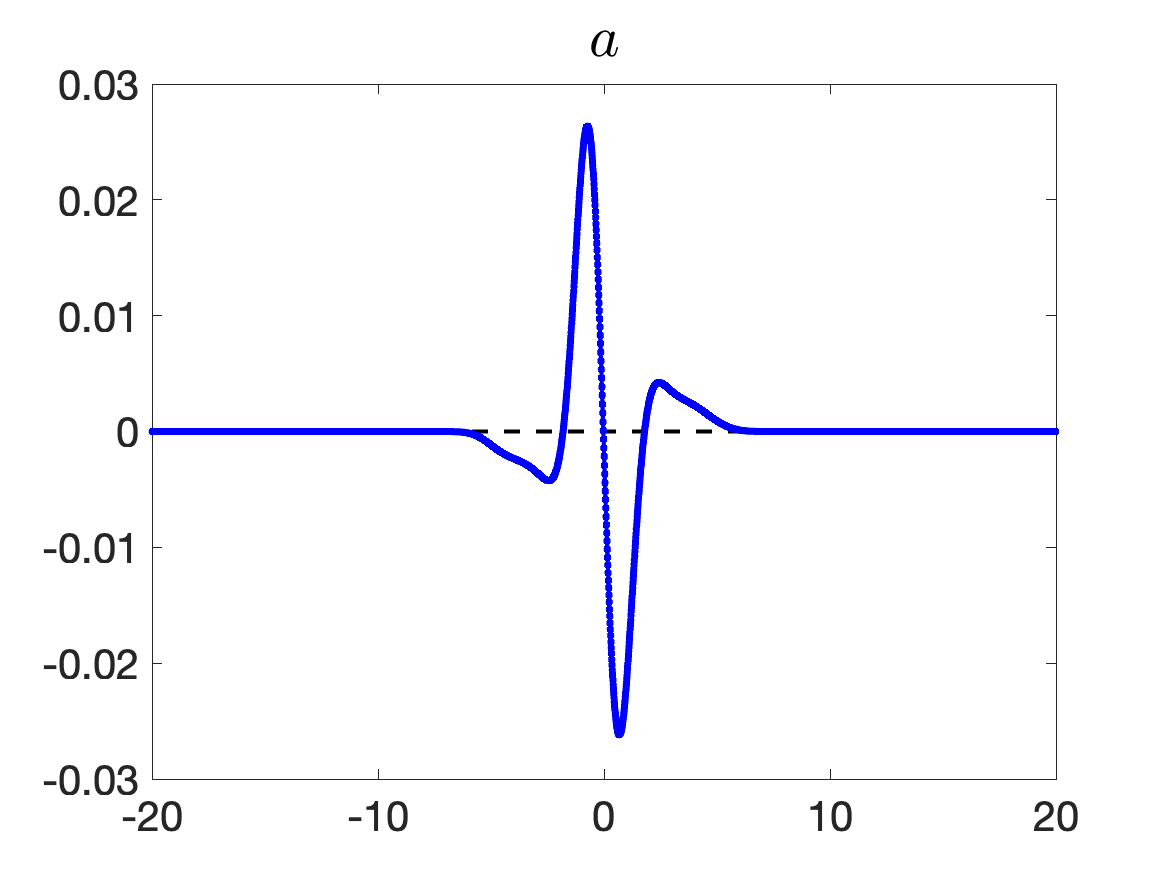}}
\caption{\sf Example 3: Computed $h(y,5)$, $u(y,5)$, $v(y,5)$, and $a(y,5)$ together with the corresponding initial conditions (dashed black
lines). Zoom at $[-20,20]$.\label{fig46f}}
\end{figure}
\begin{table}[ht!]
\centering
\small{
\begin{tabular}{|c|cc|cc|cc|cc|}
\hline
$N$&$\|h_N-h_{2N}\|_1$&Rate&$\|u_N-u_{2N}\|_1$&Rate&$\|v_N-v_{2N}\|_1$&Rate&$\|a_N-a_{2N}\|_1$&Rate\\
\hline
4000 &1.26e-03&2.19&1.69e-03&1.93&1.10e-03&1.86&1.70e-03&2.34\\
8000 &2.74e-04&2.21&3.84e-04&2.14&2.57e-04&2.09&2.59e-04&2.71\\
16000&6.21e-05&2.14&7.38e-05&2.38&6.11e-05&2.07&4.43e-05&2.55\\
32000&1.49e-05&2.06&1.38e-05&2.42&1.50e-05&2.02&7.23e-06&2.62\\
64000&3.67e-06&2.02&2.79e-06&2.30&3.72e-06&2.01&1.30e-06&2.48\\
\hline
\end{tabular}}
\caption{\sf Example 3: $L^1$-errors and the corresponding experimental rates of convergence.\label{tab43}}
\end{table}

We should emphasize that Figure \ref{fig46f} confirms the scenario of magneto-geostrophic adjustment sketched in \S\ref{sec411}, showing
that at $t=5$ fast magneto-inertia-gravity waves have already been evacuated from the perturbation location and slow Alfv\'en waves are
being emitted, as follows from the phase relations between $u$ and $a$, which were already discussed in Example 1. We now test the
magneto-geostrophic equilibrium of the quasi-stationary central part of the perturbation. To this end, we compute the numerical solution
until a relatively large final time $t=40$ and measure the quantities on both sides of \eref{4.5} as they are to be the same at the
aforementioned steady state. However, at $t=40$, these quantities remain quite different, as shown in Figure \ref{fig47} (left). This often
occurs in geostrophic adjustments when waves have near-zero group velocities and thus stay in the center of the computational domain for
long times; see, e.g., the discussion in \cite{Kurganov2020Well}. Under such circumstances, the magneto-geostrophic balance should be checked
for time-averaged components. We therefore take the time averages in \eref{4.5},
\begin{equation}
\frac{1}{T-2T_f}\int\limits_{2T_f}^T\left(gh_y-bb_y\right){\rm d}t=-\frac{1}{T-2T_f}\int\limits_{2T_f}^Tfu\,{\rm d}t,
\label{4.6}
\end{equation}
where $T_f=2\pi/f$, and measure the LHS and RHS of \eref{4.6} at $T=40$. The obtained results are reported in Figure \ref{fig47} (right),
where one can see that the proposed method does indeed time-advance the numerical solution to the expected magneto-geostrophic equilibrium.
\begin{figure}[ht!]
\centerline{\includegraphics[trim=0.8cm 0.9cm 1.5cm 0.9cm, clip, width=7.0cm]{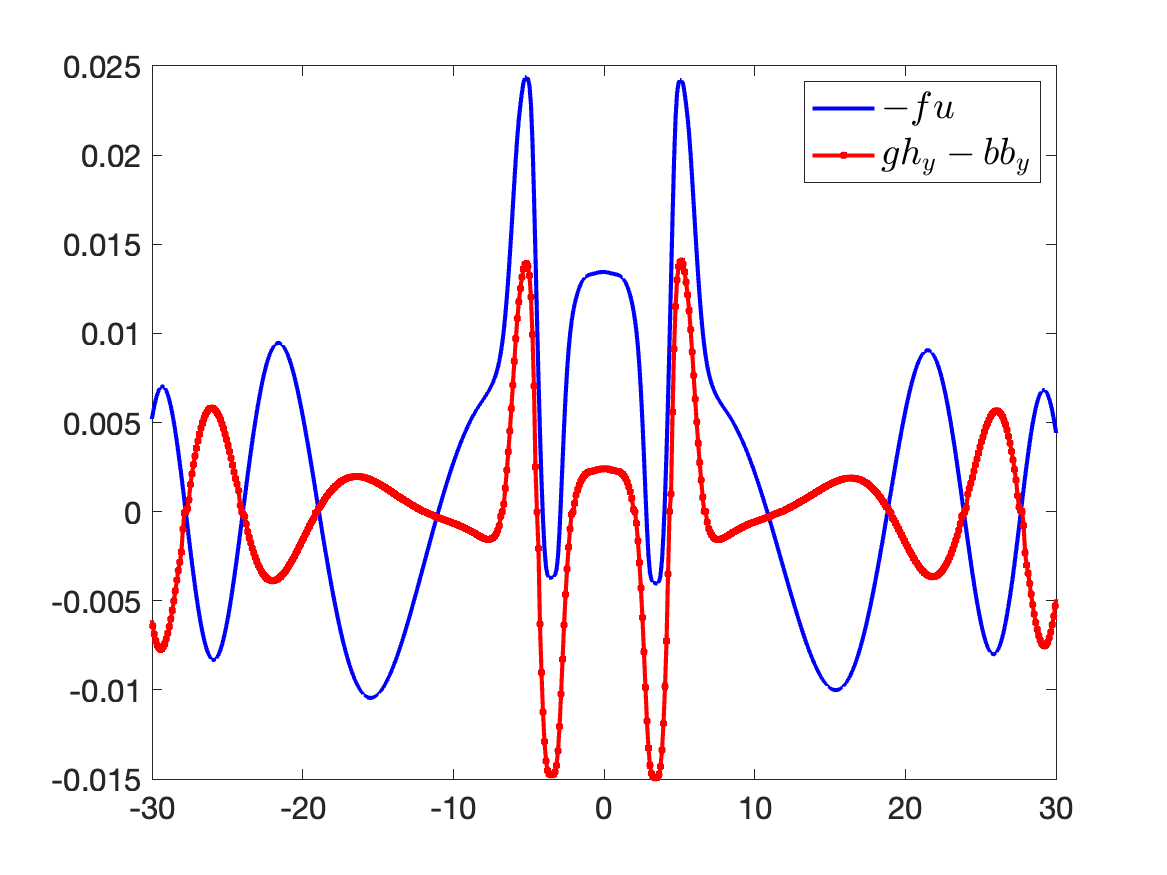}\hspace*{0.5cm}
            \includegraphics[trim=0.8cm 0.9cm 1.5cm 0.9cm, clip, width=7.0cm]{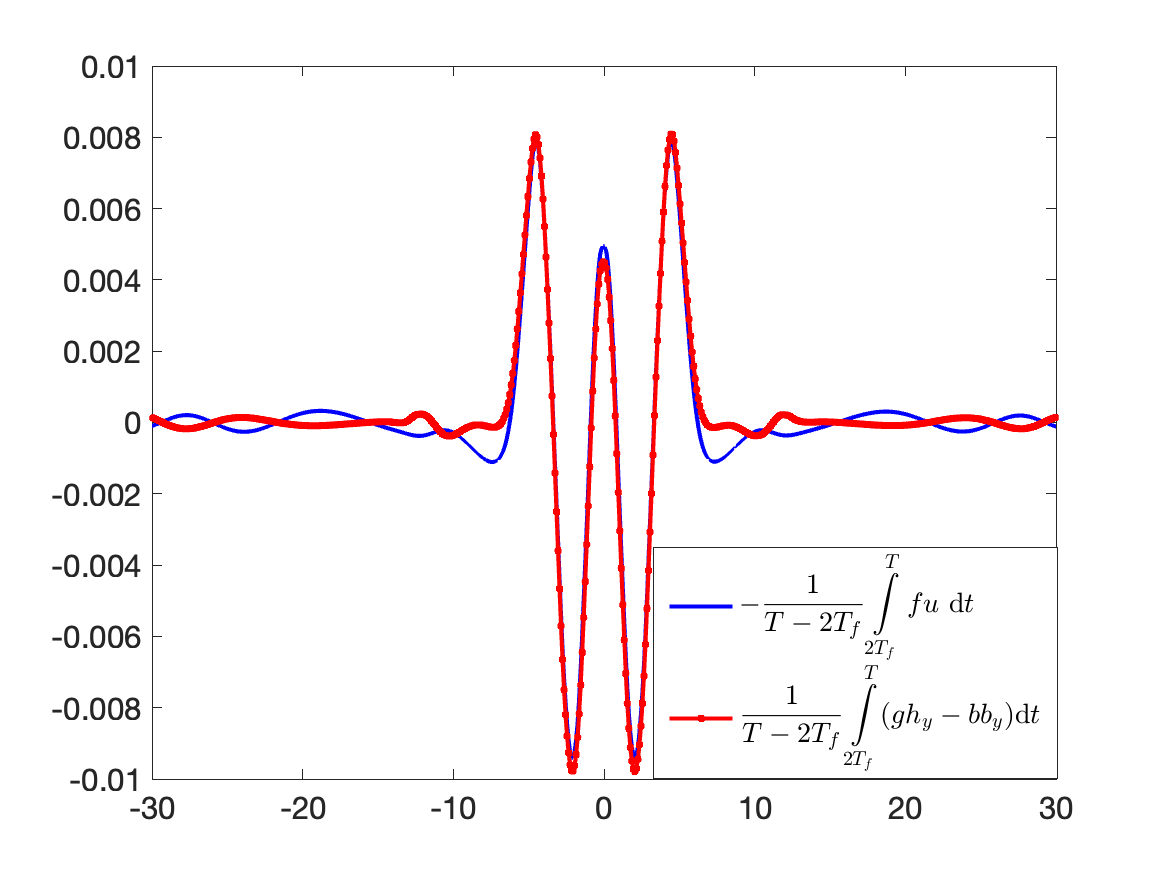}}
\caption{\sf Example 3: $gh_y-bb_y$ and $-fu$ at time $t=40$ (left) and the corresponding time-averaged quantities (right). Zoom at
$[-30,30]$.\label{fig47}}
\end{figure}

\subsubsection*{Example 4---Magneto-Geostrophic Adjustment at High Rossby Numbers}
We continue our exploration of magneto-geostrophic adjustment with an example where $Ro>1$ and $Ro_m>1$. The results are to be compared with
those in ``pure'', non-magnetic RSW model \cite{zeitlin2004rsw}, this is why the initial conditions are the same, in what concerns velocity
and thickness, with a superimposed constant meridional magnetic field:
\begin{equation*}
\begin{aligned}
&h(y,0)\equiv1,\quad u(y,0)=\frac{11}{10}\cdot\frac{(1+\tanh\p{4x+2})(1-\tanh\p{4x-2})}{(1+\tanh 2)^2},\\
&v(y,0)\equiv0,\quad a(y,0)\equiv0,\quad b(y,0)\equiv1.1,
\end{aligned}
\end{equation*}
with the constant Coriolis parameter $f(y)\equiv1$ and flat bottom topography $Z(y)\equiv0$ on the computational domain $[-200,200]$ subject
to the outflow boundary conditions.

{We first compute the solution by the proposed 1-D flux globalization based WB PCCU scheme until $t=5$ on a uniform mesh with $N=32000$. The
obtained $h$, $u$, $v$, and $a$ are shown in Figure \ref{fig48}, where one can clearly see that by that time, compared with the previous
example, the solution has developed left- and right-propagating discontinuities. Shock formation is expected, in the light of the results in
\cite{zeitlin2004rsw}, in the $h$ and $v$ fields, although the form of both signals differs from those in \cite[Figure 2]{zeitlin2004rsw}.
So, the evolution of these fields is affected by the mean magnetic field. Moreover, the discontinuities are also clearly seen in the $u$ and
$a$ fields. These contact/tangential discontinuities, which are transverse to the direction of propagation, are well-known in MHD; see,
e.g., \cite{LandauLifshitz}. A difference, more clearly seen in the left-moving waves, between the speed of the discontinuities observed in
the $h$ and $v$ fields compared with the $u$ and $a$ ones is since they are associated with faster magneto-inertia-gravity
and slower rotation-modified Alfv\'en waves, respectively. Wave-breaking and shock formation also manifest themselves in the evolution of
the total energy \eref{energy1} presented in Figure \ref{fig49}, where one can see that the energy first remains practically constant but
then diminishes after the appearance of the discontinuity. We have also verified that the energy drops across the shock, as it should.
\begin{figure}[ht!]
\centerline{\includegraphics[trim=0.6cm 0.7cm 1.5cm 0.2cm, clip, width=6.0cm]{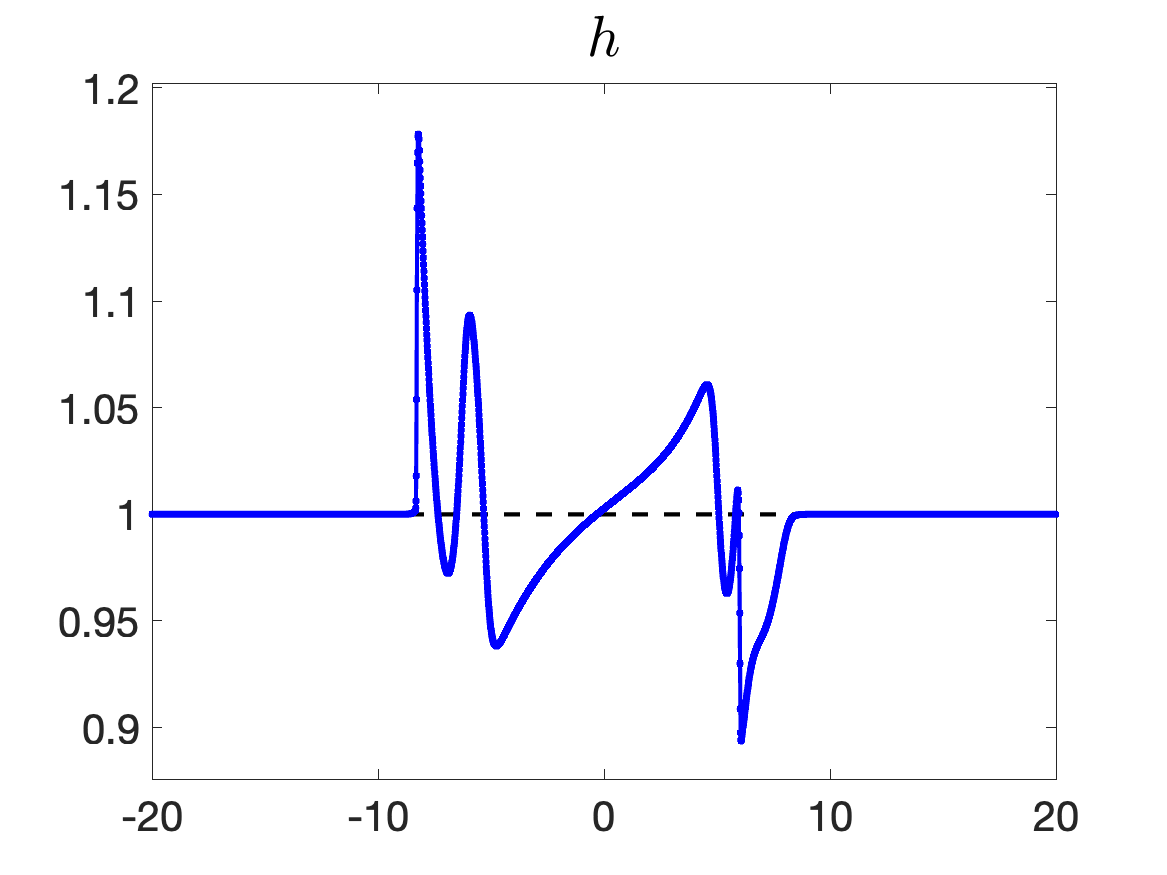}\hspace*{0.2cm}
            \includegraphics[trim=0.6cm 0.7cm 1.5cm 0.2cm, clip, width=6.0cm]{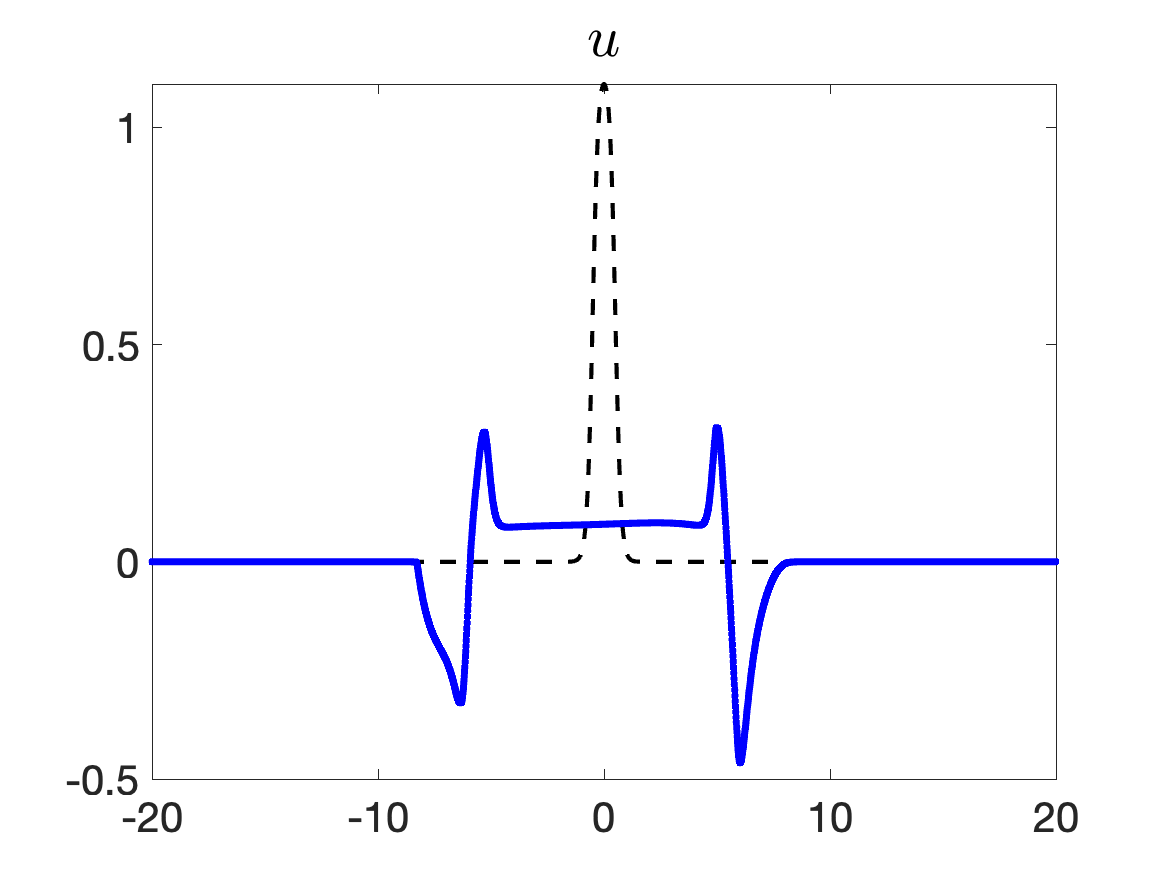}}
\vskip7pt
\centerline{\includegraphics[trim=0.6cm 0.7cm 1.5cm 0.2cm, clip, width=6.0cm]{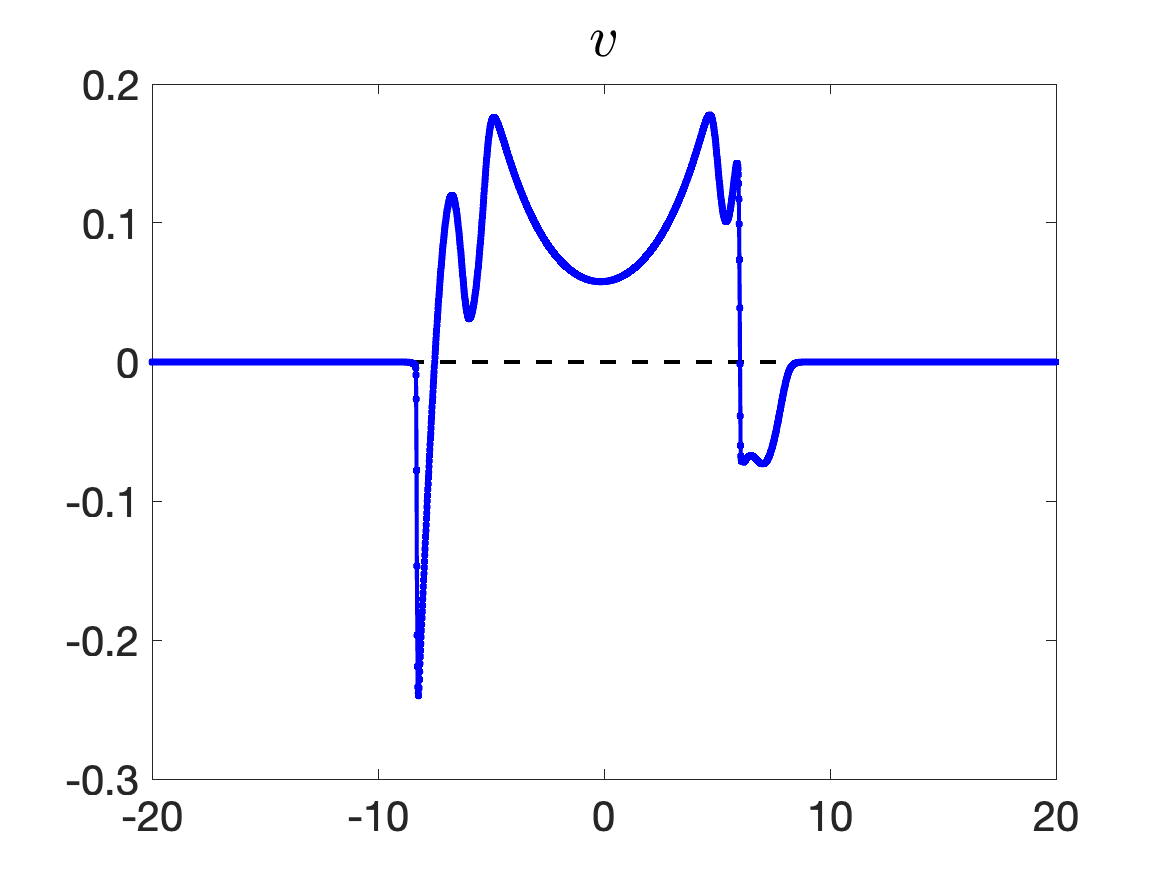}\hspace*{0.2cm}
            \includegraphics[trim=0.6cm 0.7cm 1.5cm 0.2cm, clip, width=6.0cm]{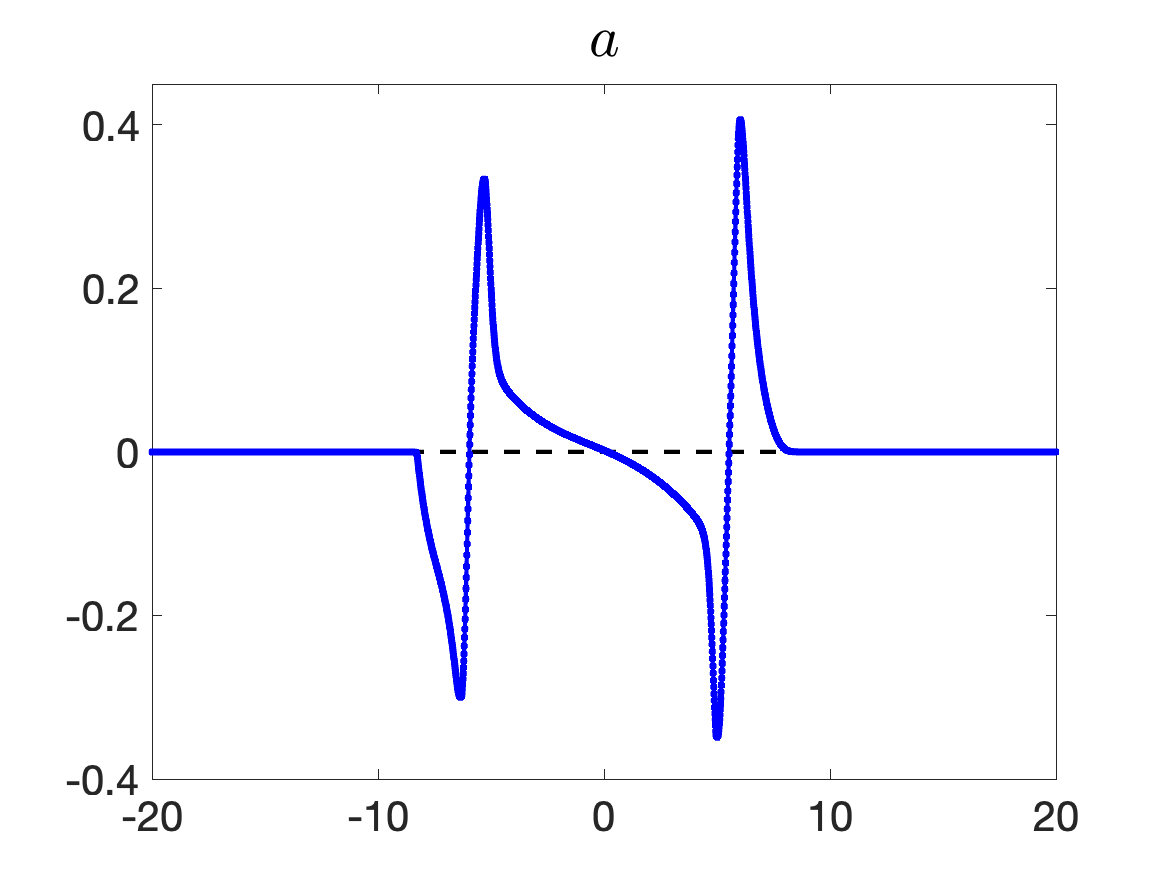}}
\caption{\sf Example 4: Computed $h(y,5)$, $u(y,5)$, $v(y,5)$, and $a(y,5)$ together with the corresponding initial conditions (dashed black
lines). Zoom at $[-20,20]$.\label{fig48}}
\end{figure}
\begin{figure}[ht!]
\centerline{\includegraphics[trim=0.0cm 0.5cm 3.4cm 0.5cm, clip, width=6.0cm]{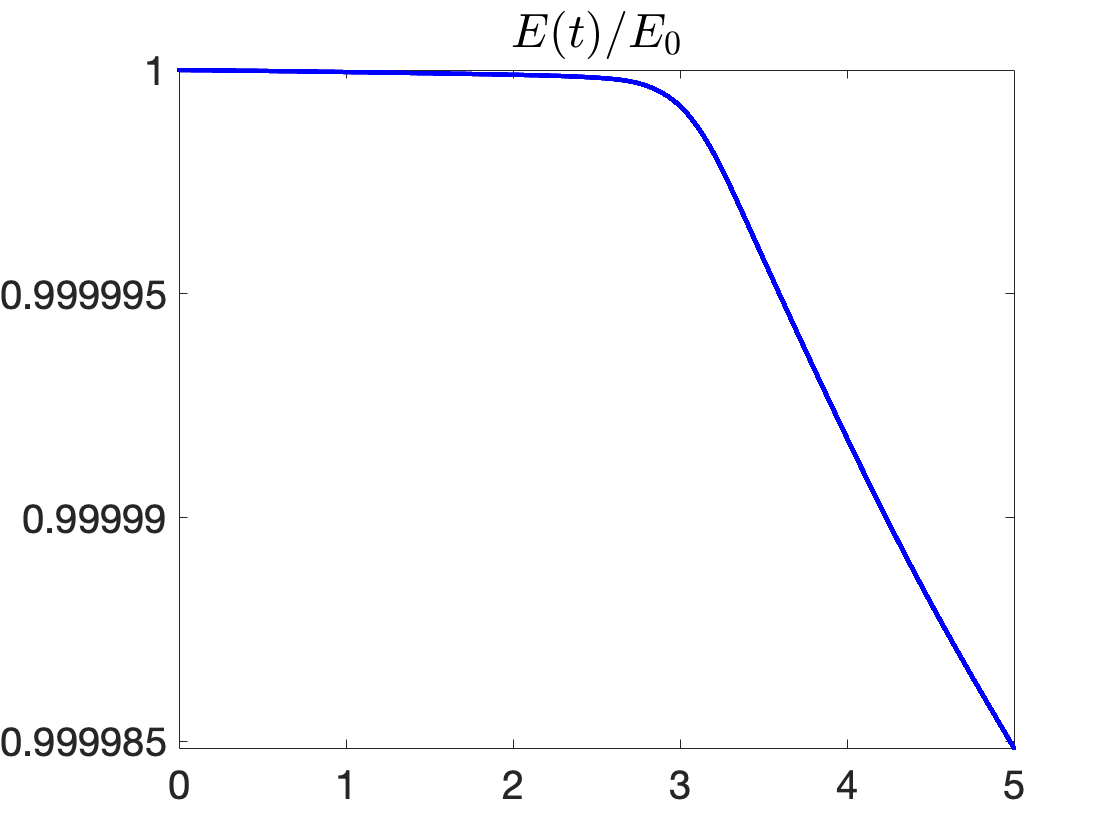}}
\caption{\sf Example 4: Time evolution of total energy $E(t)$, scaled by the initial total energy $E(0)$.\label{fig49}}
\end{figure}

Similarly to the experiments conducted in the previous example, we then study the convergence towards the magneto-geostrophic equilibrium.
To this end, we use the time-averaged computation (see \eref{4.6}) to verify whether the magneto-geostrophic balance has been achieved by
$T=100$ (we take a larger final time than in Example 3 as the solution here is nonsmooth and its magnetic field is stronger, so the
convergence is expected to be slower). The obtained results are reported in Figure \ref{fig410}, where we plot the LHS and RHS of
\eref{4.6}. One can observe a reasonable agreement, although there is a discrepancy in the center of the original jet, presumably
due to nonlinear effects.
\begin{figure}[ht!]
\centerline{\includegraphics[trim=1.9cm 0.9cm 1.4cm 0.4cm, clip, width=7.0cm]{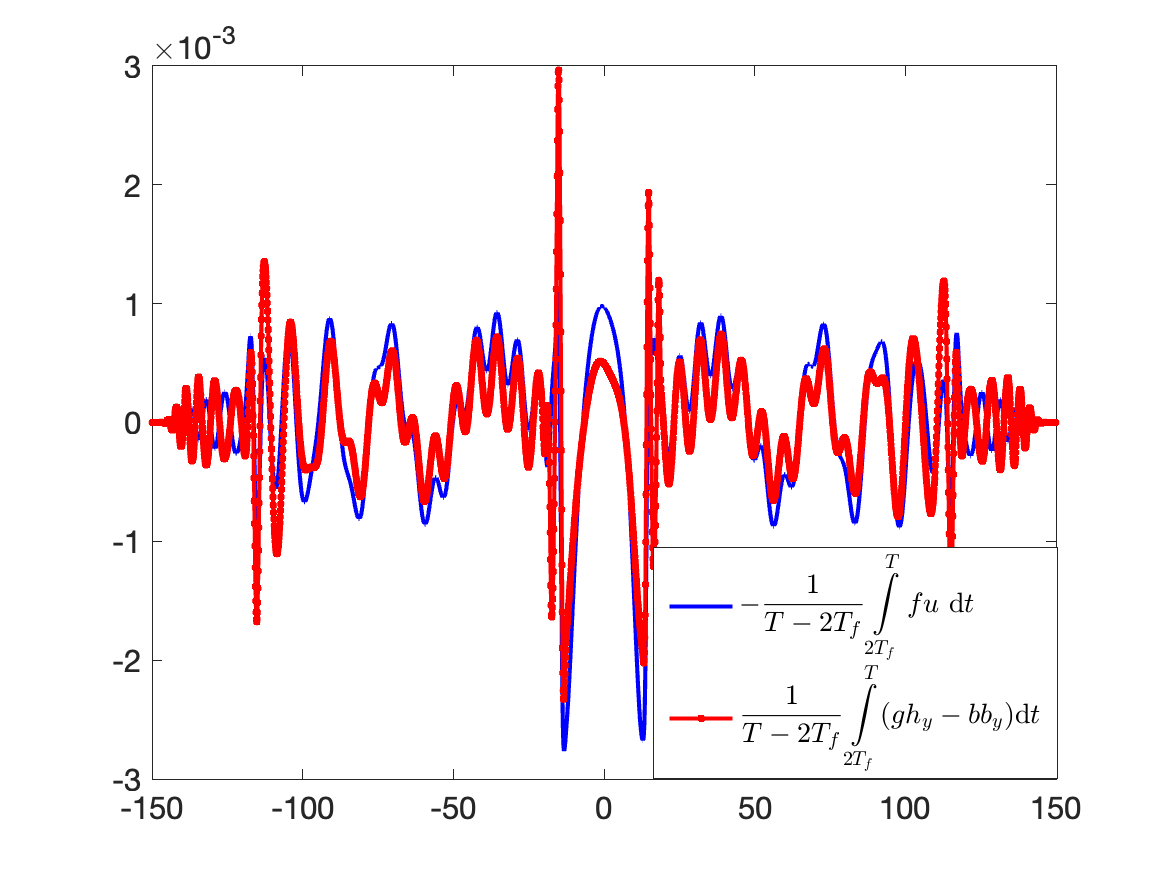}}
\caption{\sf Example 4: $\frac{1}{T-2T_f}\int_{2T_f}^T\left(gh_y-bb_y\right){\rm d}t$ and $-\frac{1}{T-2T_f}\int_{2T_f}^Tfu\,{\rm d}t$ at
$T=100$. Zoom at $[-150,150]$.\label{fig410}}
\end{figure}

\subsection{2-D Numerical Examples}\label{sec42}
\subsubsection{General Facts About the 2-D MRSW Model}\label{sec421}
A straightforward linearization of the system \eref{1.1} over the state of rest with a constant magnetic field reveals the presence of fast
magneto-inertia-gravity waves and slow rotation-modified Alfv\'en waves. The former can propagate in any direction, while the latter
propagates only along the direction defined by the magnetic field vector.

We note that, in general, neither geostrophic,
\begin{equation*}
f\bm u^\perp=-g\nabla h,~~\bm b=\bm0,
\end{equation*}
nor magneto-geostrophic,
\begin{equation*}
f\bm u^\perp=-g\nabla h+\bm b\cdot\nabla\bm b,
\end{equation*}
equilibria are steady solutions of the 2-D MRSW model unless the former one is unidirectional, like in \S\ref{sec411}. This situation is
analogous to that of non-magnetic RSW equations, for which numerous studies show that, at least at low Rossby numbers, any localized
perturbation rapidly adjusts to a state close to geostrophic equilibrium, which then slowly evolves, obeying the so-called quasi-geostrophic
dynamics, which is, essentially, vortex dynamics; see, e.g., \cite{zeitlin2018geophysical} and references therein. Magneto quasi-geostrophic
(MQG) approximation of the MRSW equations, which can be constructed at small Rossby and magnetic Rossby numbers
\cite{Lahaye2022Coherent,Zeitlin2013Remarks}, describes slow rotation-modified Alfv\'en waves together with vortices. Magneto-geostrophic
adjustment remains largely unstudied in the 2-D MRSW model; the only work in this direction, to the best of our knowledge, is an
investigation in \cite{Magill2019vortexadjustment} of the evolution of a strong small-scale non-equilibrated Gaussian vortex in a uniform
magnetic field.

We should emphasize that an advantage of the 2-D MRSW model compared to the 1-D one is that it allows describing such physically important
dynamical entities as vortices. The simplest idealized vortex configuration is axisymmetric with only azimuthal velocity components. In order
to better understand this type of structure, it is useful to rewrite the system \eref{1.1} in polar coordinates $r$ and $\theta$ in the
nonconservative form:
\begin{equation}
\begin{aligned}
&\frac{\d}{\d t}h+\frac{\d}{\d r}(h\mathfrak u)+\frac{1}{r}\frac{\d}{\d\theta}(h\mathfrak u)+\frac{1}{r}(h\mathfrak u)=0,\\
&\frac{{\rm d}\mathfrak u}{{\rm d}t}-\frac{\mathfrak v^2}{r}-f\mathfrak v=-g\frac{\d}{\d r}(h+Z)+\mathfrak a\frac{\d\mathfrak a}{\d r}+
\mathfrak b\frac{\d\mathfrak a}{\d\theta}-\frac{1}{r}\mathfrak b^2,\\
&\frac{{\rm d}\mathfrak v}{{\rm d}t}+\frac{\mathfrak u\mathfrak v}{r}+f\mathfrak u=-g\frac{1}{r}\frac{\d}{\d\theta}(h+Z)+
\mathfrak a\frac{\d\mathfrak b}{\d r}+\mathfrak b\frac{\d\mathfrak b}{\d\theta}+\frac{1}{r}\mathfrak a\mathfrak b,\\
&\frac{{\rm d}\mathfrak a}{{\rm d}t}-\mathfrak a\frac{\d\mathfrak u}{\d r}-\mathfrak b\frac{\d\mathfrak u}{\d\theta}=0,\\
&\frac{{\rm d}\mathfrak b}{{\rm d}t}-\mathfrak a\frac{\d\mathfrak v}{\d r}-\mathfrak b\frac{\d\mathfrak v}{\d\theta}+
\frac{1}{r}(\mathfrak a\mathfrak v-\mathfrak b\mathfrak u)=0,\\
&\frac{\d}{\d r}(h\mathfrak a)+\frac{1}{r}\frac{\d}{\d\theta}(h\mathfrak b)+\frac{1}{r}(h\mathfrak a)=0,
\end{aligned}
\label{2Dpolar}
\end{equation}
where the polar decomposition of velocity and magnetic fields is used: $\bm u=\mathfrak u\bm{\hat r}+\mathfrak v\bm{\hat\theta}$,
$\bm b=\mathfrak a\bm{\hat r}+\mathfrak b\bm{\hat\theta}$, where $\bm{\hat r}$ and $\bm{\hat\theta}$ are unit vectors in $\bm r$ and
$\bm\theta$ directions, respectively, and $\frac{{\rm d}}{{\rm d}t}:=\frac{\d}{\d t}+\mathfrak u\frac{\d}{\d r}+
\mathfrak v\frac{1}{r}\frac{\d}{\d\theta}$.

We notice that the divergence-free condition expressed by the last equation in \eref{2Dpolar} forbids configurations with a purely radial
nonsingular magnetic field. We also notice that axisymmetric {\em magneto-cyclo-geostrophic equilibria} between $v(r)$, $h(r)$, $Z(r)$, and
$b(r)$:
\begin{equation}
\frac{\mathfrak v^2}{r}+f\mathfrak v=g\frac{\d}{\d r}(h+Z)+\frac{1}{r}\mathfrak b^2,
\label{magcycgeo}
\end{equation}
with $\mathfrak a=\mathfrak u=0$ are exact steady-state solutions. Hence, we expect relaxation to one of these equilibria
(magneto-cyclo-geostrophic adjustment) for initial configurations close to those described by \eref{magcycgeo}. Obviously,
magneto-cyclo-geostrophic equilibria exist only on the $f$-plane as the beta-effect destroys axial symmetry.

\subsubsection*{Example 5---Quasi 1-D Steady-State with Linear Coriolis Parameter ($f(y)=0.1y$)}
In the first 2-D numerical example, we demonstrate the ability of the proposed flux globalization based WB PCCU scheme to preserve a quasi
1-D moving-water steady-state. To this end, we use the following initial conditions that satisfy \eref{3.11a}--\eref{3.12}:
\begin{equation}
\begin{aligned}
&u(x,y,0)=u_{\rm eq}(x,y)\equiv0.25,&&v(x,y,0)=v_{\rm eq}(x,y)\equiv0,&&a(x,y,0)=a_{\rm eq}(x,y)\equiv3,\\
&b(x,y,0)=b_{\rm eq}(x,y)\equiv0,&&E(x,y,0)=E_{\rm eq}(x,y)\equiv6,
\end{aligned}
\label{5.1}
\end{equation}
the bottom topography $Z(x,y)=\hf e^{-y^2}$, and the outflow boundary conditions set for the equilibrium variables. Notice that in this
example, unlike the 1-D Examples 1 and 2, the profile of $h_{\rm eq}(x,y)$, which depends on $y$ only, can be computed analytically using
\eref{3.11a} and \eref{5.1}; its 1-D slice is shown in Figure \ref{fig411}.
\begin{figure}[ht!]
\centerline{\includegraphics[trim=5.0cm 0.6cm 5.0cm 0.9cm, clip, width=5.5cm]{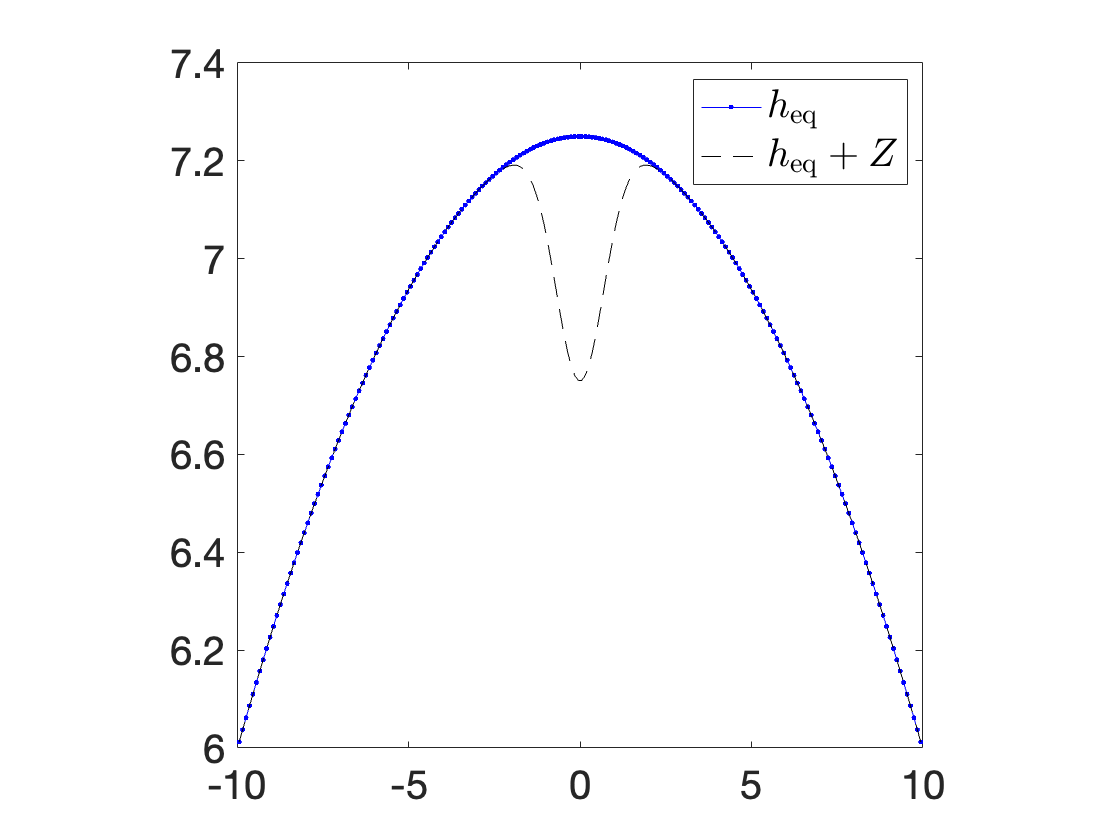}}
\caption{\sf Example 5: The 1-D slice along the $y$-axis of the steady-state fluid depth $h_{\rm eq}$ and fluid level $h_{\rm eq}+Z$.
\label{fig411}}
\end{figure}

We compute the numerical solutions on the computational domain $[10,10]\times[-10,10]$ using a uniform $100\times100$ mesh by both the WB
and NWB schemes until the final time $t=1$. The results reported in Table \ref{tab44} show that the WB scheme, as expected, preserves the
considered quasi 1-D steady state within the machine accuracy, while the NWB scheme fails to do so.
\begin{table}[ht!]
\begin{center}
\begin{tabular}{|c|c|c|c|c|}
\hline
Scheme&$\|h(\cdot,\cdot,5)-h_{\rm eq}\|_\infty$&$\|u(\cdot,\cdot,5)-u_{\rm eq}\|_\infty$&$\|v(\cdot,\cdot,5)-v_{\rm eq}\|_\infty$&
$\|a(\cdot,\cdot,5)-a_{\rm eq}\|_\infty$\\ \hline
WB &6.21e-15&5.27e-15&2.21e-15&2.66e-15\\
NWB&1.25e-03&4.99e-05&3.68e-03&6.22e-15\\
\hline
\end{tabular}
\end{center}
\caption{\sf Example 5 (capturing the steady state): Errors for the WB and NWB schemes.\label{tab44}}
\end{table}
Next, we examine the ability of the proposed WB scheme to capture a small perturbation of the studied steady state accurately. This is done
by perturbing the equilibrium fluid depth, namely, we take the following initial data for $h$:
\begin{equation*}
h(x,y,0)=h_{\rm eq}(x,y)+\begin{cases}0.05&\mbox{if }\sqrt{(x-2)^2+(y-2)^2}<\frac{1}{4},\\0&\textrm{otherwise}.\end{cases}
\end{equation*}

We compute the numerical solutions by both the WB and NWB schemes until the final time $t=1$ on a sequence of $100\times100$,
$200\times200$, and $400\times400$ uniform meshes. The obtained results ($h(x,y,1)-h_{\rm eq}(x,y)$) are plotted in Figure \ref{fig412},
where one can see that the NWB scheme fails to capture the correct solution on a coarse $100\times100$ mesh and even when the mesh is
refined the NWB solution contains visible oscillations. At the same time, the WB solution is oscillation-free even when on a coarse mesh. In
order to better illustrate the difference between the WB and NWB results, we also plot two 1-D slices $h(x,2,1)-h_{\rm eq}(x,2)$ and
$h(2,y,1)-h_{\rm eq}(2,y)$; see Figure \ref{fig413}.
\begin{figure}[ht!]
\centerline{\includegraphics[trim=1.0cm 0.6cm 0.3cm 1.2cm, clip, width=5.75cm]{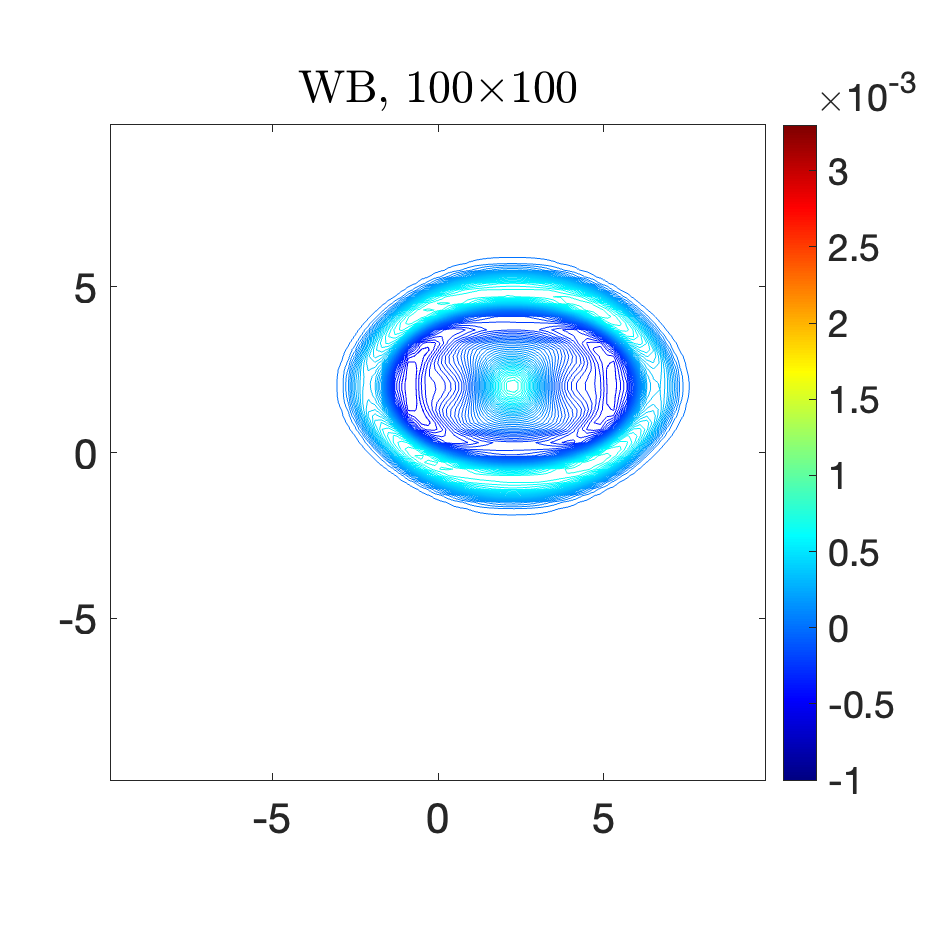}
            \includegraphics[trim=1.0cm 0.6cm 0.3cm 1.2cm, clip, width=5.75cm]{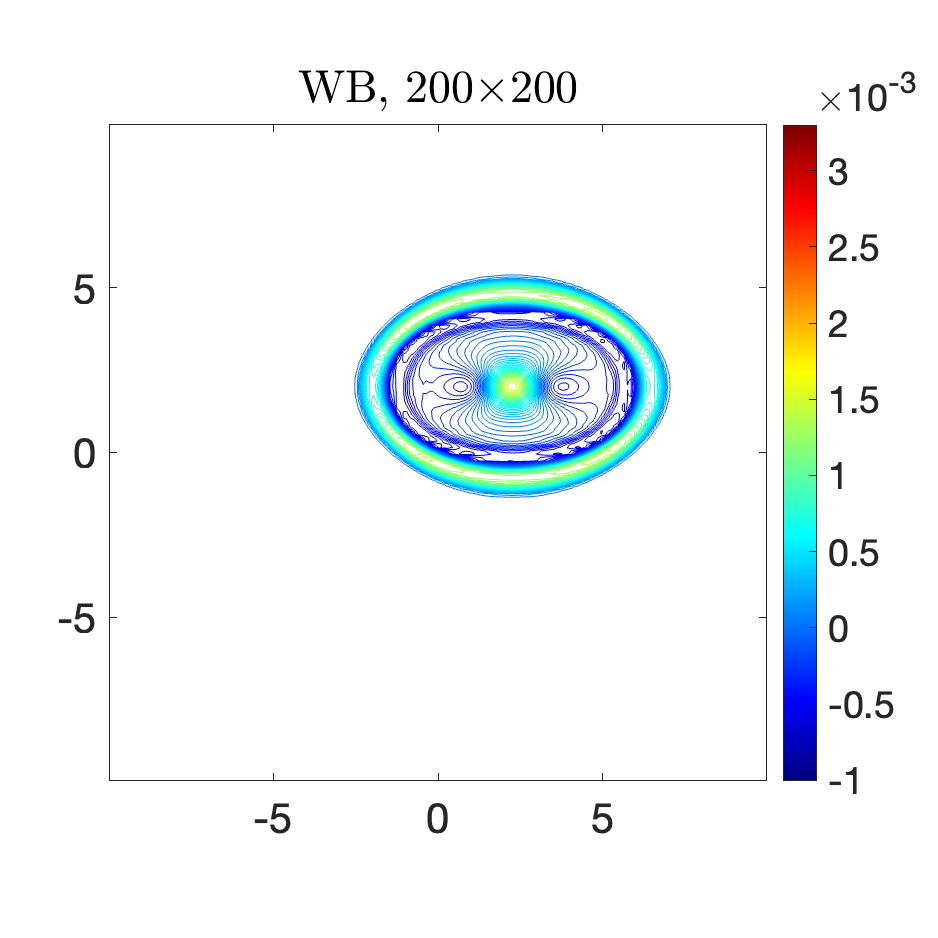}
            \includegraphics[trim=1.0cm 0.6cm 0.3cm 1.2cm, clip, width=5.75cm]{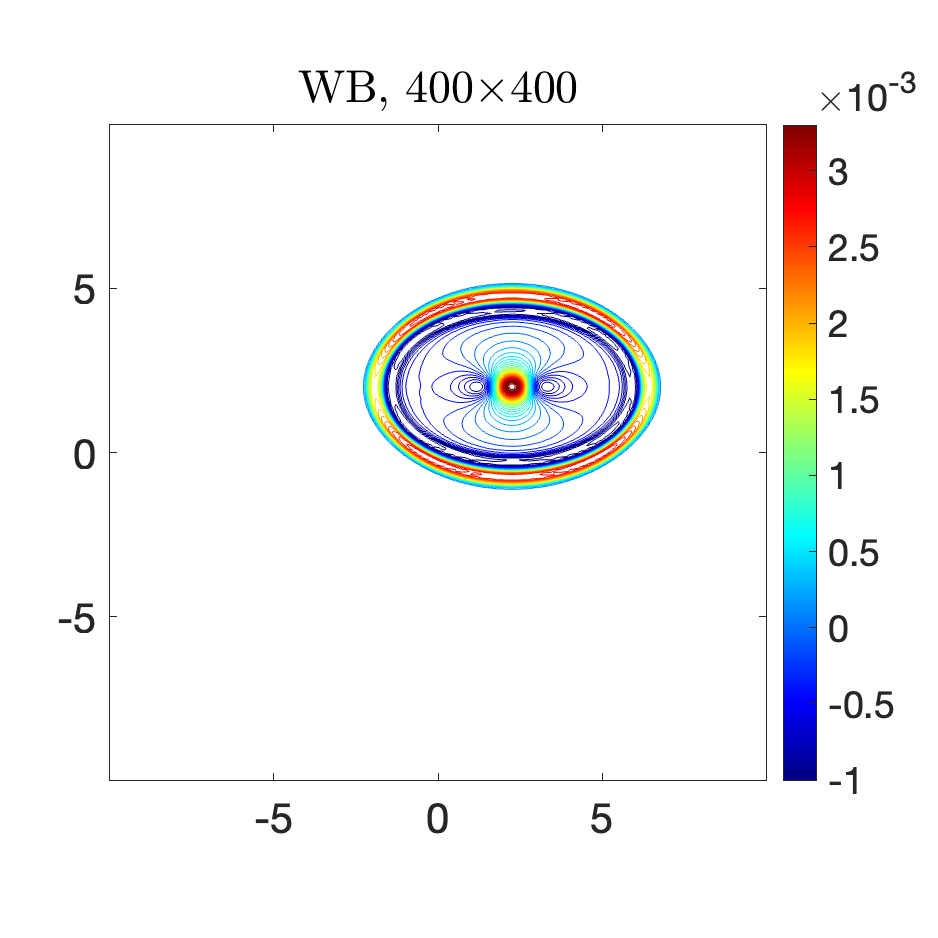}}
\centerline{\includegraphics[trim=1.0cm 0.6cm 0.3cm 1.2cm, clip, width=5.75cm]{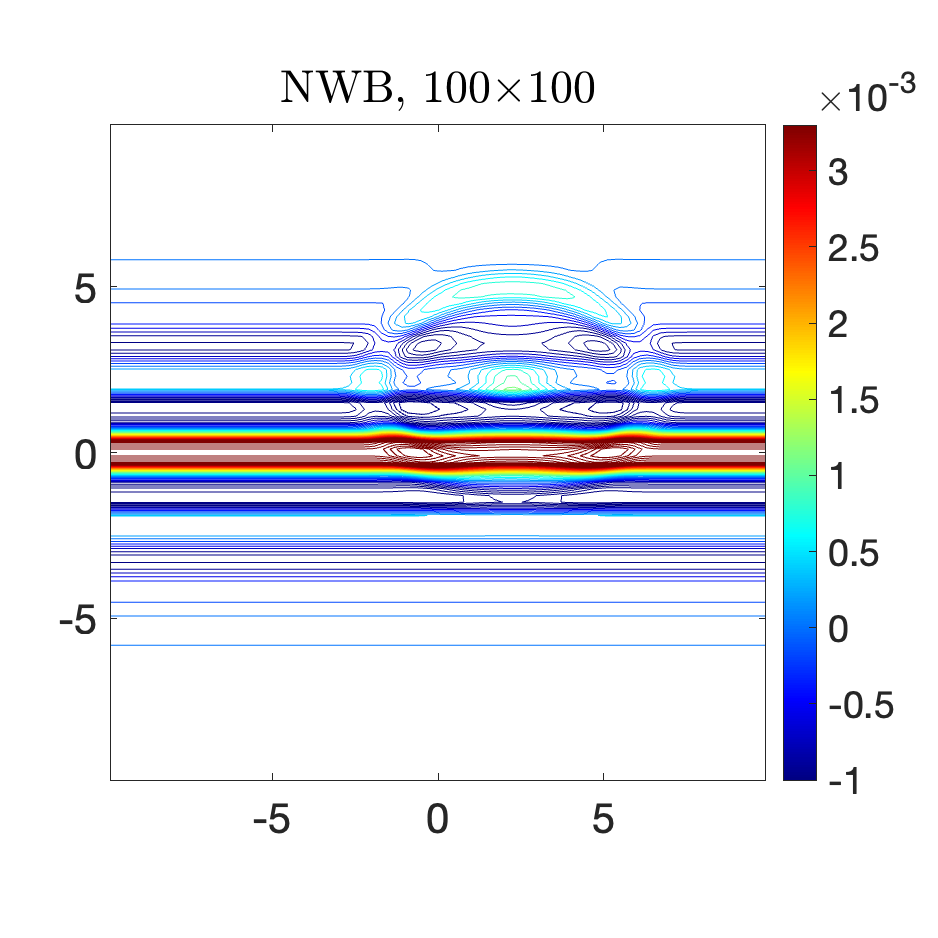}
            \includegraphics[trim=1.0cm 0.6cm 0.3cm 1.2cm, clip, width=5.75cm]{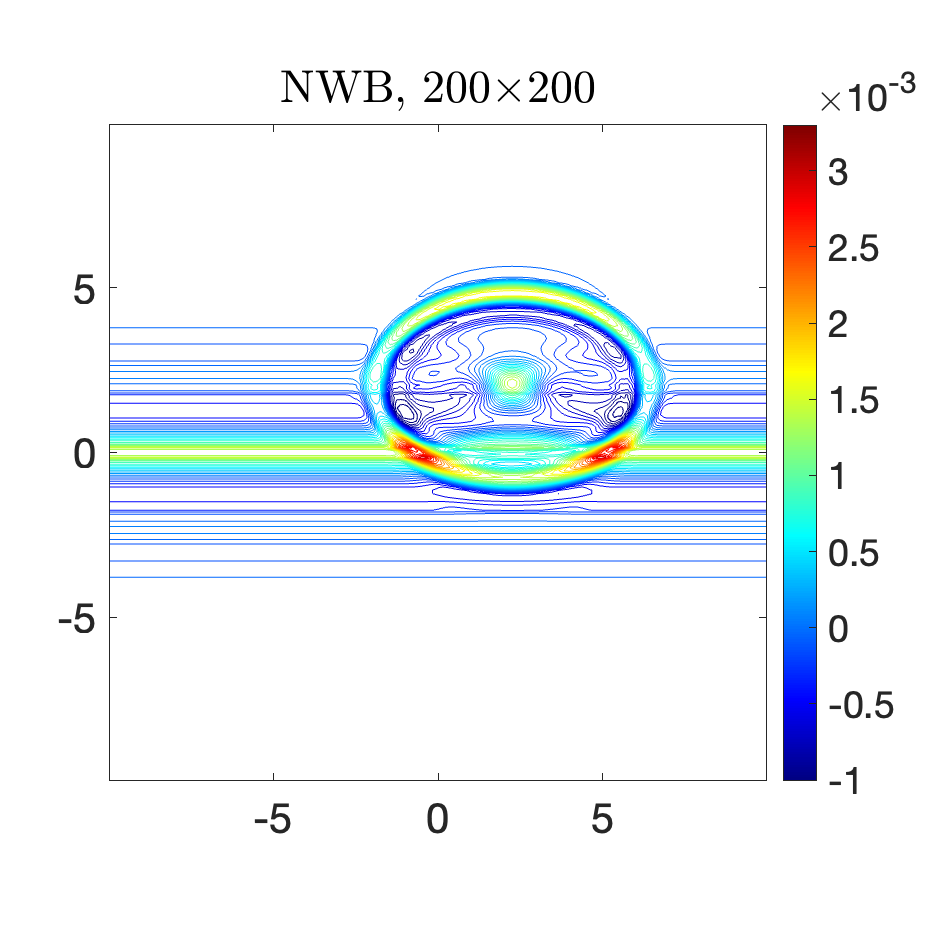}
            \includegraphics[trim=1.0cm 0.6cm 0.3cm 1.2cm, clip, width=5.75cm]{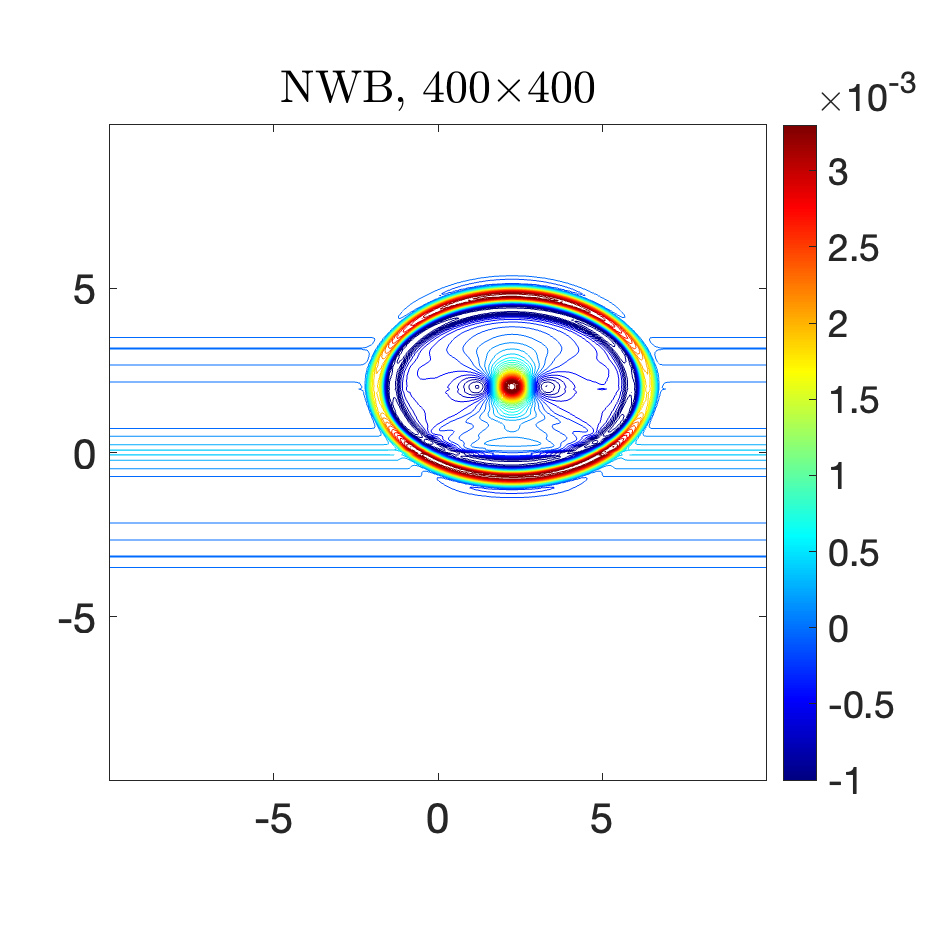}}
\caption{\sf Example 5 (small perturbation of the steady state): $h(x,y,1)-h_{\rm eq}(x,y)$ computed by both the WB (top row) and NWB
(bottom row) schemes on $100\times100$ (left column), $200\times200$ (middle column), and $400\times400$ (right column) uniform meshes.
\label{fig412}}
\end{figure}
\begin{figure}[ht!]
\centerline{\includegraphics[trim=1.2cm 0.8cm 1.1cm 0.3cm, clip, width=5.4cm]{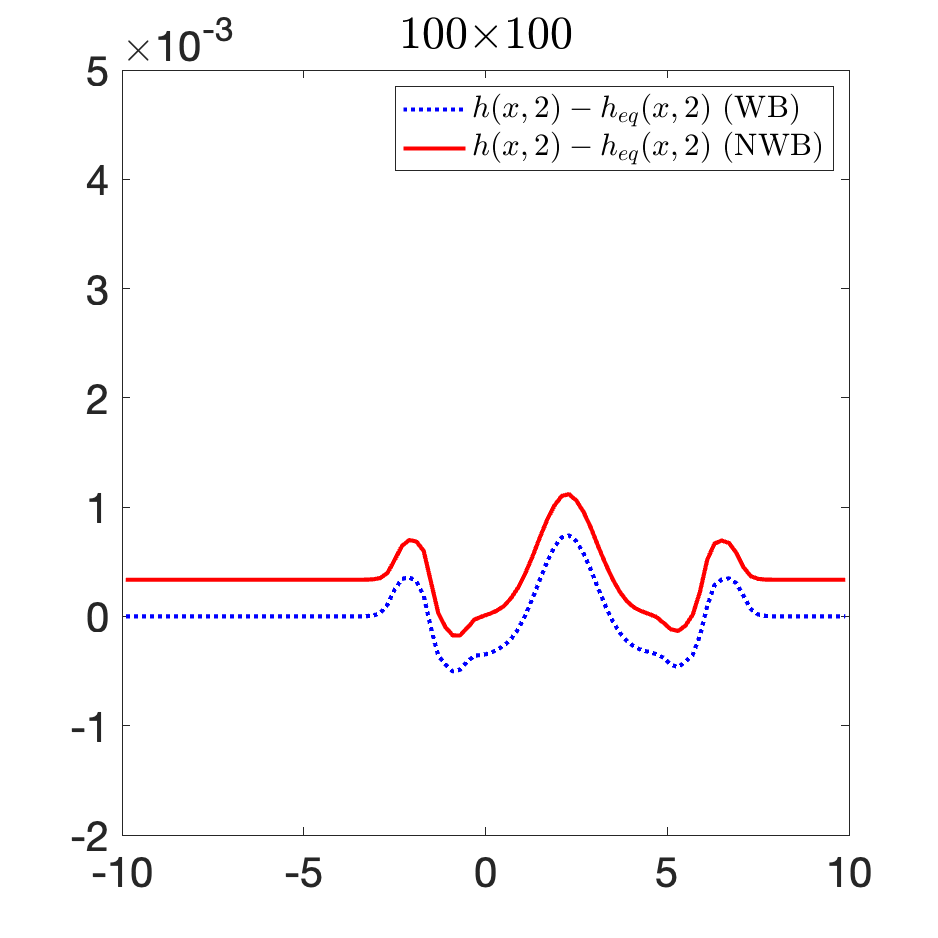}\hspace*{0.4cm}
            \includegraphics[trim=1.2cm 0.8cm 1.1cm 0.3cm, clip, width=5.4cm]{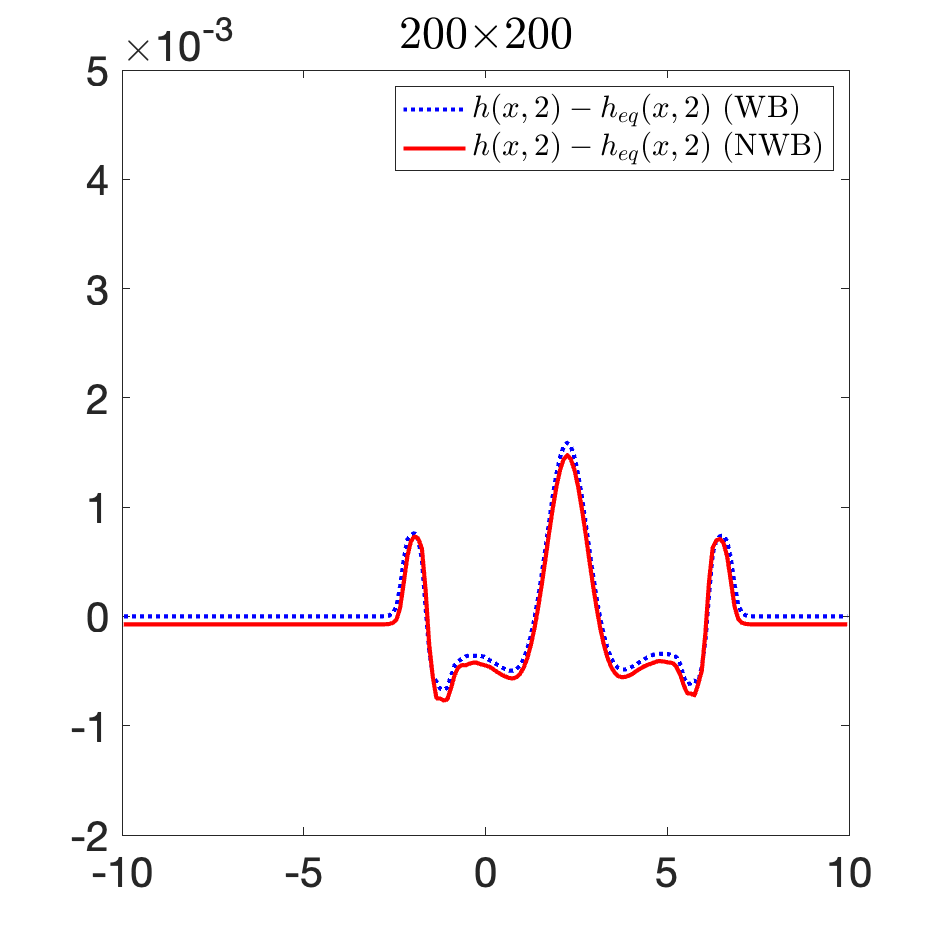}\hspace*{0.4cm}
            \includegraphics[trim=1.2cm 0.8cm 1.1cm 0.3cm, clip, width=5.4cm]{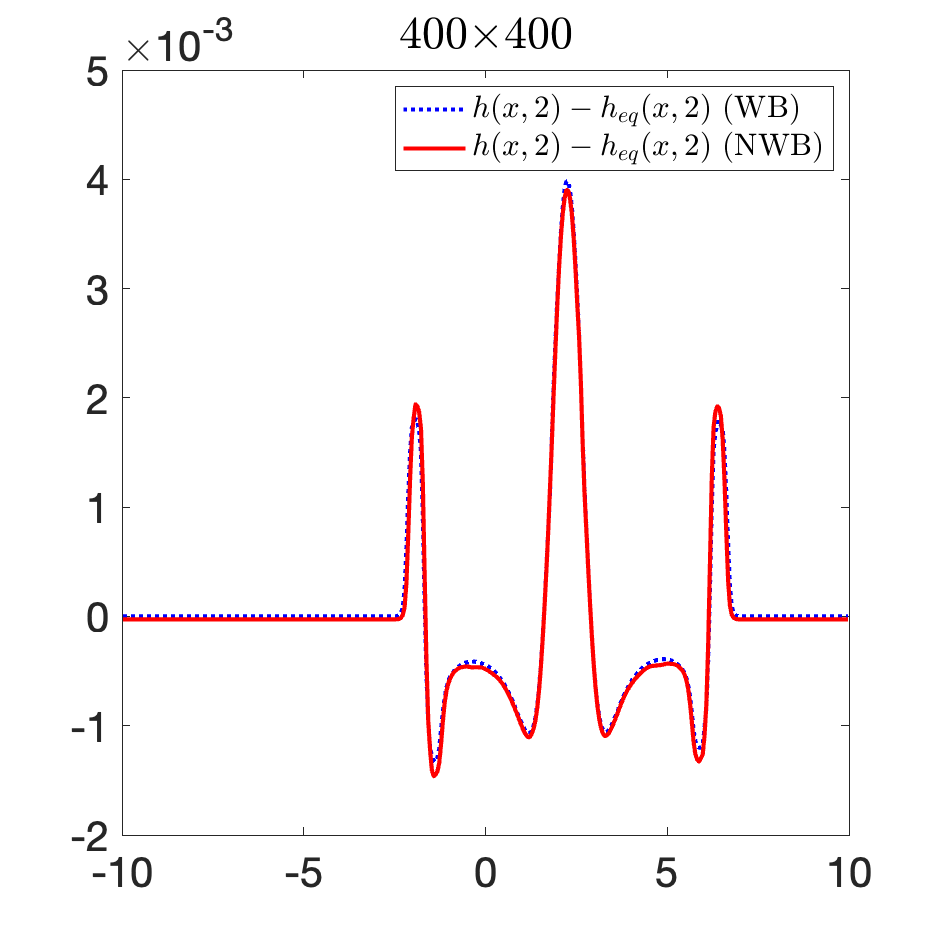}}
\vskip7pt
\centerline{\includegraphics[trim=1.2cm 0.8cm 1.1cm 0.3cm, clip, width=5.4cm]{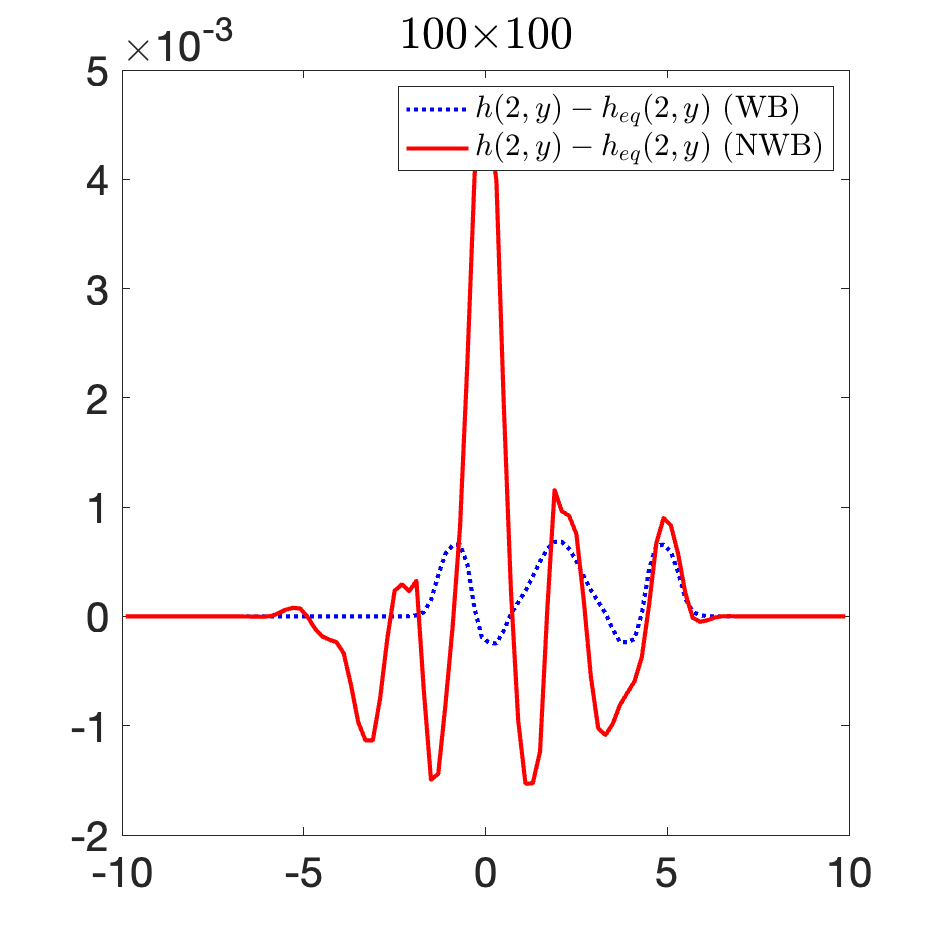}\hspace*{0.4cm}
            \includegraphics[trim=1.2cm 0.8cm 1.1cm 0.3cm, clip, width=5.4cm]{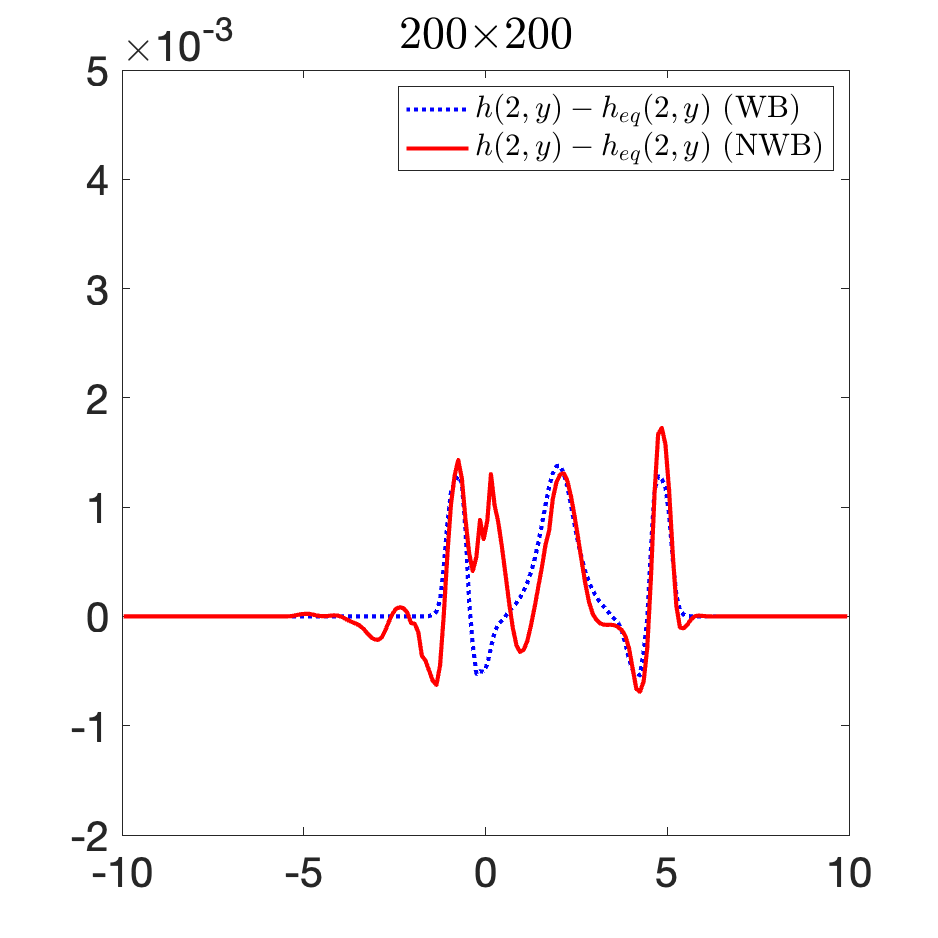}\hspace*{0.4cm}
            \includegraphics[trim=1.2cm 0.8cm 1.1cm 0.3cm, clip, width=5.4cm]{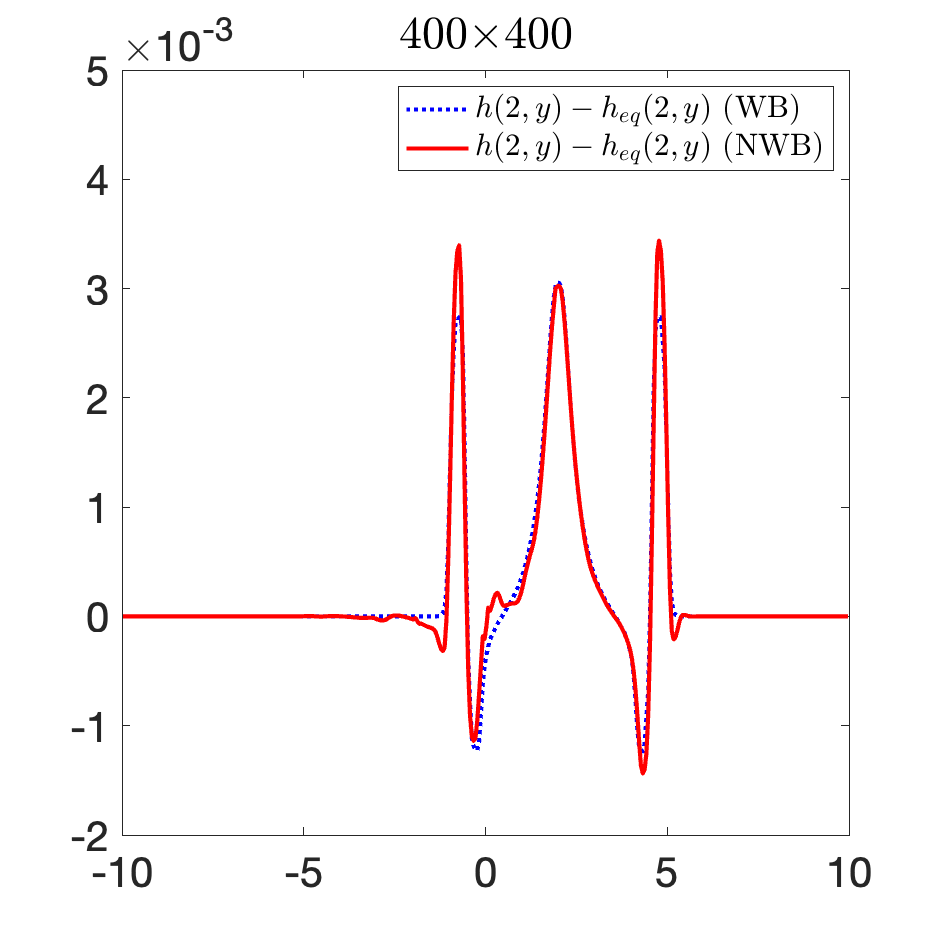}}
\caption{\sf Example 5 (small perturbation of the steady state): 1-D slices $h(x,2,1)-h_{\rm eq}(x,2)$ (top row) and
$h(2,y,1)-h_{\rm eq}(2,y)$ (bottom row) computed by both the WB and NWB schemes on $100\times100$ (left column), $200\times200$ (middle
column), and $400\times400$ (right column) uniform meshes.\label{fig413}}
\end{figure}

\subsubsection*{Example 6---Magneto-Cyclo-Geostrophic Adjustment of a Circular Magnetic Anomaly}
In this example, we test the process of magneto-cyclo-geostrophic adjustment of an initial configuration with a circular magnetic field
only, that is, with a flat $h$ and zero velocity. According to the mechanism of magneto-cyclo-geostrophic adjustment explained in
\S\ref{sec421}, it is expected to evolve towards a vortex in magneto-cyclo-geostrophic equilibrium by emitting outward-traveling
inertia-gravity waves. Notice that there are no Alfv\'en waves that could propagate in this direction.

We consider the following initial conditions:
\begin{equation*}
h(x,y,0)\equiv1,\quad u(x,y,0)=v(x,y,0)\equiv0,\quad a(x,y,0)=2ye^{-(x^2+y^2)},\quad b(x,y,0)=-2xe^{-(x^2+y^2)},
\end{equation*}
with the constant Coriolis parameter $f(y)\equiv1$ and flat bottom topography $Z(y)\equiv0$ on the computational domain
$[-10,10]\times[-10,10]$ subject to the outflow boundary conditions.

We use the WB scheme to compute the solution until $t=8$ on a $400\times400$ uniform mesh and plot the obtained $h$, $|\bm u|$, and
$|\bm b|$ in Figure \ref{fig414}, where one can observe the expected structures: a circular wave train of inertia-gravity waves and a
central vortex. We also present four-time snapshots at $t=2$, 4, 6, and 8 of the vorticity $\zeta:=v_x-u_y$ and velocity divergence
$u_x+v_y$ in Figures \ref{fig415} and \ref{fig416}, respectively. They confirm the presence of the central vortex and outgoing
inertia-gravity waves.
\begin{figure}[ht!]
\centerline{\includegraphics[trim=1.0cm 1.7cm 0.6cm 1.1cm, clip, width=5.75cm]{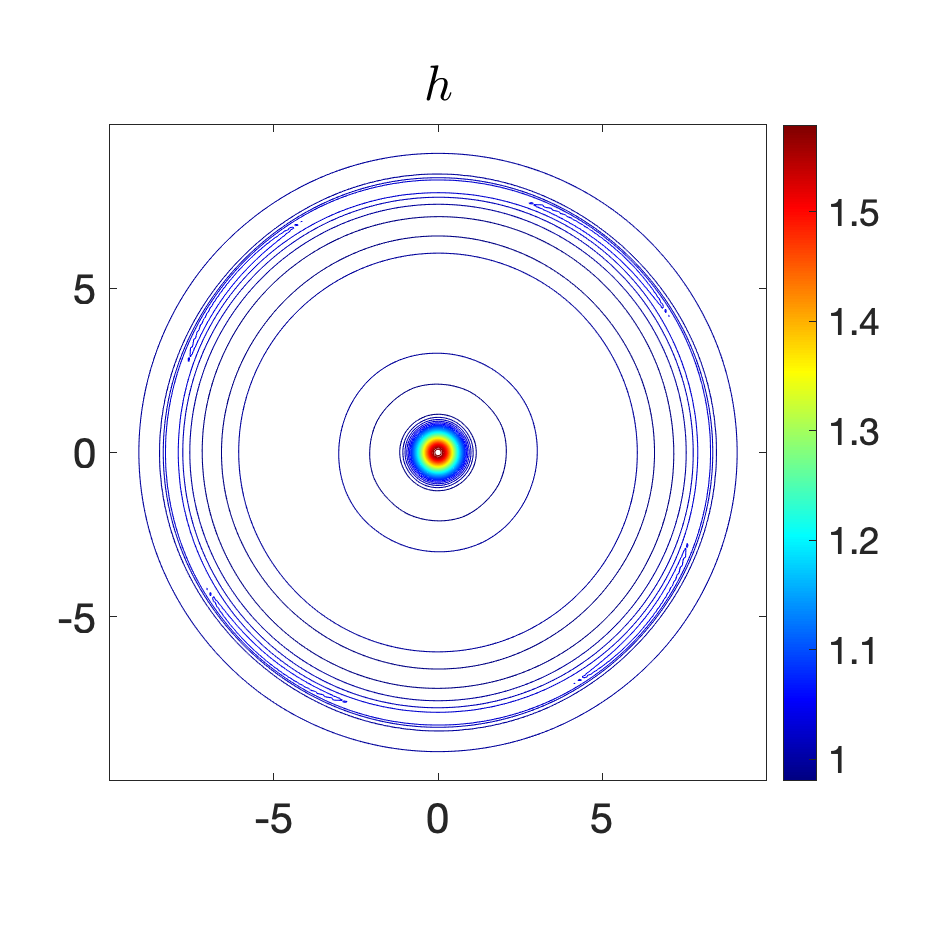}
            \includegraphics[trim=1.0cm 1.7cm 0.6cm 1.1cm, clip, width=5.75cm]{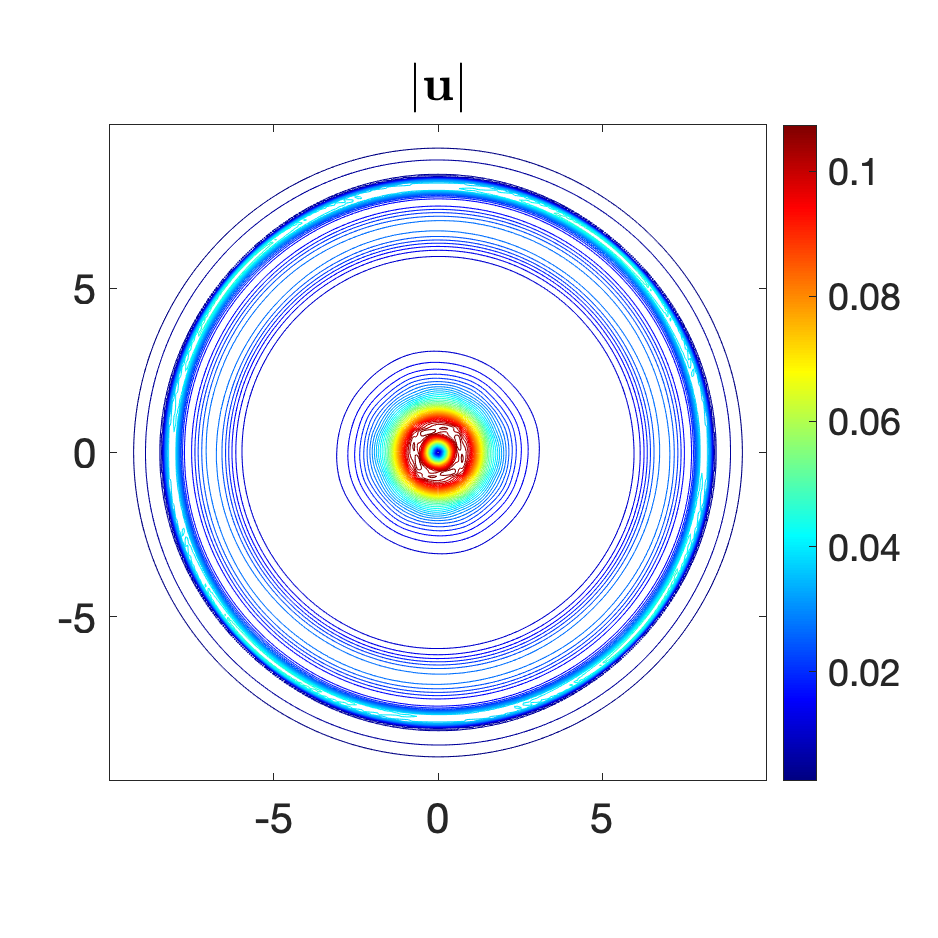}
            \includegraphics[trim=1.0cm 1.7cm 0.6cm 1.1cm, clip, width=5.75cm]{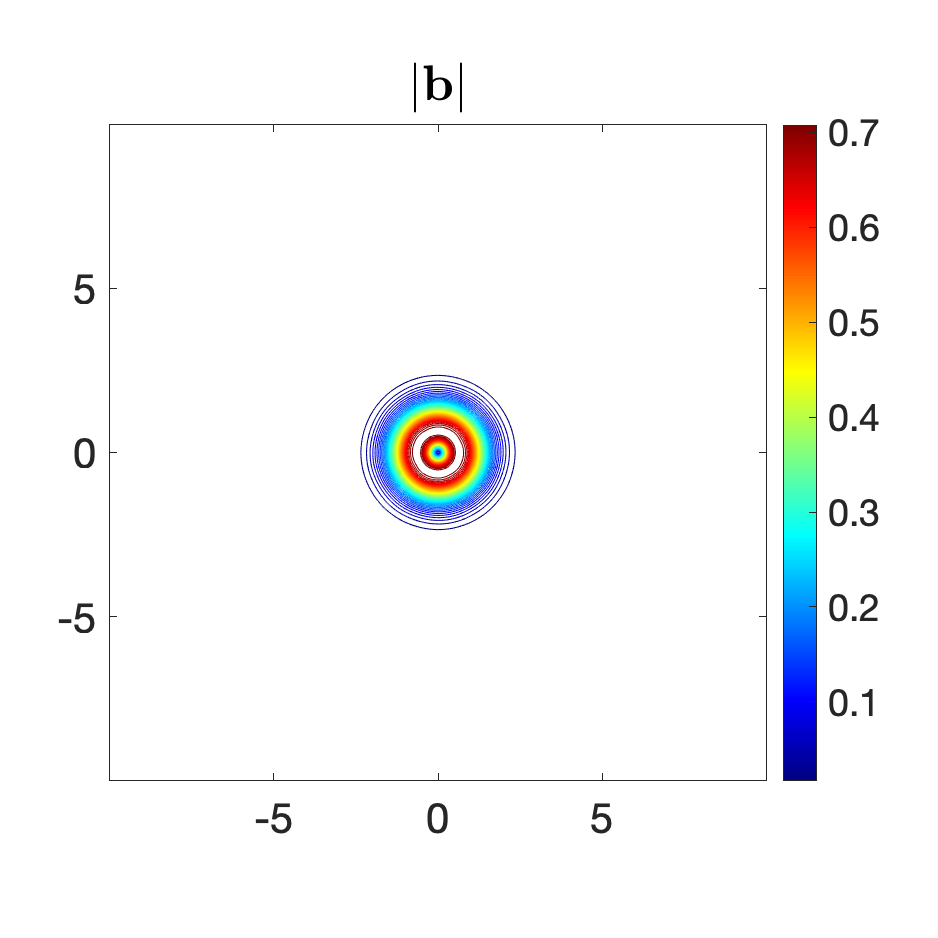}}
\caption{\sf Example 6: Contour plots of $h$, $|\bm u|$, and $|\bm b|$ with 40 equally spaced contours each at $t=8$.\label{fig414}}
\end{figure}
\begin{figure}[ht!]
\centerline{\includegraphics[trim=1.0cm 1.7cm 0.8cm 1.1cm, clip, width=4.4cm]{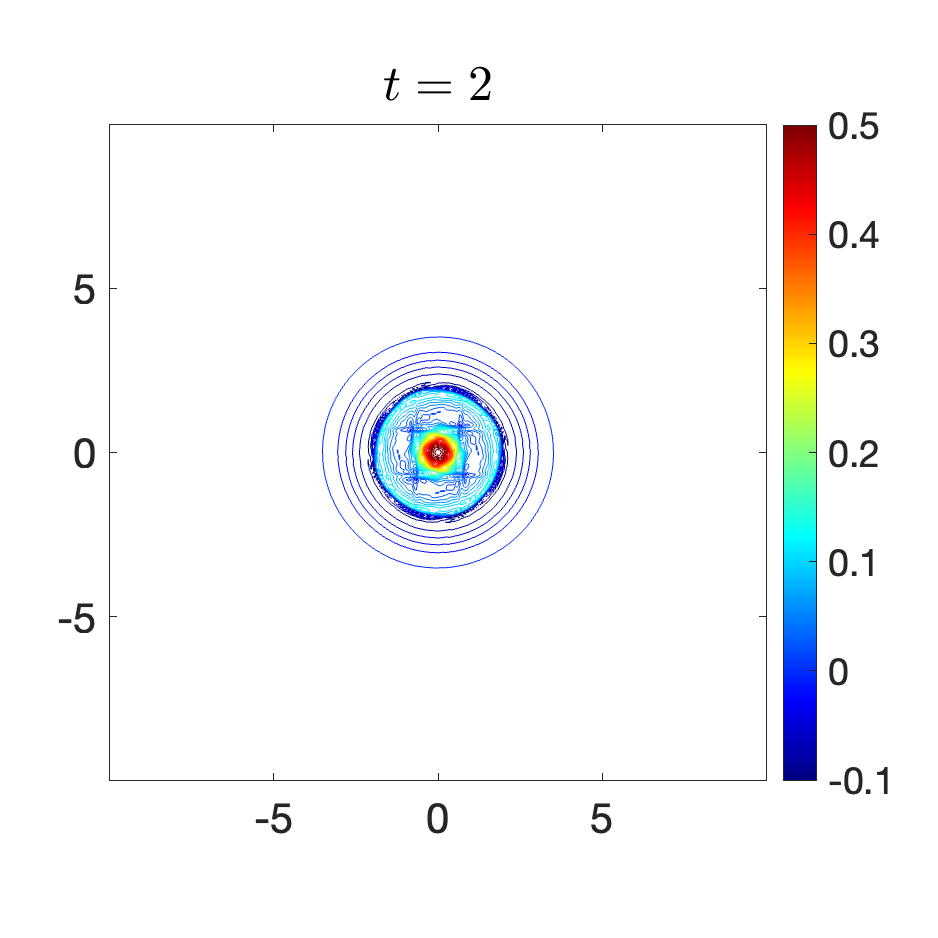}
            \includegraphics[trim=1.0cm 1.7cm 0.8cm 1.1cm, clip, width=4.4cm]{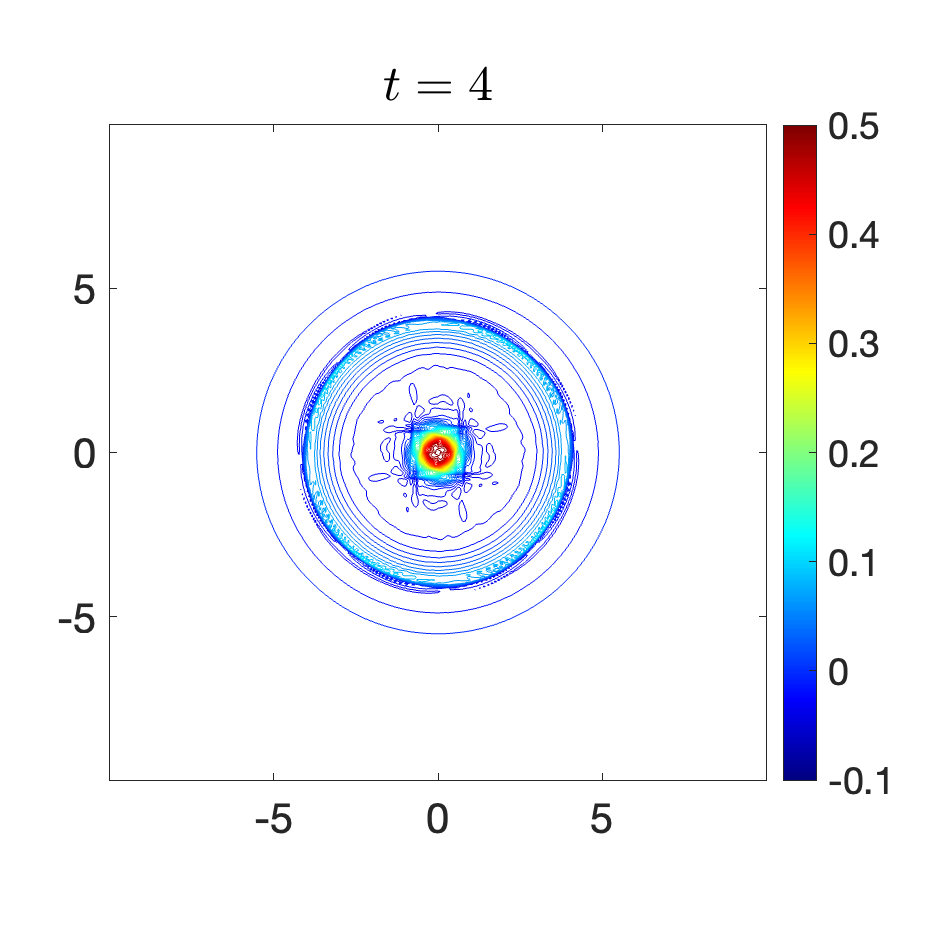}
            \includegraphics[trim=1.0cm 1.7cm 0.8cm 1.1cm, clip, width=4.4cm]{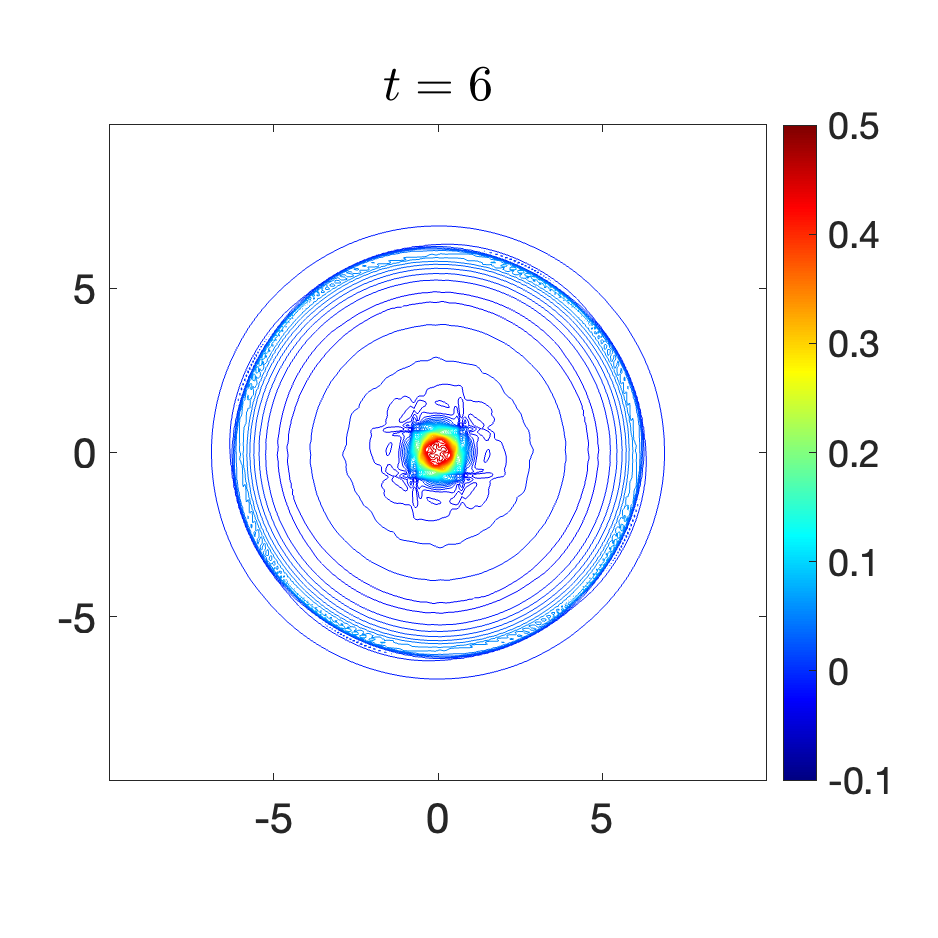}
            \includegraphics[trim=1.0cm 1.7cm 0.8cm 1.1cm, clip, width=4.4cm]{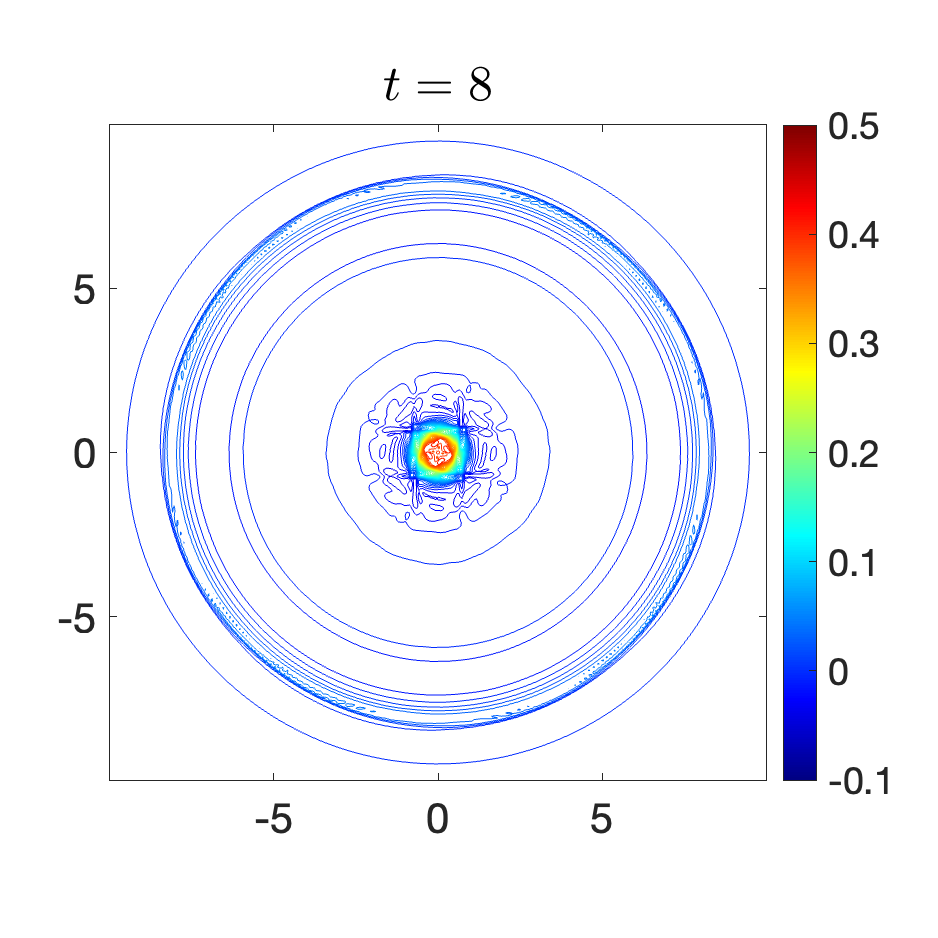}}
\caption{\sf Example 6: Time snapshots of the vorticity $\zeta$: contour plots with 40 equally spaced contours each.\label{fig415}}
\end{figure}
\begin{figure}[ht!]
\centerline{\includegraphics[trim=1.0cm 1.7cm 0.4cm 1.1cm, clip, width=4.4cm]{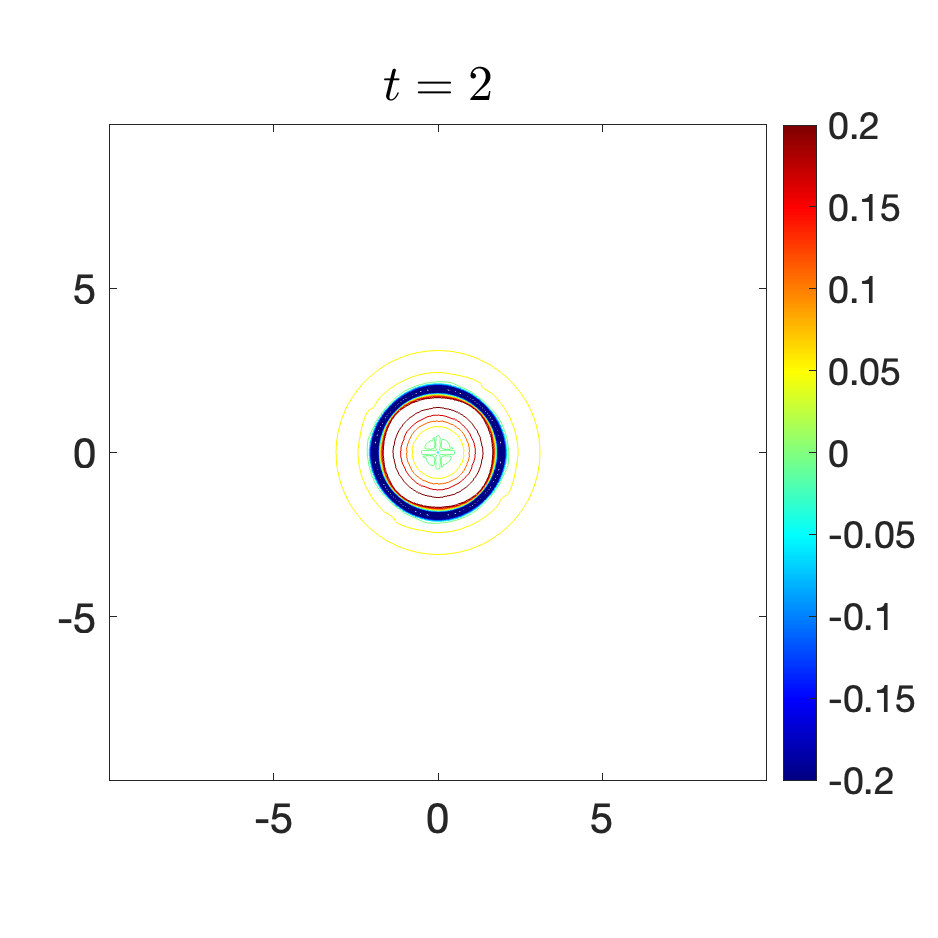}
            \includegraphics[trim=1.0cm 1.7cm 0.4cm 1.1cm, clip, width=4.4cm]{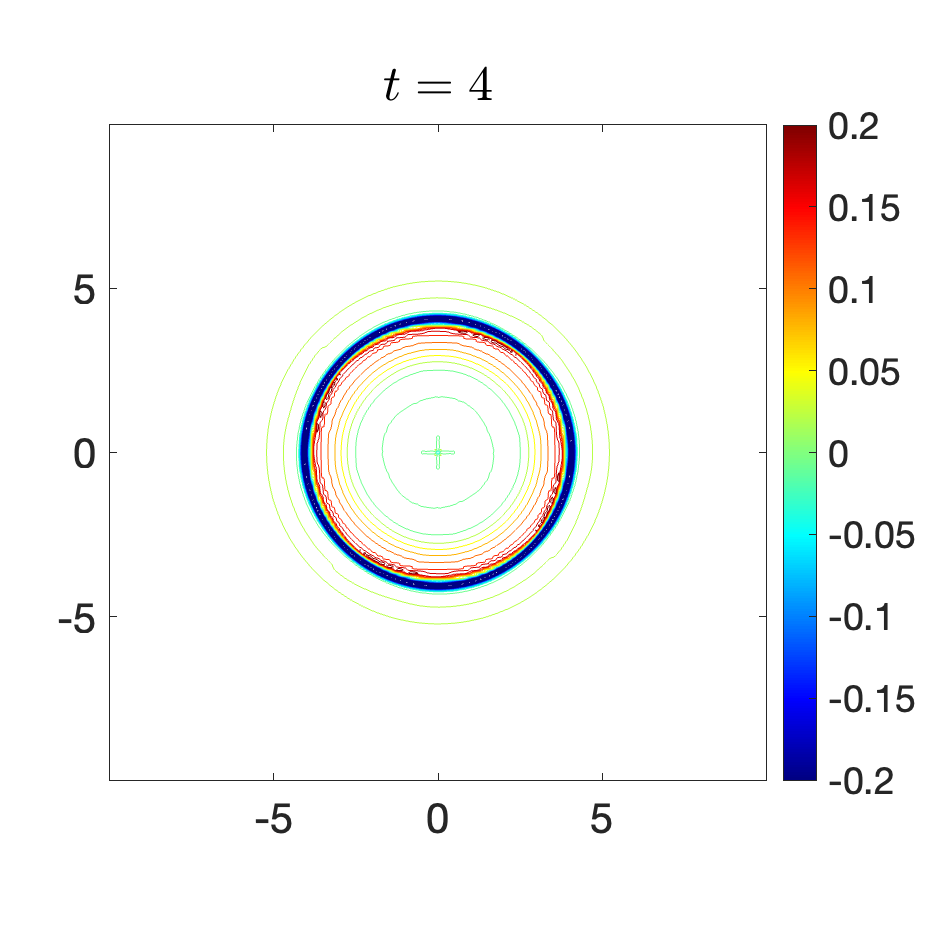}
            \includegraphics[trim=1.0cm 1.7cm 0.4cm 1.1cm, clip, width=4.4cm]{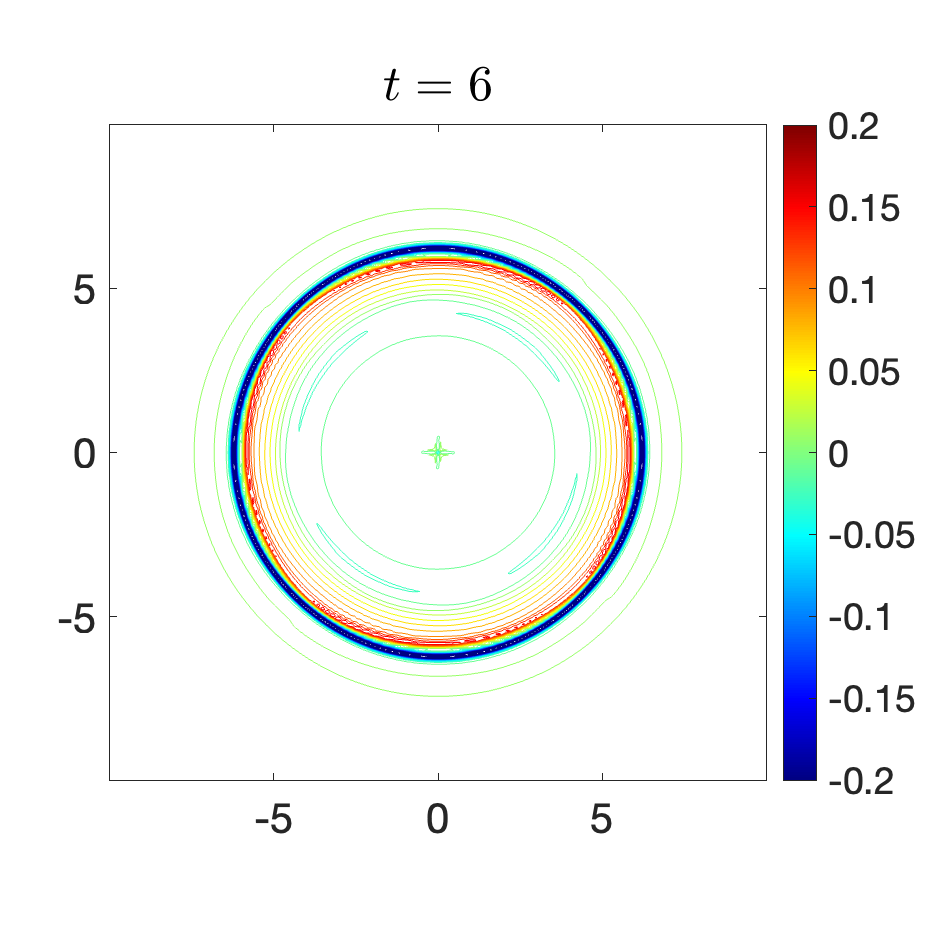}
            \includegraphics[trim=1.0cm 1.7cm 0.4cm 1.1cm, clip, width=4.4cm]{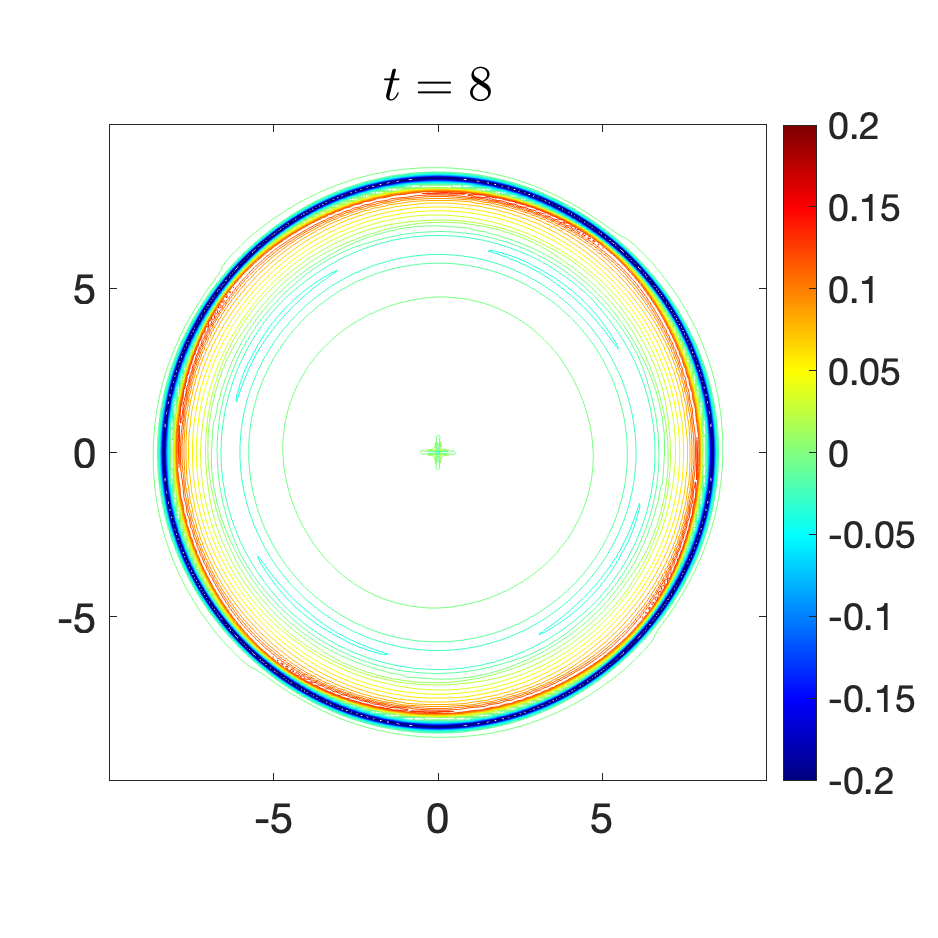}}
\caption{\sf Example 6: Time snapshots of the velocity divergence $u_x+v_y$: contour plots with 40 equally spaced contours each.
\label{fig416}}
\end{figure}

\subsubsection*{Example 7---Magneto-Cyclo-Geostrophic Adjustment of a Balanced Magnetized Vortex}
In this example, we continue to test the fundamental process of magneto-cyclo-geostrophic adjustment by adding the effects of topography and
exploring the influence of a magnetic field imposed onto an initially balanced vortex. We take a constant Coriolis parameter $f(y)\equiv2$, a
Gaussian-shaped axisymmetric bottom topography $Z=\frac{1}{20}e^{-r^2}$, and consider the following initial conditions that correspond to a
magnetized balanced vortex:
\begin{equation*}
\begin{aligned}
&h(x,y,0)\equiv1,\quad u(x,y,0)=-V(r)\sin\theta,\quad v(x,y,0)=V(r)\cos\theta,\\
&a(x,y,0)=-1.1e^{-r}\sin\theta,\quad b(x,y,0)=1.1e^{-r}\cos\theta,
\end{aligned}
\end{equation*} 
where the initial velocity magnitude $V(r)$ is found by solving the quadratic cyclo-geostrophic balance equation, which reads as
$$
\frac{V^2}{r}+fV=g\frac{{\rm d}}{{\rm d}r}(h+Z).
$$
We conduct the computations in the square $[-10,10]\times[-10,10]$ and set outward boundary conditions at the domain boundaries. We use the 
WB scheme to compute the solution until the final time $t=8$ on a $400\times400$ uniform mesh and plot time snapshots of the obtained fluid 
level $h+Z$ and velocity divergence $u_x+v_y$ at $t=2$, 4, 6, and 8 in Figure \ref{fig421}. We also present graphs of $|\bm u|$ and
$|\bm b|$ at the final time $t=8$ in Figure \ref{fig422}. Like in the preceding magneto-geostrophic adjustment problem (Example 6), one can
observe a circular inertia gravity wave train and the remaining central vortex. The 1-D slices $h(x,0,t)$, $|\bm u(x,0,t)|$, and
$|\bm b(x,0,t)|$ at $t=0$ and 8 are displayed in Figure \ref{fig418f}. They clearly show the evolution towards an equilibrium state
verifying \eref{magcycgeo}.
\begin{figure}[ht!]
\centerline{\includegraphics[trim=1.0cm 1.7cm 0.6cm 1.1cm, clip, width=4.4cm]{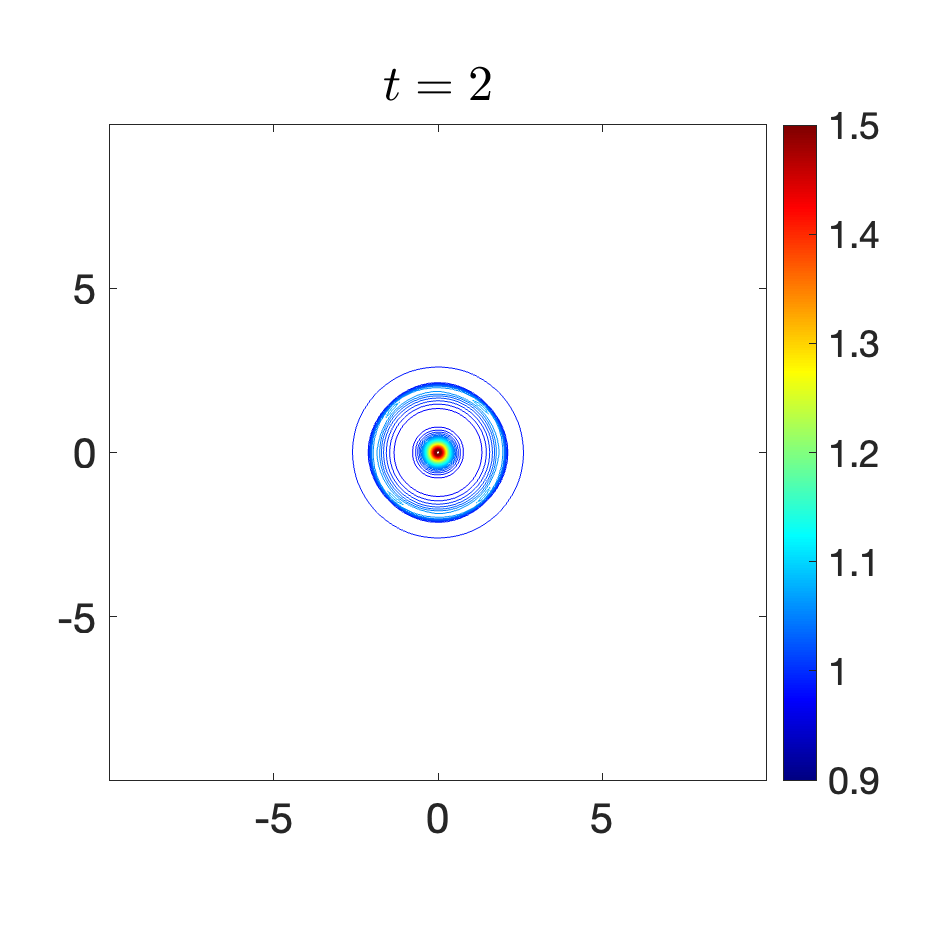}
            \includegraphics[trim=1.0cm 1.7cm 0.6cm 1.1cm, clip, width=4.4cm]{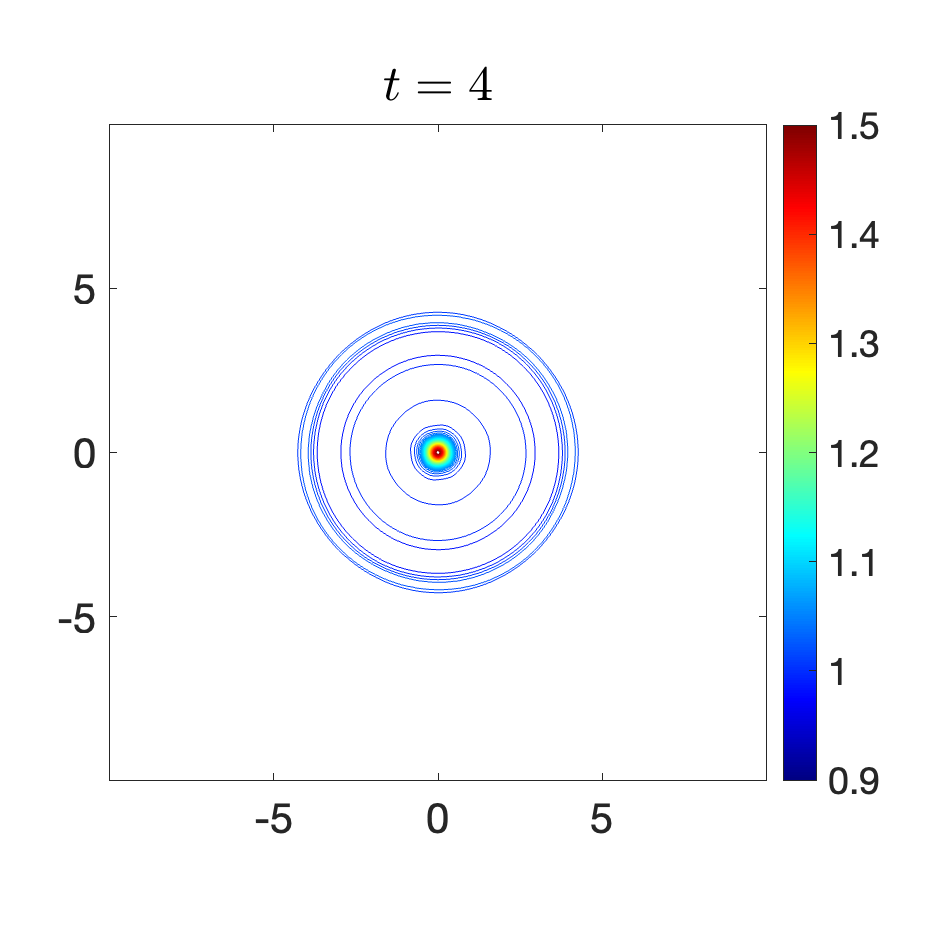}
            \includegraphics[trim=1.0cm 1.7cm 0.6cm 1.1cm, clip, width=4.4cm]{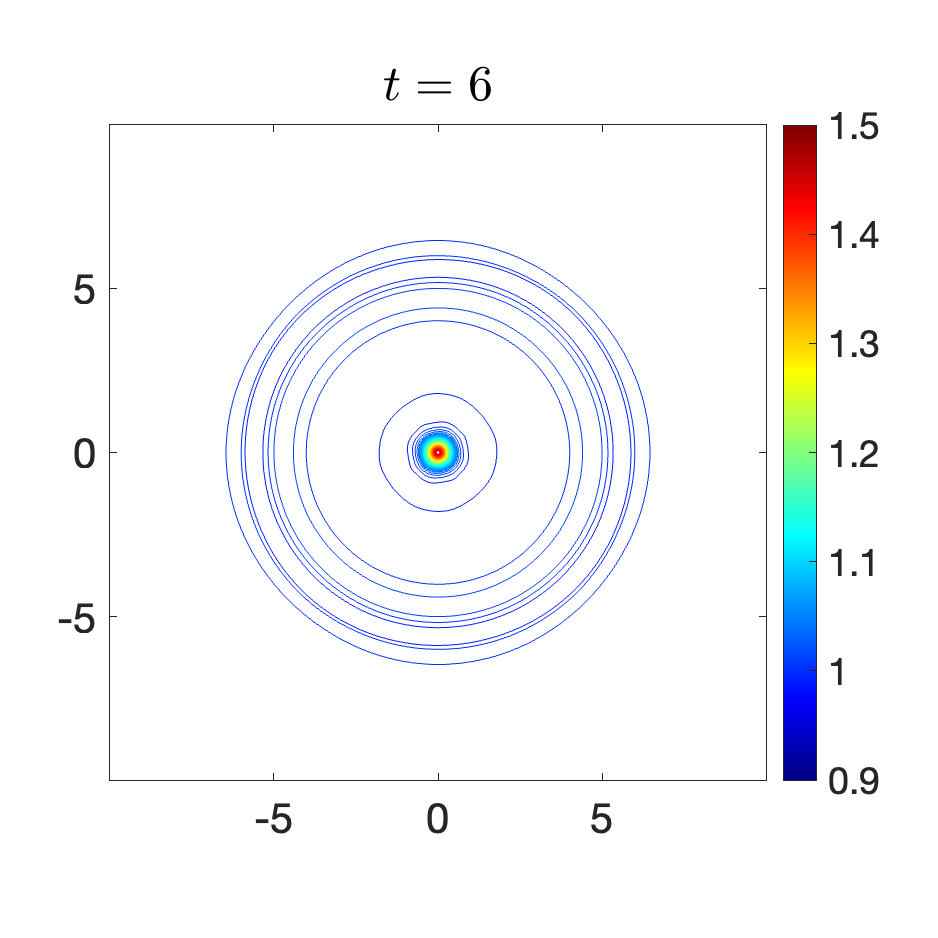}
            \includegraphics[trim=1.0cm 1.7cm 0.6cm 1.1cm, clip, width=4.4cm]{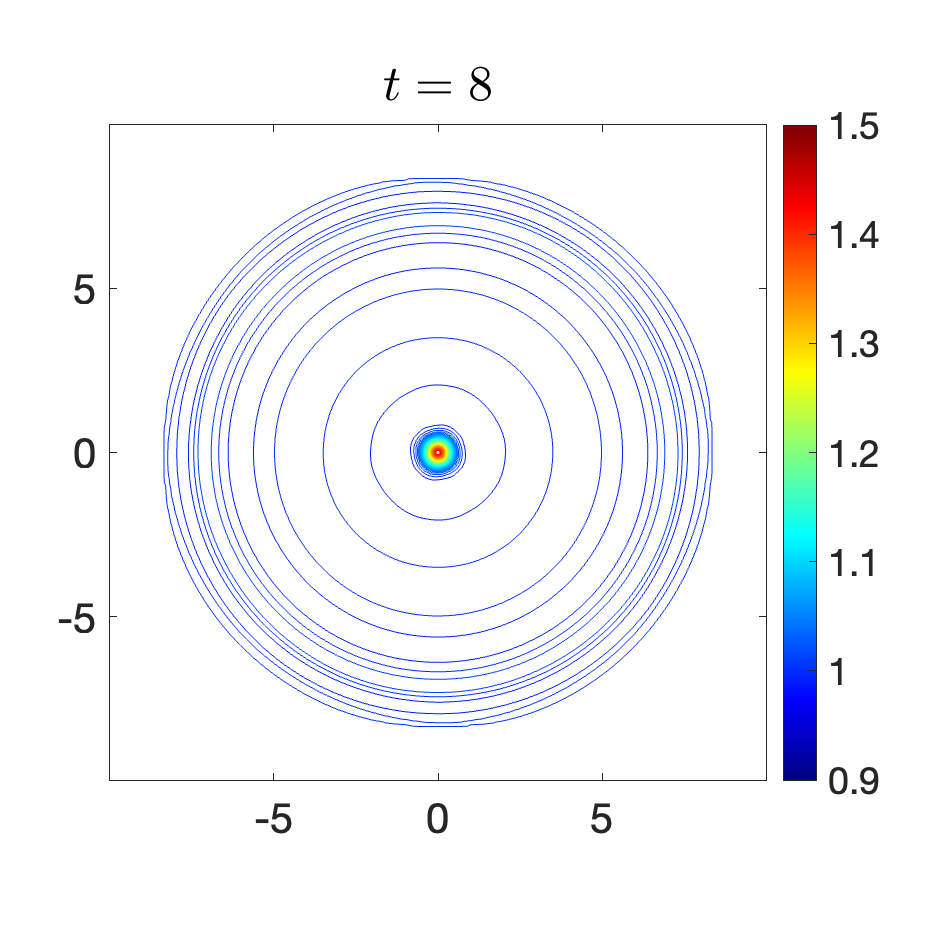}}
\vskip7pt
\centerline{\includegraphics[trim=1.0cm 1.7cm 0.4cm 1.1cm, clip, width=4.45cm]{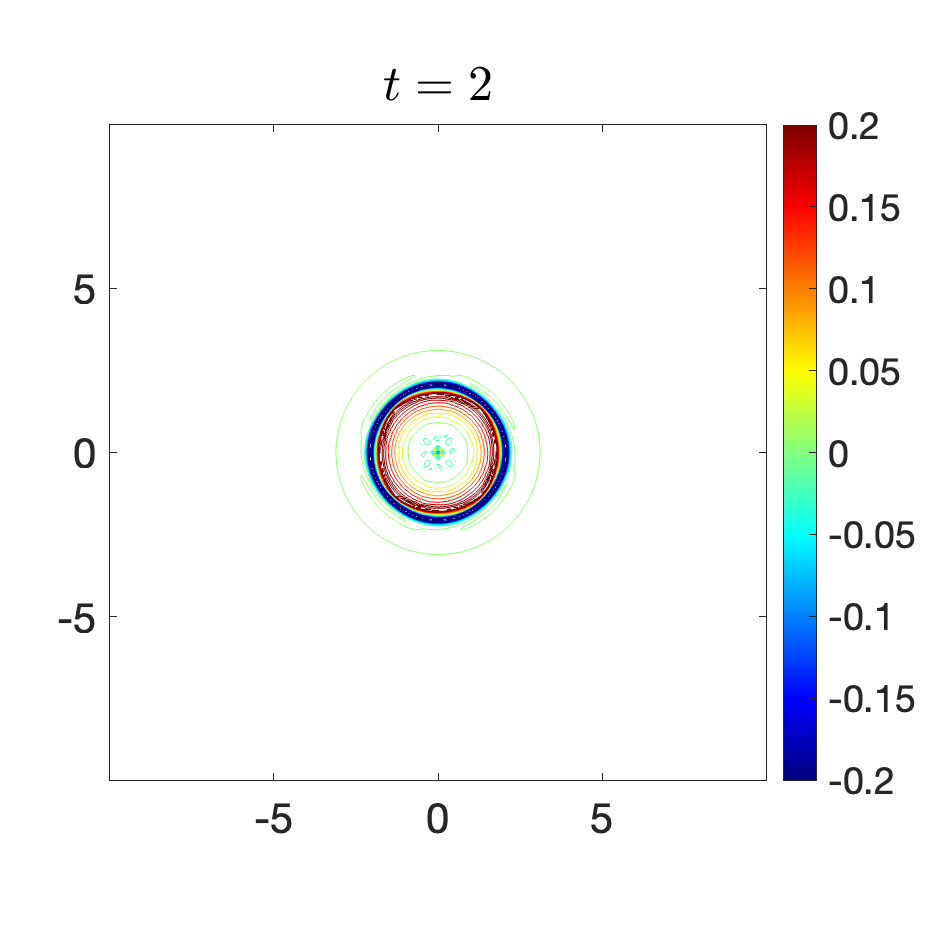}
            \includegraphics[trim=1.0cm 1.7cm 0.4cm 1.1cm, clip, width=4.45cm]{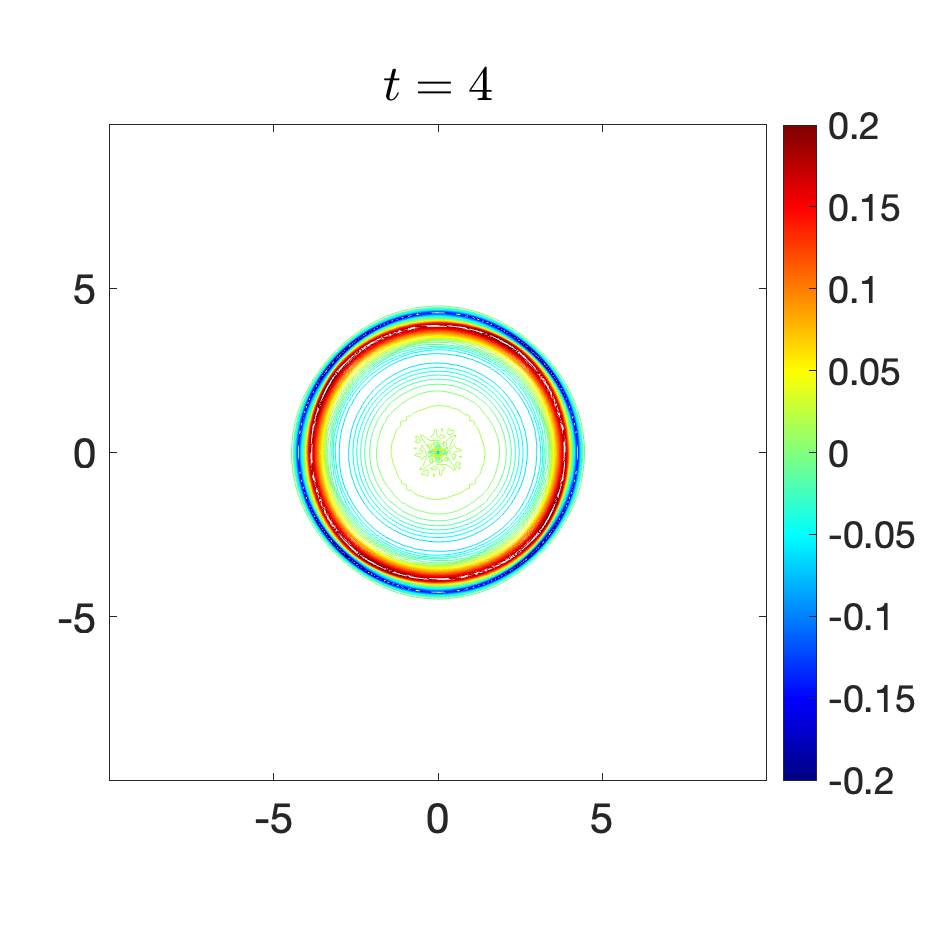}
            \includegraphics[trim=1.0cm 1.7cm 0.4cm 1.1cm, clip, width=4.45cm]{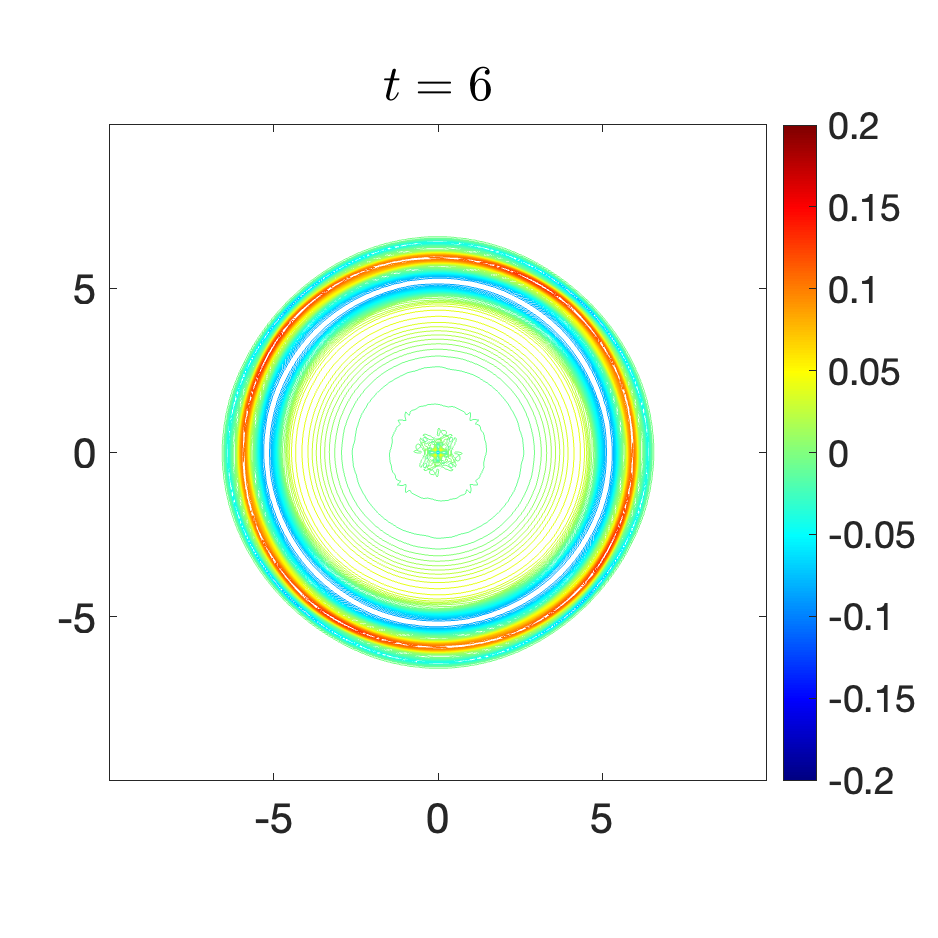}
            \includegraphics[trim=1.0cm 1.7cm 0.4cm 1.1cm, clip, width=4.45cm]{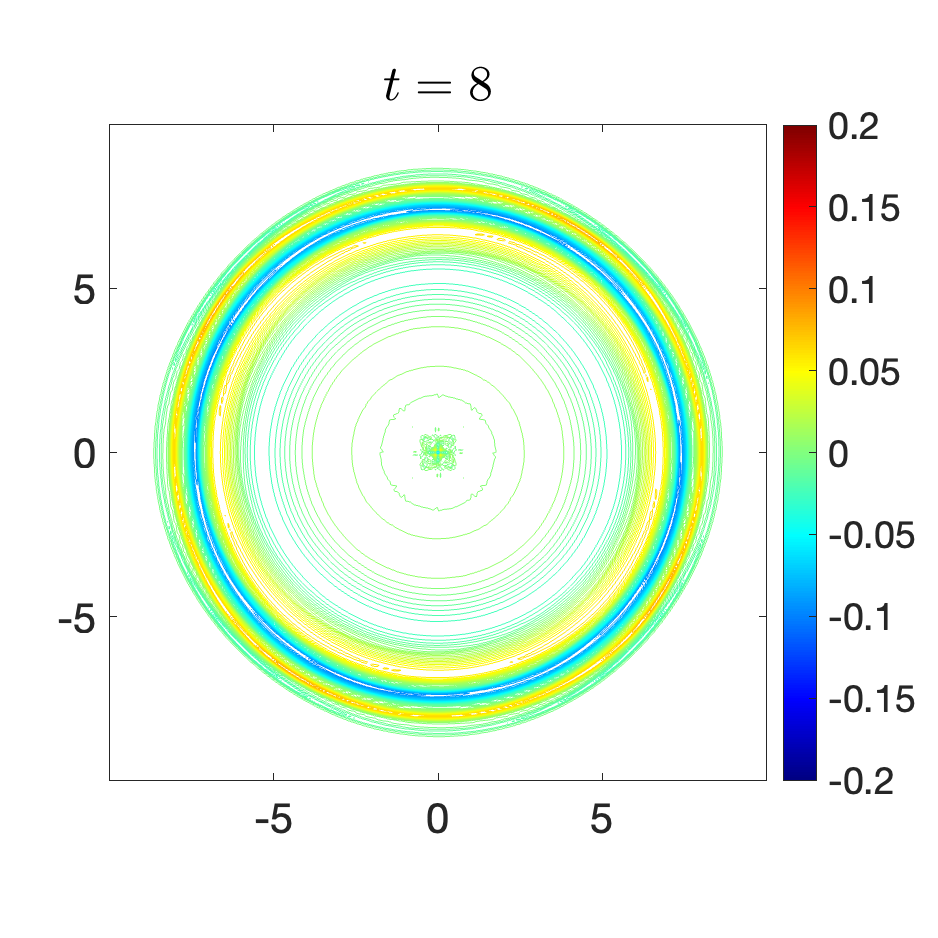}}
\caption{\sf Example 7: Time snapshots of the fluid level $h+Z$ (top row) and velocity divergence $u_x+v_y$ (bottom row): contour plots with
40 equally spaced contours each.\label{fig421}}
\end{figure}
\begin{figure}[ht!]
\centerline{\includegraphics[trim=1.0cm 1.7cm 0.6cm 1.1cm, clip, width=5.75cm]{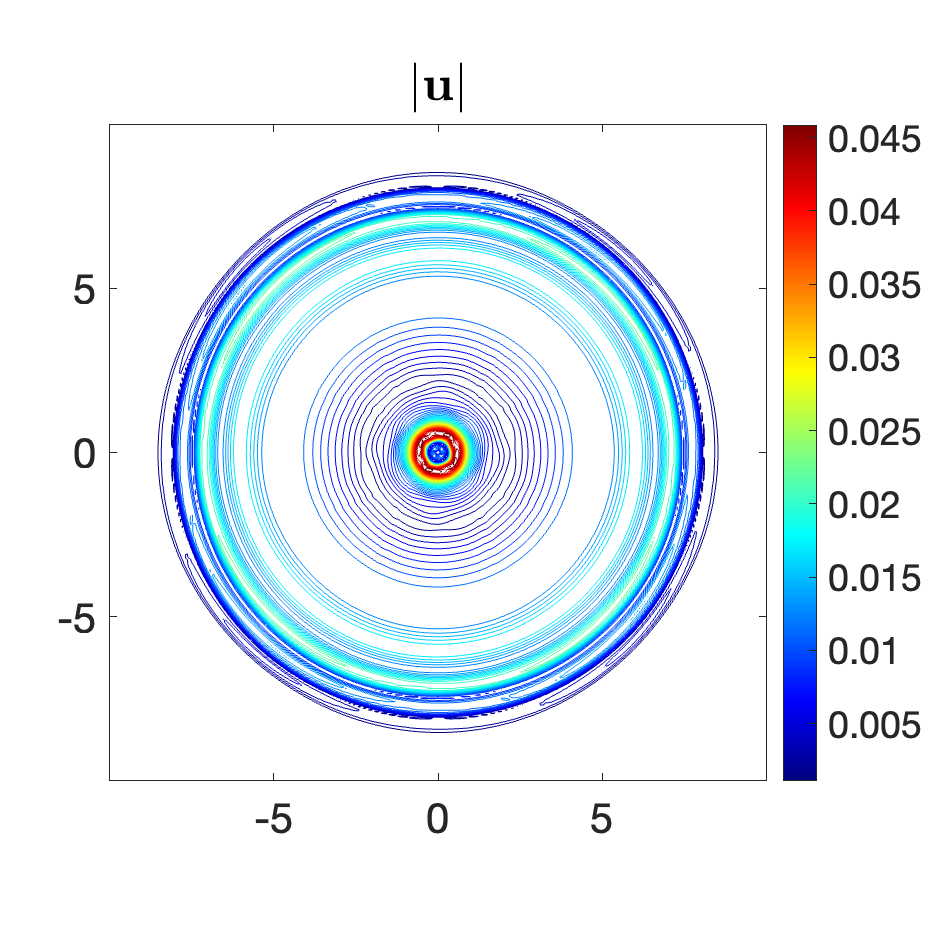}\hspace*{0.5cm}
            \includegraphics[trim=1.0cm 1.7cm 0.6cm 1.1cm, clip, width=5.75cm]{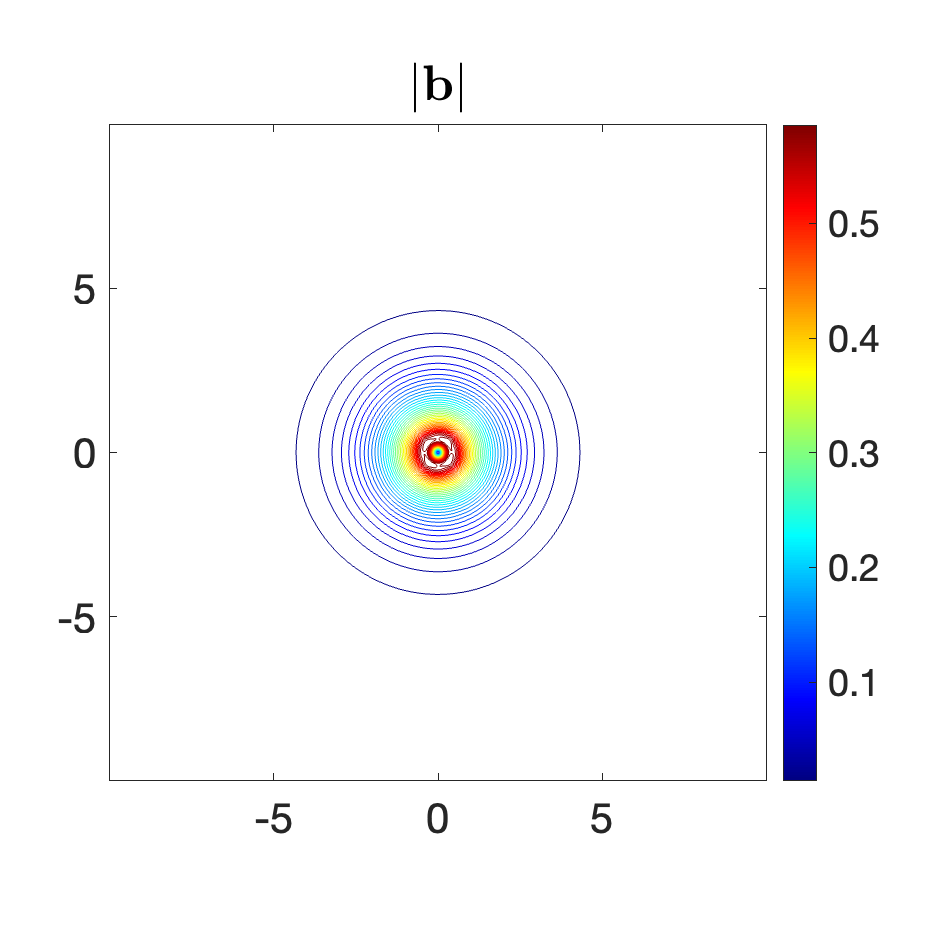}}
\caption{\sf Example 7: Contour plots of $|\bm u|$ and $|\bm b|$ at $t=8$ with 40 equally spaced contours each.\label{fig422}}
\end{figure}
\begin{figure}[ht!]
\centerline{\includegraphics[trim=0.4cm 0.8cm 1.1cm 0.2cm, clip, width=5.2cm]{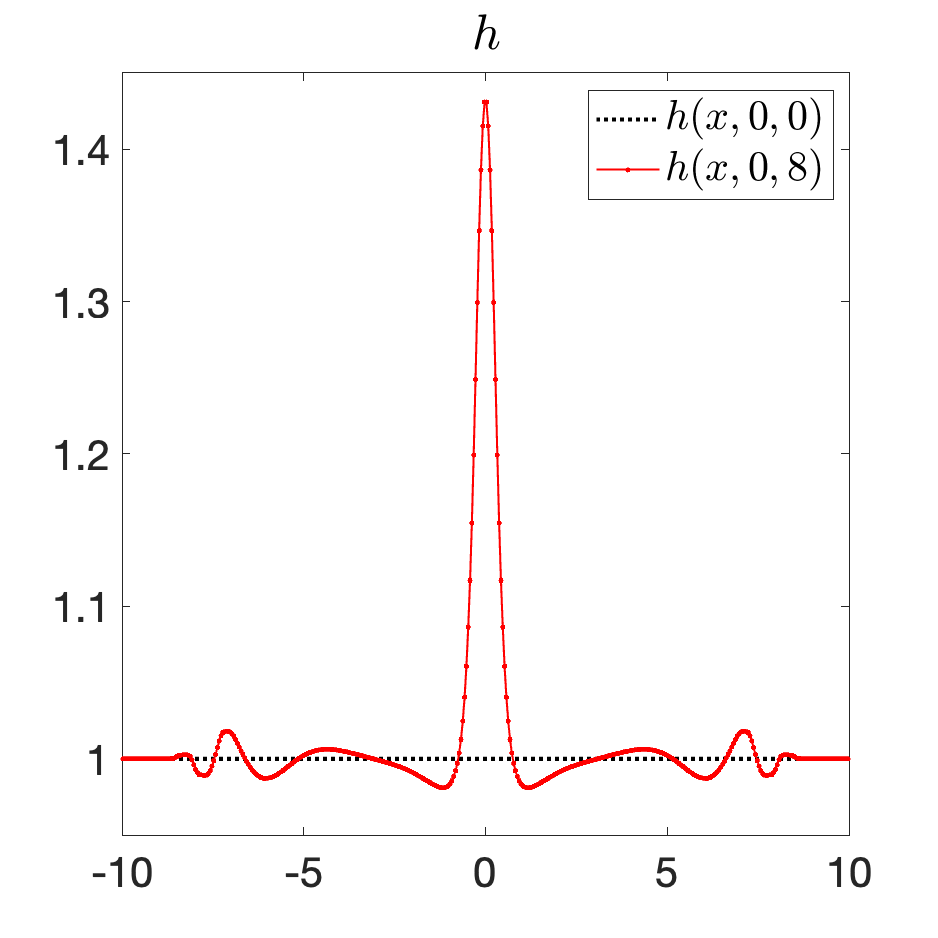}\hspace*{0.5cm}
	    \includegraphics[trim=0.4cm 0.8cm 1.1cm 0.2cm, clip, width=5.2cm]{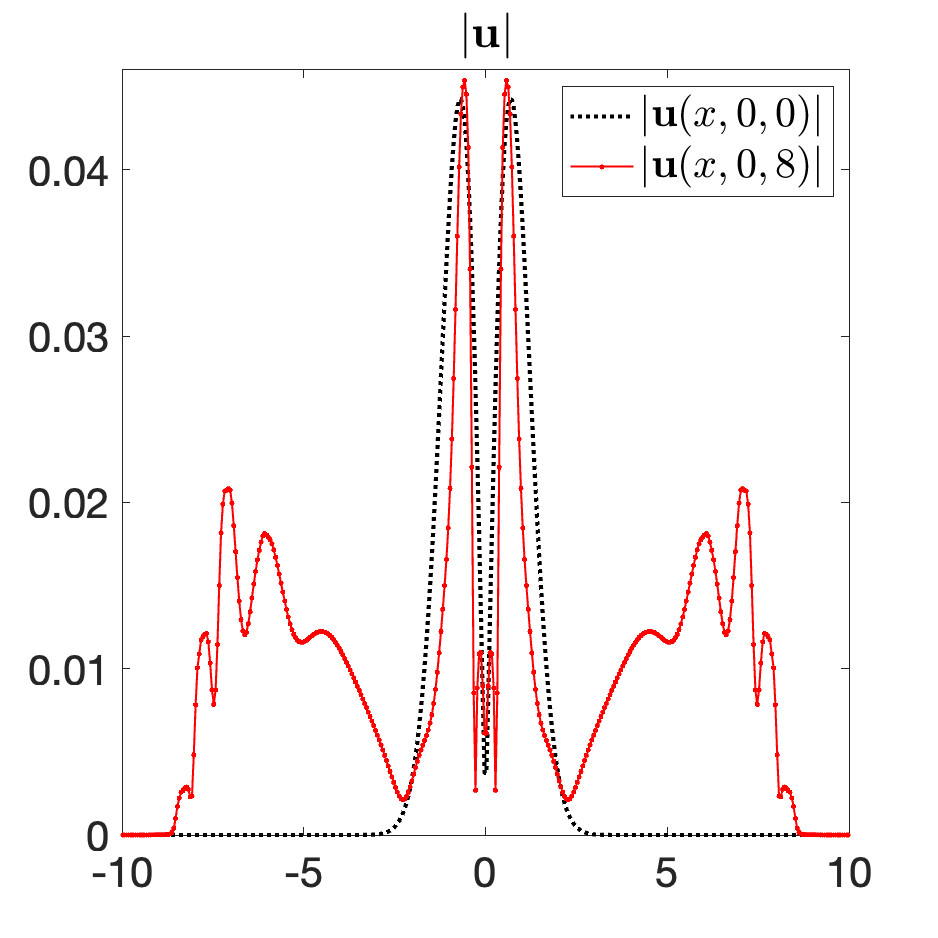}\hspace*{0.5cm}
            \includegraphics[trim=0.4cm 0.8cm 1.1cm 0.2cm, clip, width=5.2cm]{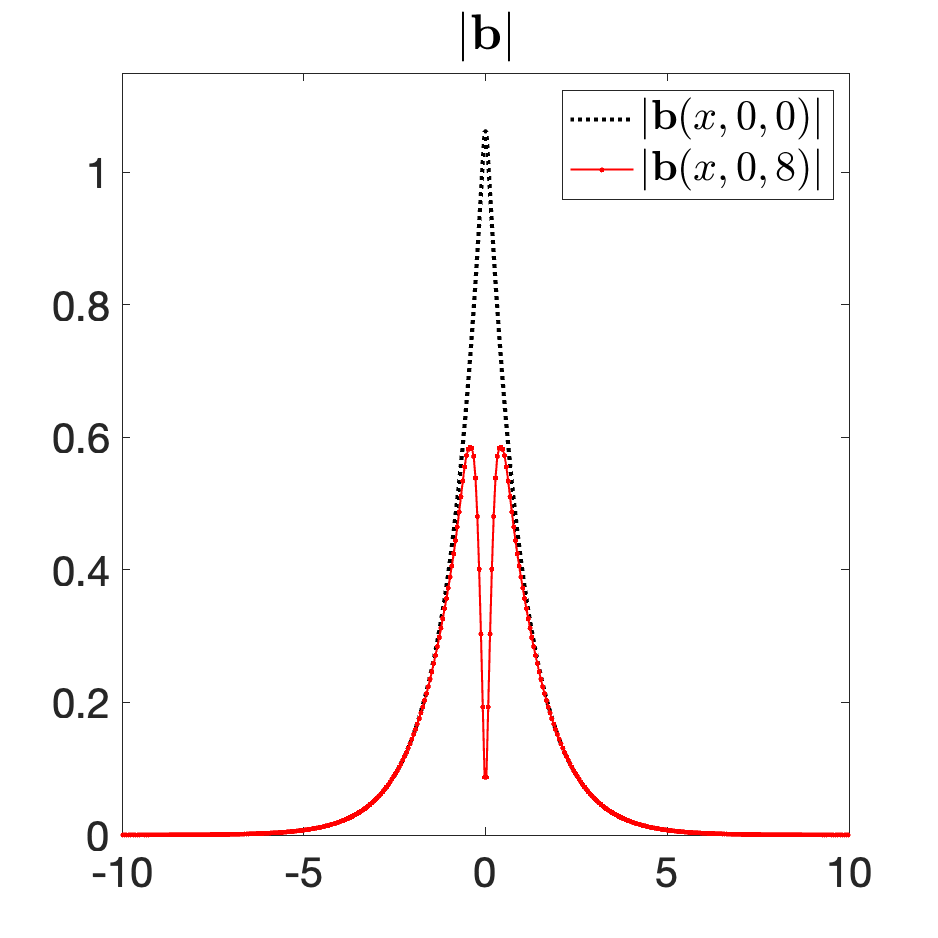}}
\caption{\sf Example 7: 1-D slices of $h$, $|\bm u|$, and $|\bm b|$ along $y=0$ at $t=0$ and 8.\label{fig418f}}
\end{figure}

\subsubsection*{Example 8---Geostrophic Adjustment with Constant Magnetic Field}
In the final example, we study how the standard process of (cyclo-) geostrophic adjustment (see \cite{zeitlin2018geophysical} and references
therein) is influenced by a constant magnetic field, which we chose to be oriented in the $x$-direction. We consider the following initial
conditions:
\begin{equation*}
h(x,y,0)=1+e^{-(x^2+y^2)},\quad(ha)(x,y,0)\equiv1,\quad u(x,y,0)=v(x,y,0)=b(x,y,0)\equiv0,
\end{equation*}
with the constant Coriolis parameter $f(y)\equiv1$ and flat bottom topography $Z(y)\equiv0$ on the computational domain
$[-10,10]\times[-10,10]$ subject to the outflow boundary conditions.

We use the WB scheme to compute the solution until the time $t=8$ on a $200\times200$ uniform mesh. We show the time snapshots of the fluid
depth $h$ and vorticity $\zeta$ at $t=2$, 4, 6, and 8 in Figure \ref{fig417}, in which we see (i) an elongated in the $x$-direction,
consistently with the imposed magnetic field, wave-packet of magneto-inertia-gravity waves; and (ii) two Alfv\'en wave packets that arise
according to the scenario of 1-D adjustment illustrated in Example 3. However, this scenario is modified by 2-D effects, producing vortices
linked to the wave packets and, consequently, traveling outward along the $x$-axis. It is worth noting that, as follows from
the comparison of the height and vorticity fields at the later stages (third and fourth columns of Figure \ref{fig417}), these vortices are
in approximate geostrophic equilibrium with negative vorticity corresponding to greater $h$. In order to confirm the origin of the wave
packets, we present 1-D slices of $v(x,0,t)$ and $b(x,0,t)$ at $t=1$, 2, 3, and 4 in Figure \ref{fig418}, where the characteristic
Alfv\'en-wave signature in these fields can be clearly seen.
\begin{figure}[ht!]
\centerline{\includegraphics[trim=1.0cm 1.7cm 0.6cm 1.1cm, clip, width=4.4cm]{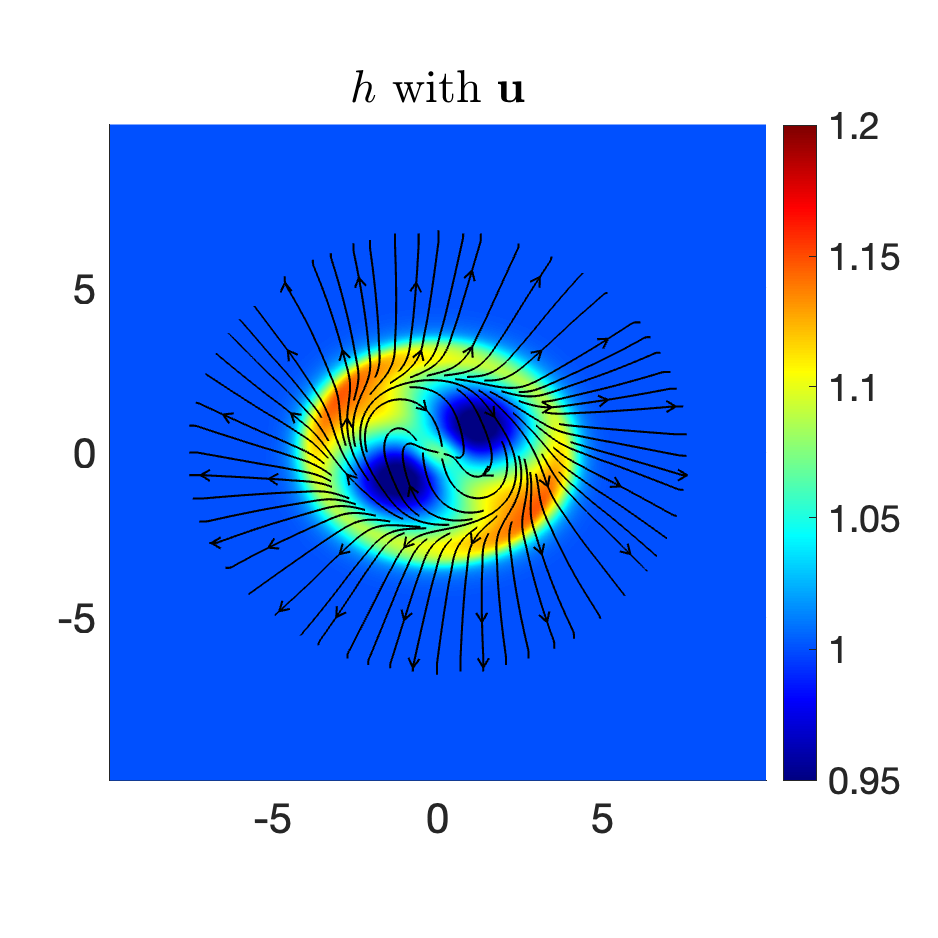}
            \includegraphics[trim=1.0cm 1.7cm 0.6cm 1.1cm, clip, width=4.4cm]{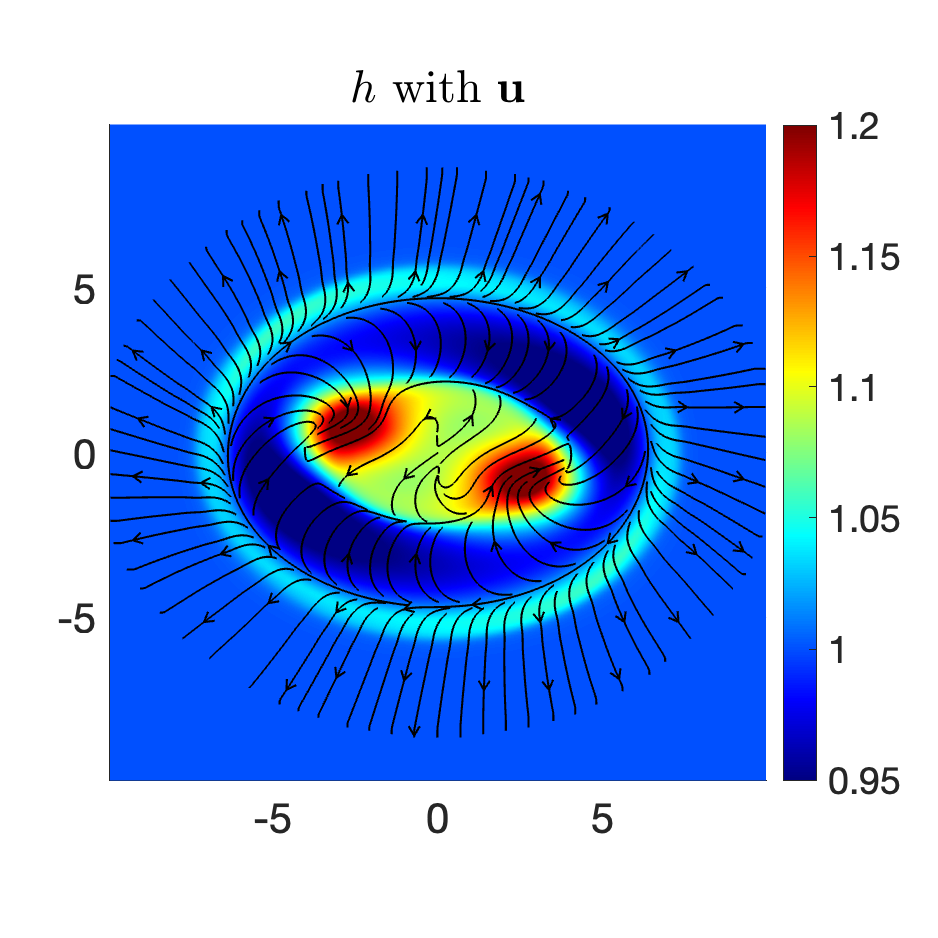}
            \includegraphics[trim=1.0cm 1.7cm 0.6cm 1.1cm, clip, width=4.4cm]{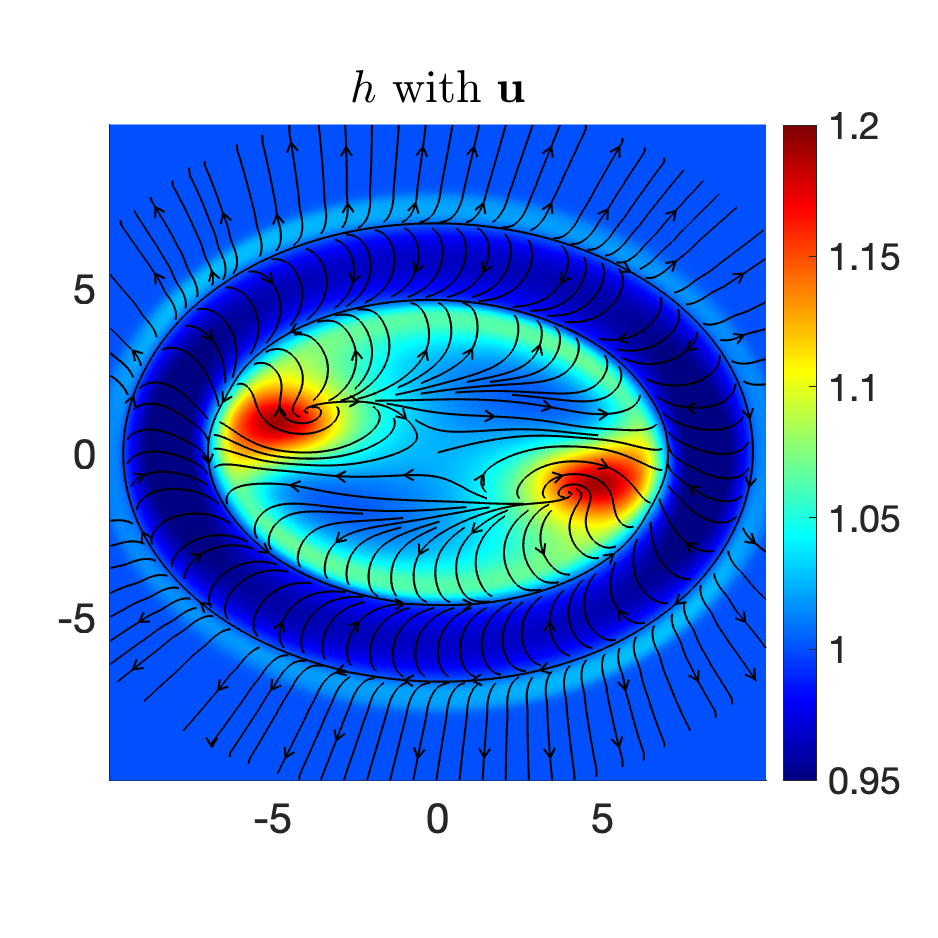}
            \includegraphics[trim=1.0cm 1.7cm 0.6cm 1.1cm, clip, width=4.4cm]{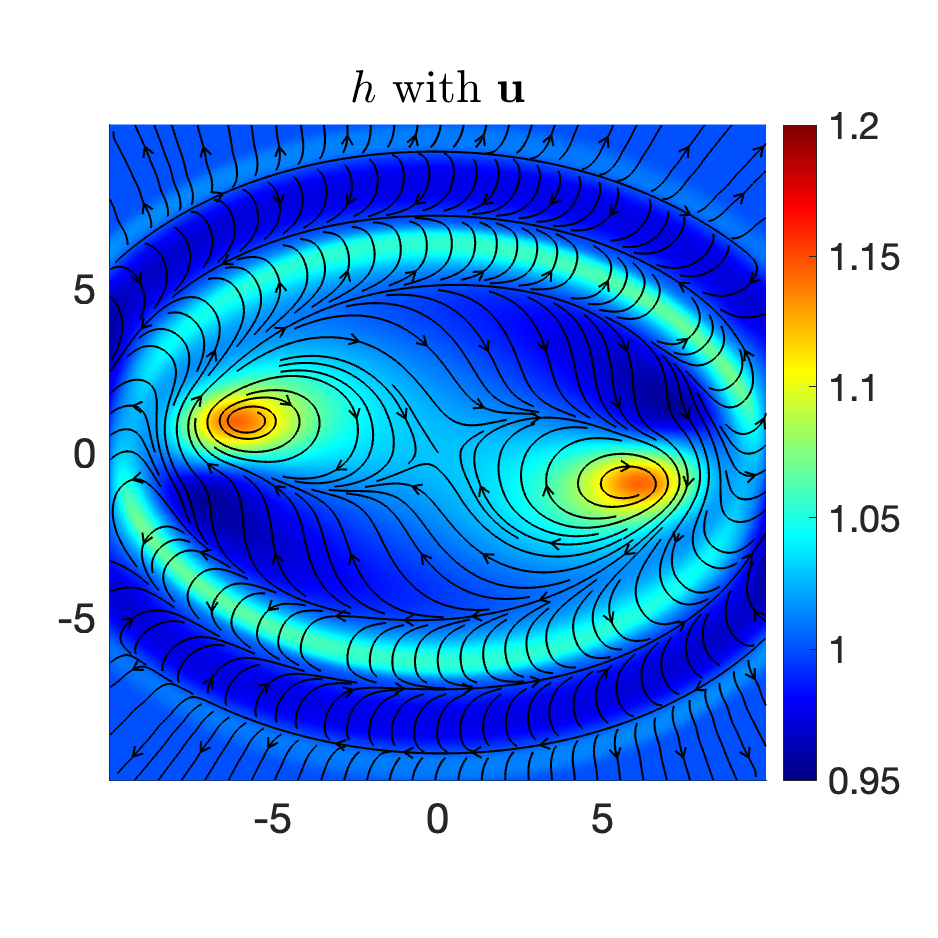}}
\vskip7pt
\centerline{\includegraphics[trim=1.0cm 1.7cm 0.6cm 1.1cm, clip, width=4.4cm]{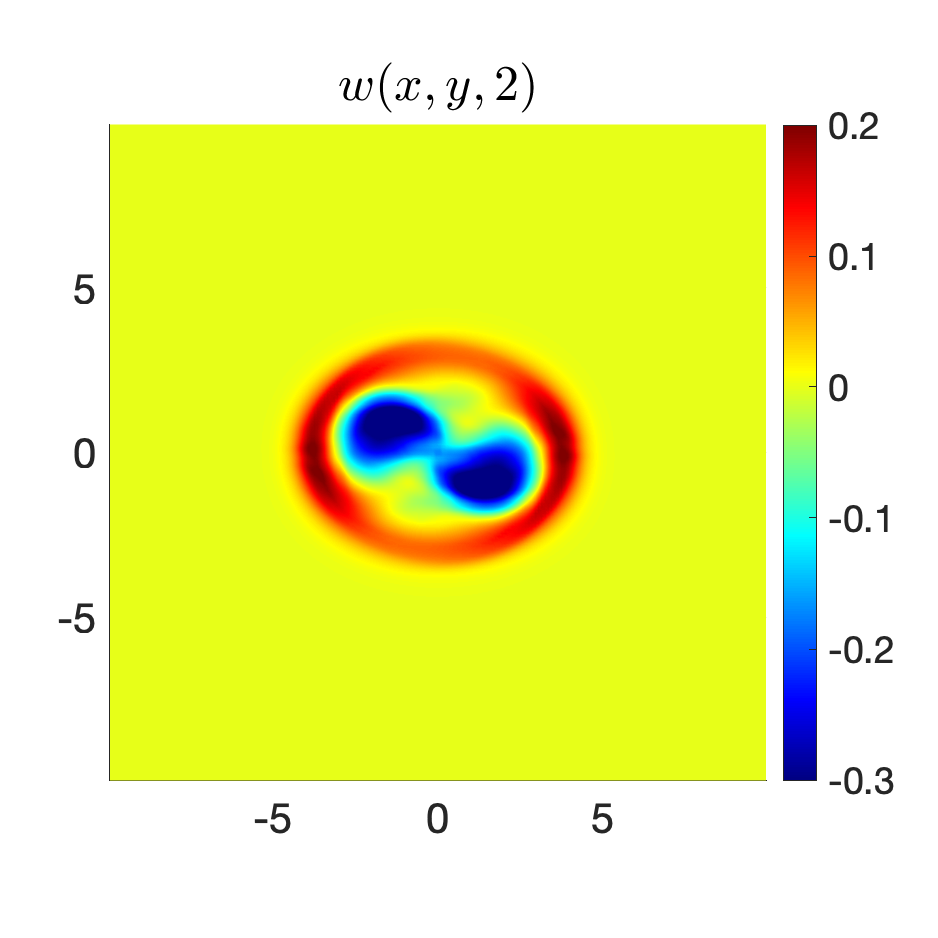}
            \includegraphics[trim=1.0cm 1.7cm 0.6cm 1.1cm, clip, width=4.4cm]{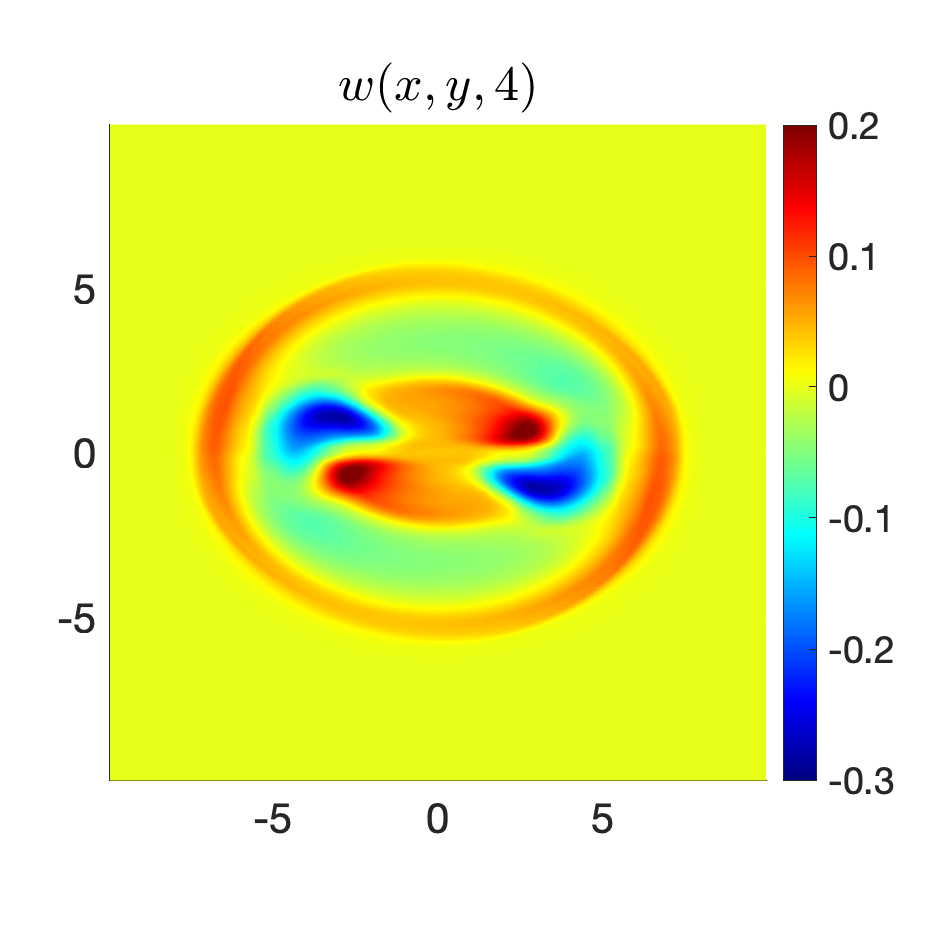}
            \includegraphics[trim=1.0cm 1.7cm 0.6cm 1.1cm, clip, width=4.4cm]{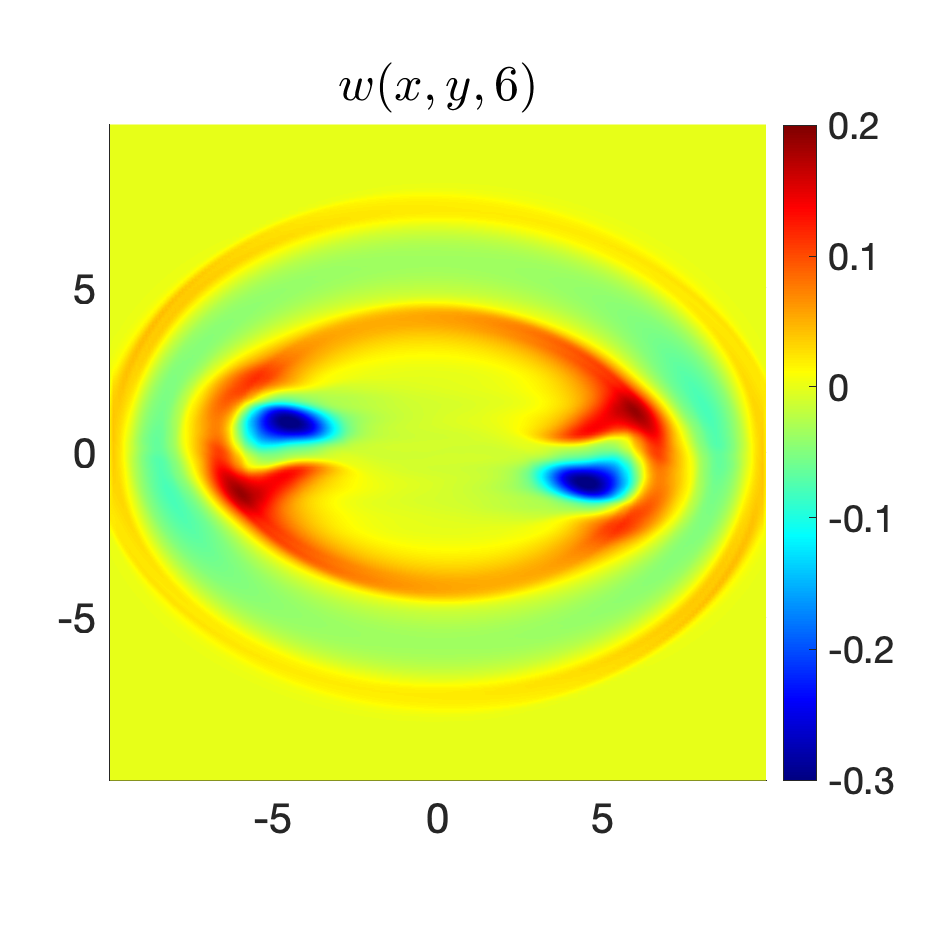}
            \includegraphics[trim=1.0cm 1.7cm 0.6cm 1.1cm, clip, width=4.4cm]{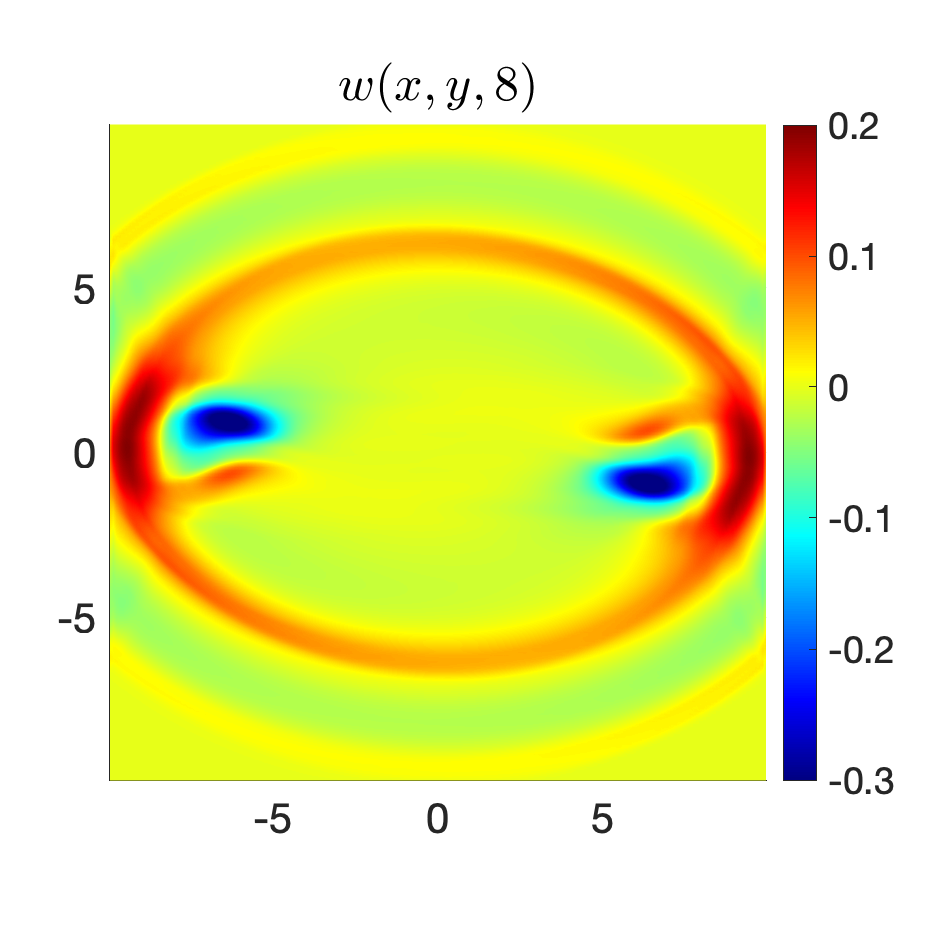}}
\caption{\sf Example 8: Time snapshots of the fluid depth $h$ with velocity streamlines (top row) and vorticity $\zeta$ (bottom row).
\label{fig417}}
\end{figure}
\begin{figure}[ht!]
\centerline{\includegraphics[trim=0.6cm 1.7cm 2.2cm 0.5cm, clip, width=4.3cm]{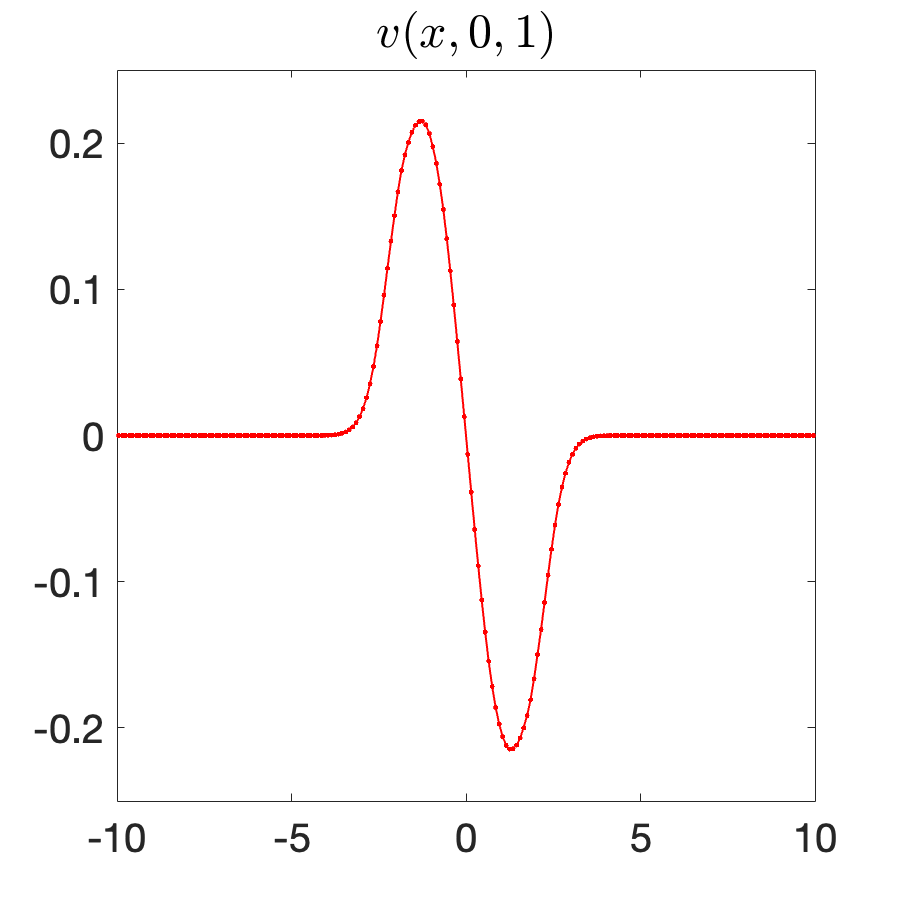}
            \includegraphics[trim=0.6cm 1.7cm 2.2cm 0.5cm, clip, width=4.3cm]{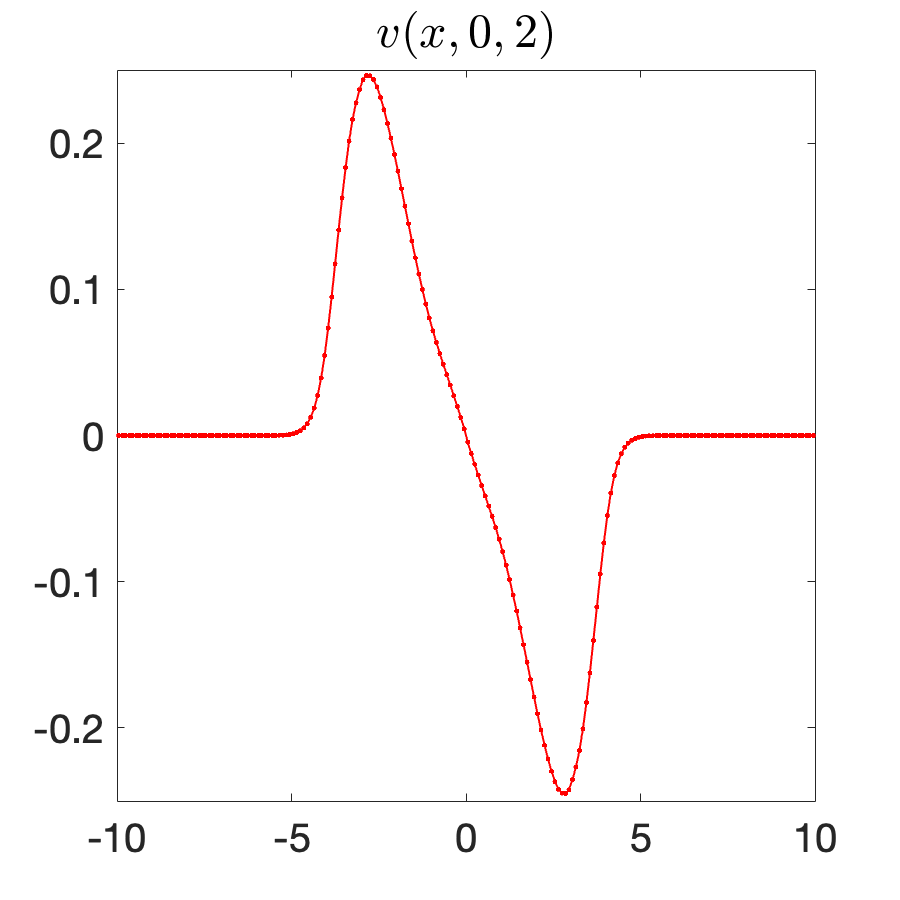}
            \includegraphics[trim=0.6cm 1.7cm 2.2cm 0.5cm, clip, width=4.3cm]{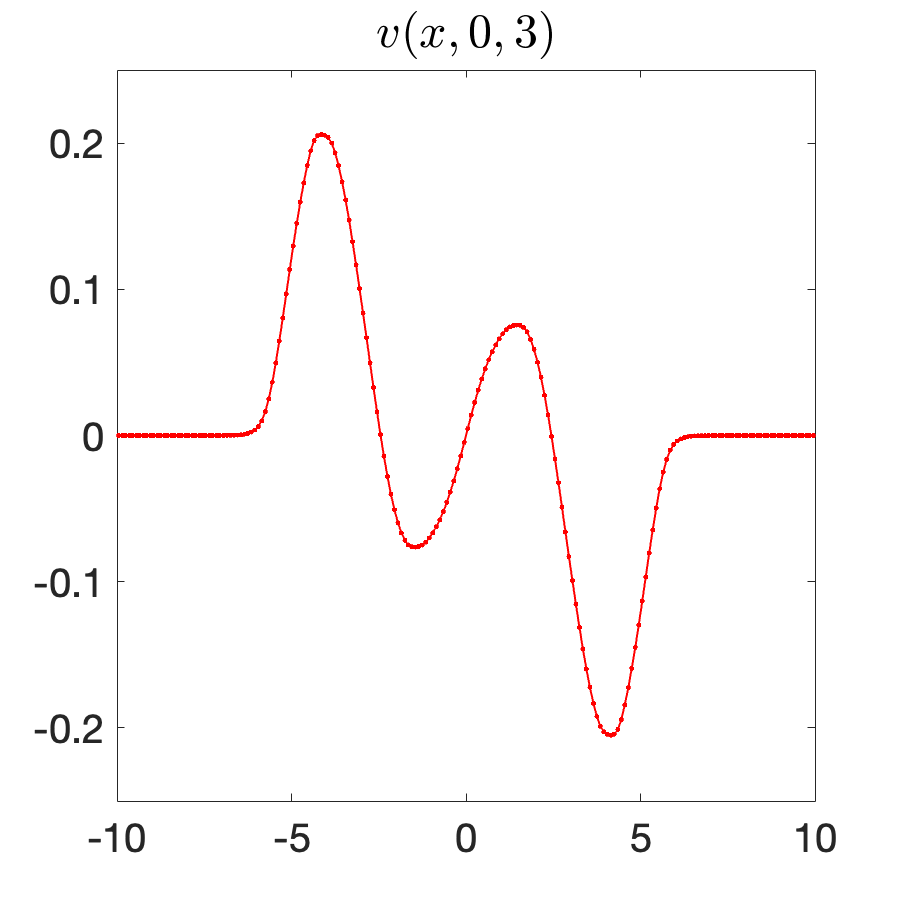}
            \includegraphics[trim=0.6cm 1.7cm 2.2cm 0.5cm, clip, width=4.3cm]{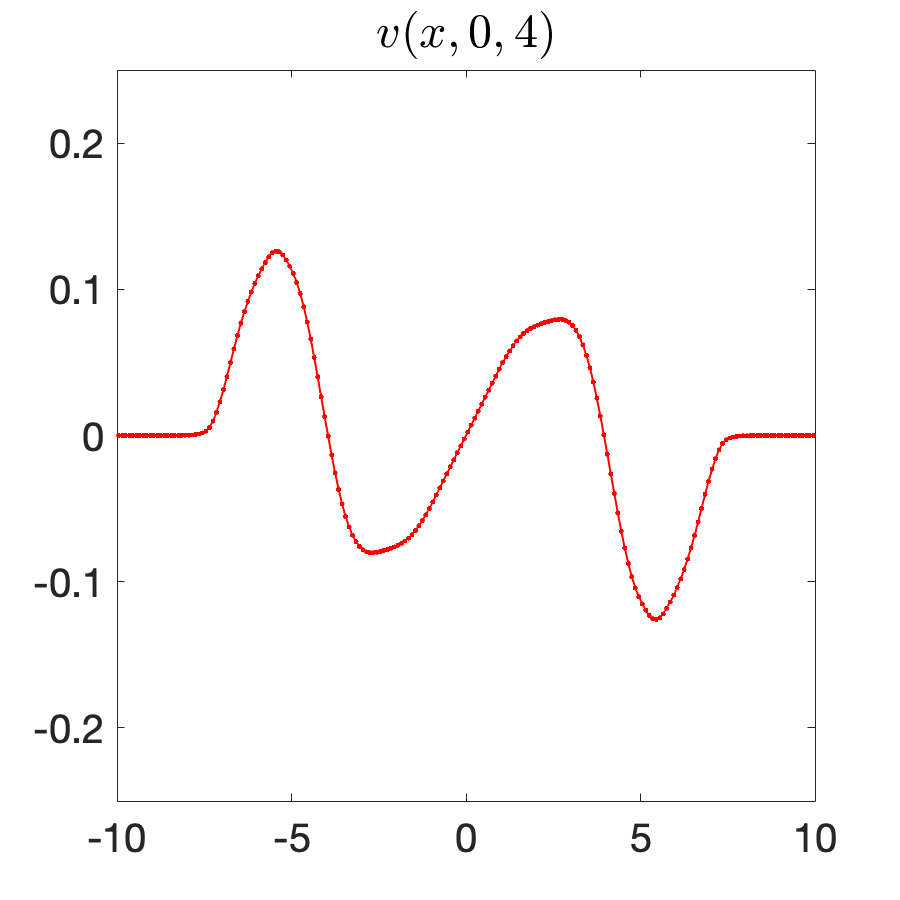}}
\vskip7pt
\centerline{\includegraphics[trim=0.6cm 1.7cm 2.2cm 0.5cm, clip, width=4.3cm]{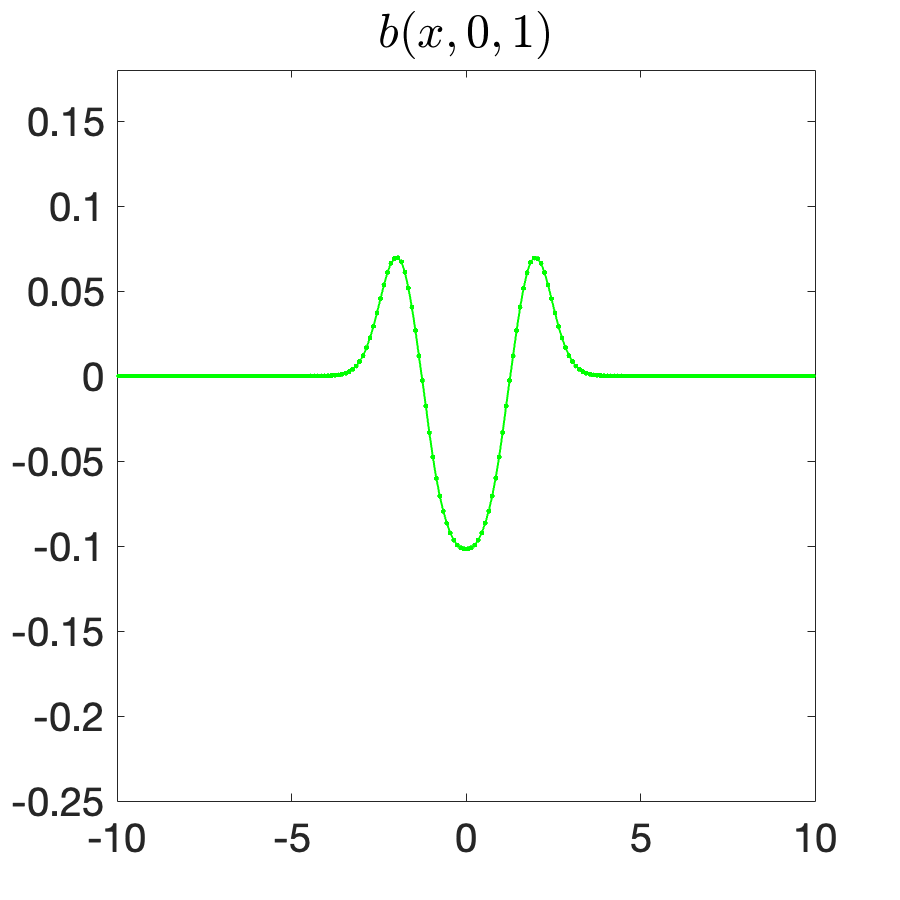}
            \includegraphics[trim=0.6cm 1.7cm 2.2cm 0.5cm, clip, width=4.3cm]{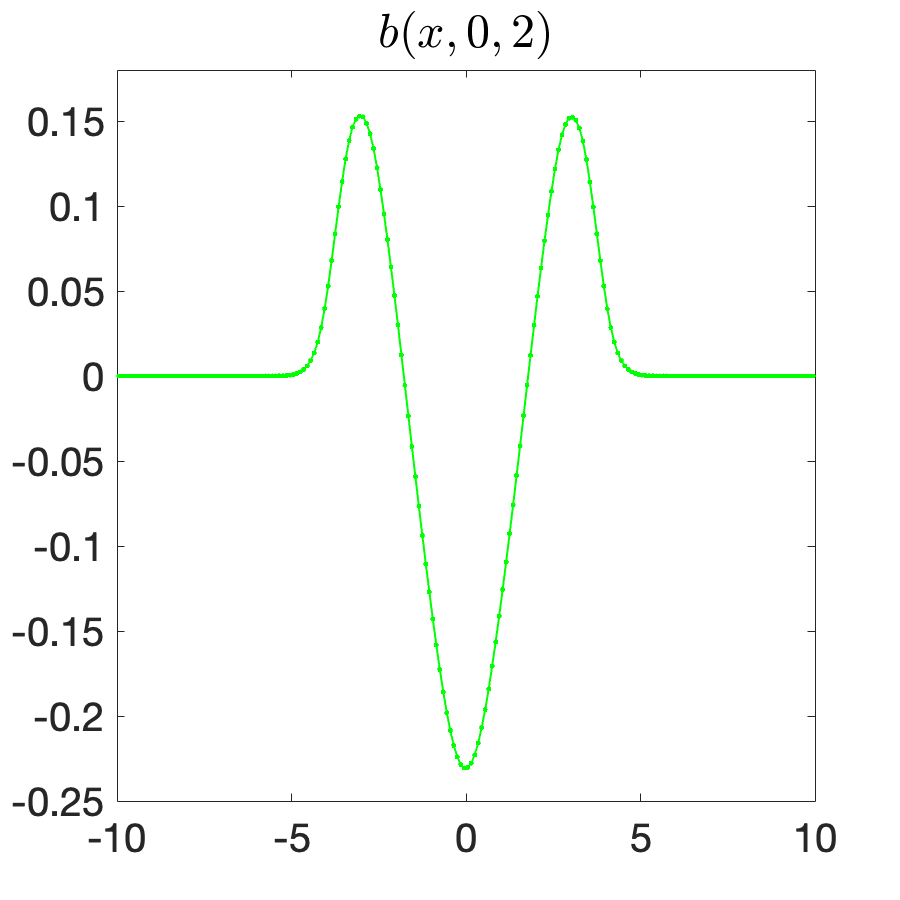}
            \includegraphics[trim=0.6cm 1.7cm 2.2cm 0.5cm, clip, width=4.3cm]{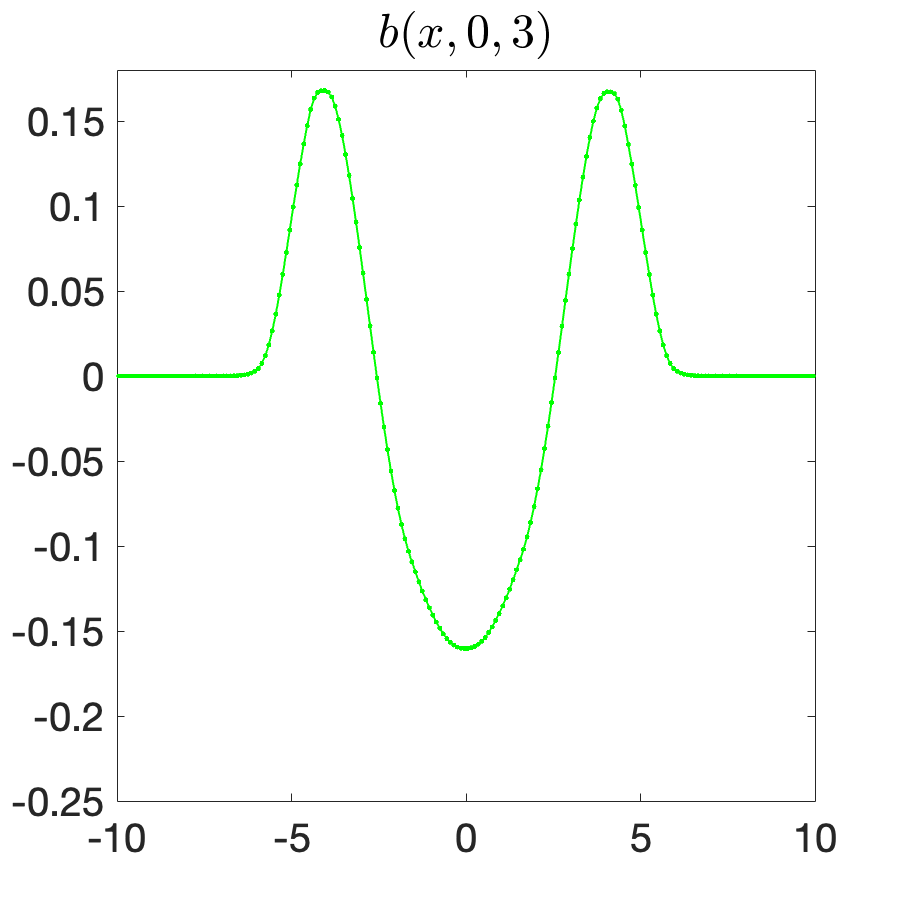}
            \includegraphics[trim=0.6cm 1.7cm 2.2cm 0.5cm, clip, width=4.3cm]{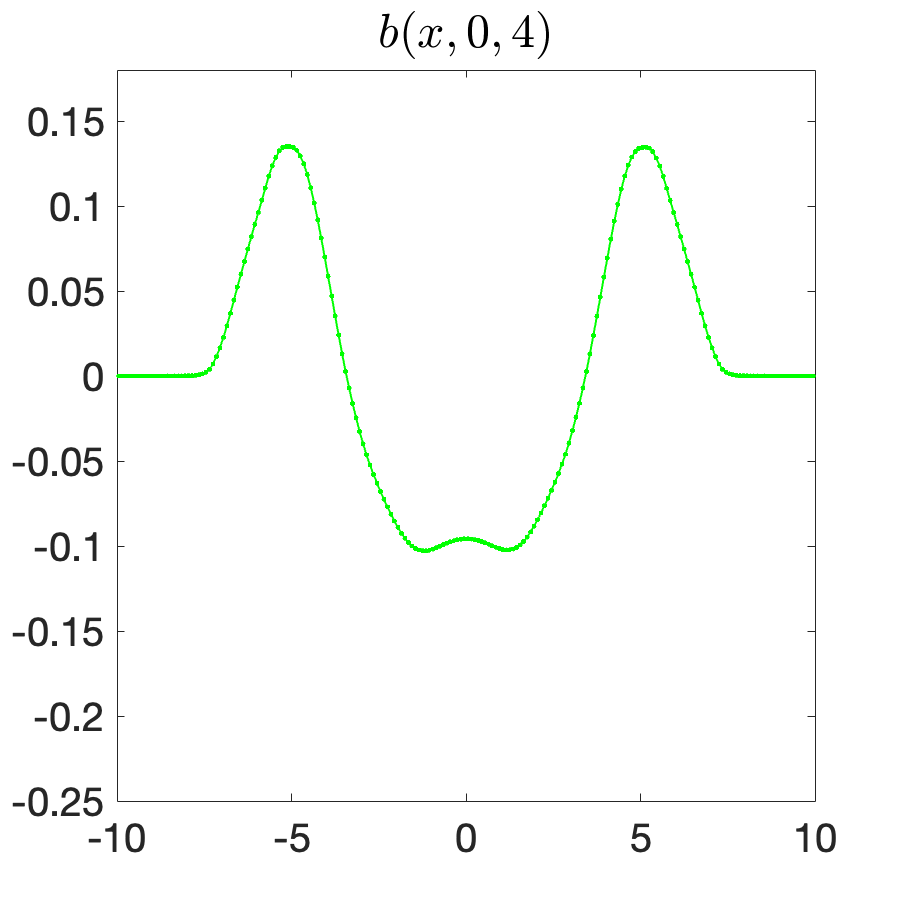}}
\caption{\sf Example 8: Time snapshots of $v(x,0,t)$ (top row) and $b(x,0,t)$ (bottom row) at times $t=1$ (first column), 2 (second column),
3 (third column), and 4 (fourth column).\label{fig418}}
\end{figure}

\section{Conclusion}\label{sec6}
In this paper, we have developed a novel second-order flux globalization based path-conservative central-upwind (PCCU) scheme for rotating
shallow water magnetohydrodynamics equations. Our primary objectives in designing this scheme were twofold: firstly, to maintain the
divergence-free constraint of the magnetic field at the discrete level, and secondly, to uphold the well-balanced (WB) property needed to
preserve certain physically steady states of the underlying system precisely.

In order to enforce the local divergence-free constraint of the magnetic field, we have considered a Godunov-Powell modified version of the
studied system, introduced additional equations through spatial differentiation of the magnetic field equations, and adjusted the
reconstruction procedures for magnetic field variables. In order to guarantee the WB property, we have employed a flux globalization
technique within the PCCU scheme, enabling the method to maintain both still- and moving-water equilibria.

We have conducted a series of numerical experiments that illustrate the performance of the proposed method. The obtained numerical results
demonstrate the scheme's robustness and showcase its ability to provide high-resolution solutions without the development of spurious
oscillations. The presented numerical examples indicate that the scheme accurately captures the details of the fundamental process of
magneto-(cyclo-)geostrophic adjustment, resolving equally well the slow motions, shocks, vortices, long Alfv\'en waves, and the fast wave
motions of both kinds---Alfv\'en and magneto-inertia-gravity---which are present in the system. Thus well-tested, the proposed scheme is
ready for in-depth high-resolution numerical studies of the fundamental dynamical processes in the MRSW model, which are important for
astro- and geophysical applications. Among them, the influence of the magnetic field on vortex dynamics, which can be highly nontrivial, as
shown in recent studies \cite{Lahaye2022Coherent,Magill2019vortexadjustment} (both were not using dedicated MHD numerical schemes),
vortex-Alfv\'en wave interactions, MHD turbulence, which can be studied in the MRSW model along the lines of similar investigations
conducted in the RSW model (see \cite{Lahaye2012turbulence}), and so on.

\begin{acknowledgment}
The work of A. Chertock was partially supported by NSF grant DMS-2208438. The work of A. Kurganov was supported in part by NSFC grant
12171226 and by the fund of the Guangdong Provincial Key Laboratory of Computational Science and Material Design (No. 2019B030301001). The
work of M. Redle was supported in part by NSF grant DMS-2208438 and by the fund of the DFG Research Unit FOR5409 (grant No. 463312734).
\end{acknowledgment}

\appendix

\section{Generalized Minmod Piecewise Linear Reconstruction}\label{appxA}
In this appendix, we provide a brief description of second-order generalized minmod piecewise linear reconstructions (see, e.g.,
\cite{Lie2003artificial,Nessyahu1990Non,Sweby1984High}) in both the 1-D and 2-D cases.

In the 1-D case, we consider a function $\psi(y)$, whose values $\psi_k$ (either the cell averages or point values) at $y=y_k$ are
available, and approximate the slopes $(\psi_y)_k$ using a generalized minmod limiter to obtain
\begin{equation}
(\psi_y)_k={\rm minmod}\left(\Theta\,\frac{\psi_{k+1}-\psi_k}{\dy},\,\frac{\psi_{k+1}-\psi_{k-1}}{2\dy},\,
\Theta\,\frac{\psi_k-\psi_{k-1}}{\dy}\right),\quad\Theta\in[1,2].
\label{2.26}
\end{equation}
Here, the minmod function is defined by
\begin{equation*}
\mbox{minmod}(c_1,c_2,\ldots)=\left\{\begin{aligned}
&\min(c_1,c_2,\ldots)&&\mbox{if}~c_i>0,~\forall i,\\
&\max(c_1,c_2,\ldots)&&\mbox{if}~c_i<0,~\forall i,\\
&0&&\mbox{otherwise},
\end{aligned}
\right.
\end{equation*}
and the parameter $\Theta$ is to be chosen to adjust the amount of numerical dissipation present in the numerical scheme with larger values
of $\Theta$ leading to sharper but, in general, more oscillatory reconstructions.

We use the slopes computed in \eref{2.26} to obtain the following second-order non-oscillatory piecewise linear reconstruction of $\psi$:
\begin{equation}
\widetilde\psi(y)=\psi_k+(\psi_y)_k(y-y_k),\quad y\in C_k.
\label{2.24}
\end{equation}
This reconstruction is generically discontinuous at the cell interfaces $y=y_\kph$ and hence the one-sided point values of $\psi$ at those
points are
\begin{equation}
\psi^-_\kph:=\,\psi_k+\frac{\dy}{2}(\psi_y)_k\quad\mbox{and}\quad\psi^+_\kph:=\,\psi_{k+1}-\frac{\dy}{2}(\psi_y)_{k+1}.
\label{2.25}
\end{equation}

When a 2-D function $\psi(x,y)$ is considered, we denote by $\psi_{j,k}$ its discrete values and use the generalized minmod limiter to
approximate the $x$- and $y$-slopes:
\begin{equation}
\begin{aligned}
(\psi_x)_{j,k}={\rm minmod}\left(\Theta\,\frac{\psi_{j,k}-\psi_{j-1,k}}{\dx},\,\frac{\psi_{j+1,k}-\psi_{j-1,k}}{2\dx},\,
\Theta\,\frac{\psi_{j+1,k}-\psi_{j,k}}{\dx}\right),\\
(\psi_y)_{j,k}={\rm minmod}\left(\Theta\,\frac{\psi_{j,k}-\psi_{j,k-1}}{\dy},\,\frac{\psi_{j,k+1}-\psi_{j,k-1}}{2\dy},\,
\Theta\,\frac{\psi_{j,k+1}-\psi_{j,k}}{\dy}\right),
\end{aligned}
\quad\Theta\in[1,2].
\label{3.24}
\end{equation}
The resulting piecewise linear interpolant then reads as
\begin{equation*}
\widetilde\psi(x,y)\,=\psi_{j,k}+{(\psi_x)}_{j,k}(x-x_j)+{(\psi_y)}_{j,k}(y-y_k),\quad (x,y)\in C_{j,k},
\end{equation*}
and the corresponding point values of $\psi$ inside each cell $C_{j,k}$ are then given by
$$
\begin{aligned}
&\psi_{j,k}^{\rm E}:=\psi_{j,k}+\frac{\dx}{2}(\psi_x)_{j,k},\quad\psi_{j,k}^{\rm W}:=\psi_{j,k}-\frac{\dx}{2}(\psi_x)_{j,k},\\
&\psi_{j,k}^{\rm N}:=\psi_{j,k}+\frac{\dy}{2}(\psi_y)_{j,k},\quad\psi_{j,k}^{\rm S}:=\psi_{j,k}-\frac{\dy}{2}(\psi_y)_{j,k}.
\end{aligned}
$$

\section{1-D Fifth-Order WENO-Z Interpolant}\label{appB}
In this appendix, we briefly describe the fifth-order WENO-Z interpolant introduced in \cite{BCCD,CCD,DB}.

We consider a function $\psi(y)$, whose point values $\psi_k$ at $y=y_k$ are available and explain how to calculate an interpolated
left-sided values of $\psi$ at $y=y_\kph$, denoted by $\psi^-_\kph$. The right-sided value $\psi^+_\kph$ can then be obtained in a
mirror-symmetric way.

We construct the three parabolic interpolants ${\cal P}_{k,0}(y)$, ${\cal P}_{k,1}(y)$, and ${\cal P}_{k,2}(y)$ on the stencils
$[y_{k-2},y_{k-1},y_k]$, $[y_{k-1},y_k,y_{k+1}]$, and $[y_k,y_{k+1},y_{k+2}]$, respectively, and compute $\psi^-_\kph$ as their weighted
average:
\begin{equation*}
\psi^-_\kph=\sum_{\ell=0}^2\omega_{k,\ell}{\cal P}_{k,\ell}(y_\kph),
\end{equation*}
where
\begin{equation*}
\begin{aligned}
&{{\cal P}}_{k,0}(y_\kph)=\frac{3}{8}\,\psi_{k-2}-\frac{5}{4}\,\psi_{k-1}+\frac{15}{8}\,\psi_k,\\
&{{\cal P}}_{k,1}(y_\kph)=-\frac{1}{8}\,\psi_{k-1}+\frac{3}{4}\,\psi_k+\frac{3}{8}\,\psi_{k+1},\\
&{{\cal P}}_{k,2}(y_\kph)=\frac{3}{8}\,\psi_k+\frac{3}{4}\,\psi_{k+1}-\frac{1}{8}\,\psi_{k+2},
\end{aligned}
\end{equation*}
and the weights $\omega_{k,\ell}$ are computed by
\begin{equation*}
\omega_{k,\ell}=\frac{\alpha_{k,\ell}}{\alpha_{k,0}+\alpha_{k,1}+\alpha_{k,2}},\quad
\alpha_{k,\ell}=d_\ell\left[1+\bigg(\frac{\tau_{k,5}}{\beta_{k,\ell}+\varepsilon}\bigg)^r\,\right],\quad\ell=0,1,2.
\end{equation*}
Here, $d_0=\frac{1}{16}$, $d_1=\frac{5}{8}$, $d_2=\frac{5}{16}$, the smoothness indicators $\beta_{k,\ell}$ for the corresponding parabolic
interpolants ${\cal P}_{k,\ell}(y)$ are given by
\begin{equation*}
\begin{aligned}
&\beta_{k,0}=\frac{13}{12}\big(\psi_{k-2}-2\psi_{k-1}+\psi_k\big)^2+\frac{1}{4}\big(\psi_{k-2}-4\psi_{k-1}+3\psi_k\big)^2,\\
&\beta_{k,1}=\frac{13}{12}\big(\psi_{k-1}-2\psi_k+\psi_{k+1}\big)^2+\frac{1}{4}\big(\psi_{k-1}-\psi_{k+1}\big)^2,\\
&\beta_{k,2}=\frac{13}{12}\big(\psi_k-2\psi_{k+1}+\psi_{k+2}\big)^2+\frac{1}{4}\big(3\psi_k-4\psi_{k+1}+\psi_{k+2}\big)^2,
\end{aligned}
\end{equation*}
$\tau_{k,5}=\big|\beta_{k,2}-\beta_{k,0}\big|$. We have chosen $r=2$ and $\varepsilon=10^{-12}$ in all of the numerical examples.

\bibliographystyle{siam}
\bibliography{biblio}
\end{document}